\documentclass[microtype,breaklinks]{gtpart}
\usepackage{graphicx}
\usepackage{mathtools}
\usepackage{paralist}
\usepackage[OT2,T1]{fontenc}
\usepackage[title]{appendix}
\title{Discrete conformal maps and ideal hyperbolic
  polyhedra}
\author[A I Bobenko]{Alexander I.~Bobenko}
\givenname{Alexander}
\surname{Bobenko}
\address{Technische Universit\"at Berlin\\
  Institut f\"ur Mathematik\\
  Strasse des 17. Juni 136\\
  10623 Berlin, Germany
}
\email{bobenko@math.tu-berlin.de}
\urladdr{http://page.math.tu-berlin.de/~bobenko}

\author[U Pinkall]{Ulrich Pinkall}
\givenname{Ulrich}
\surname{Pinkall}
\email{pinkall@math.tu-berlin.de}
\urladdr{http://page.math.tu-berlin.de/~pinkall}

\author[B A Springborn]{Boris A.~Springborn}
\givenname{Boris}
\surname{Springborn}
\email{boris.springborn@tu-berlin.de}
\urladdr{http://page.math.tu-berlin.de/~springb}

\keyword{discrete conformal geometry}
\keyword{polyhedron}
\keyword{hyperbolic geometry}
\subject{primary}{msc2010}{52C26}
\subject{primary}{msc2010}{52B10}
\subject{secondary}{msc2010}{57M50}

\arxivreference{math.GT/1005.2698}
\arxivpassword{xsvdb}

\volumenumber{}
\issuenumber{}
\publicationyear{}
\papernumber{}
\startpage{}
\endpage{}
\doi{}
\MR{}
\Zbl{}
\received{}
\revised{}
\accepted{}
\published{}
\publishedonline{}
\proposed{}
\seconded{}
\corresponding{}
\editor{}
\version{}

\newcounter{problemprefix}
\theoremstyle{definition}
\newtheorem{condition}{Condition}

\swapnumbers

\theoremstyle{plain}
\newtheorem{theorem}{Theorem}[subsection]
\newtheorem{lemma}[theorem]{Lemma}
\newtheorem{proposition}[theorem]{Proposition}
\newtheorem{corollary}[theorem]{Corollary}

\theoremstyle{definition}
\newtheorem{definition}[theorem]{Definition}
\newtheorem{problem}[theorem]{Problem}
\newtheorem{remark}[theorem]{Remark}
\newtheorem*{example*}{Example}

\theoremstyle{remark}
\newtheorem*{warning*}{Warning}

\numberwithin{equation}{section}

\newcommand{\RP}{\R\mathrm{P}}
\newcommand{\T}{{\mathsf T}}
\newcommand{\cclass}{{\mathcal C}}
\newcommand{\Teich}{{\mathcal T}}
\newcommand{\ccr}{\operatorname{cr}}
\newcommand{\lcr}{\operatorname{lcr}}
\newcommand{\Eint}{E_{\mathit{int}}}
\newcommand{\Ebdy}{E_{\mathit{bdy}}}
\newcommand{\Vint}{V_{\mathit{int}}}
\newcommand{\Vbdy}{V_{\mathit{bdy}}}
\newcommand{\ML}{\mbox{\fontencoding{OT2}\fontfamily{wncyr}\fontseries{m}\fontshape{n}\selectfont L}}
\newcommand{\const}{\textit{const.}}
\newcommand{\grad}{\operatorname{grad}}

\newcommand{\degrees}{^{\circ}}
\newcommand{\im}{\operatorname{Im}}
\newcommand{\arsinh}{\operatorname{arsinh}}
\newcommand{\arccot}{\operatorname{arccot}}
\newcommand{\Ehyp}{E^{\mathit{h}}}
\newcommand{\Vhyp}{V_{\mathit{h}}}
\newcommand{\Vhathyp}{\widehat{V}_{\mathit{h}}}

\begin{document}

\begin{abstract}
  We establish a connection between two previously unrelated topics: a
  particular discrete version of conformal geometry for triangulated
  surfaces, and the geometry of ideal polyhedra in hyperbolic
  three-space.  Two triangulated surfaces are considered discretely
  conformally equivalent if the edge lengths are related by scale
  factors associated with the vertices. This simple definition leads
  to a surprisingly rich theory featuring M\"obius invariance, the
  definition of discrete conformal maps as circumcircle preserving
  piecewise projective maps, and two variational principles. We show
  how literally the same theory can be reinterpreted to address the
  problem of constructing an ideal hyperbolic polyhedron with
  prescribed intrinsic metric. This synthesis enables us to derive a
  companion theory of discrete conformal maps for hyperbolic
  triangulations. It also shows how the definitions of discrete
  conformality considered here are closely related to the established
  definition of discrete conformality in terms of circle packings.

\end{abstract}

\maketitle

\section{Introduction}
\label{sec:intro}

Recall that two Riemannian metrics $g$ and $\tilde g$ on a smooth
manifold~$M$ are called \emph{conformally equivalent} if
\begin{equation}
  \label{eq:tilde_g}
  \tilde g = e^{2u}g
\end{equation}
for a function $u\in C^{\infty}(M)$. In the discrete theory that we
consider here, smooth manifolds are replaced with triangulated
piecewise euclidean manifolds, and the discrete version of a conformal
change of metric is to multiply all edge lengths with scale factors
that are associated with the vertices
(Definition~\ref{def:d_conf_equiv}). Apparently, the idea to model
conformal transformations in a discrete setting by attaching scale
factors to the vertices appeared first in the four-dimensional
Lorentz-geometric context of the Regge
calculus~\cite{roek_quantization_1984}. The Riemann-geometric version
of this notion appeared in Luo's work on ``combinatorial Yamabe
flow''~\cite{luo_combinatorial_2004}. He showed that this flow is the
gradient flow of a locally convex function. Later, an explicit formula
for this function was found ($E_{\T,\Theta,\lambda}$ defined in
equation~\eqref{eq:E_of_lambda}, with $\Theta=0$), and this lead to an
efficient numerical method to compute discrete conformal maps,
suitable for applications in computer
graphics~\cite{springborn_conformal_2008}. (Some basic theory of
conformal equivalence and conformal maps in
Section~\ref{sec:discrete_conformal} and the first variational
principle in Section~\ref{sec:variational_principles} are already
covered or at least touched upon in this earlier paper.) The
variational principles described in
Section~\ref{sec:variational_principles} reduce the discrete conformal
mapping problems described in Section~\ref{sec:mapping_problems} to
problems of convex optimization. Figures~\ref{fig:rect_map}
and~\ref{fig:discrete_riemann} show examples of discrete conformal
maps that were obtained this way.
\begin{figure}
  \centering
  \includegraphics[width=0.495\textwidth]{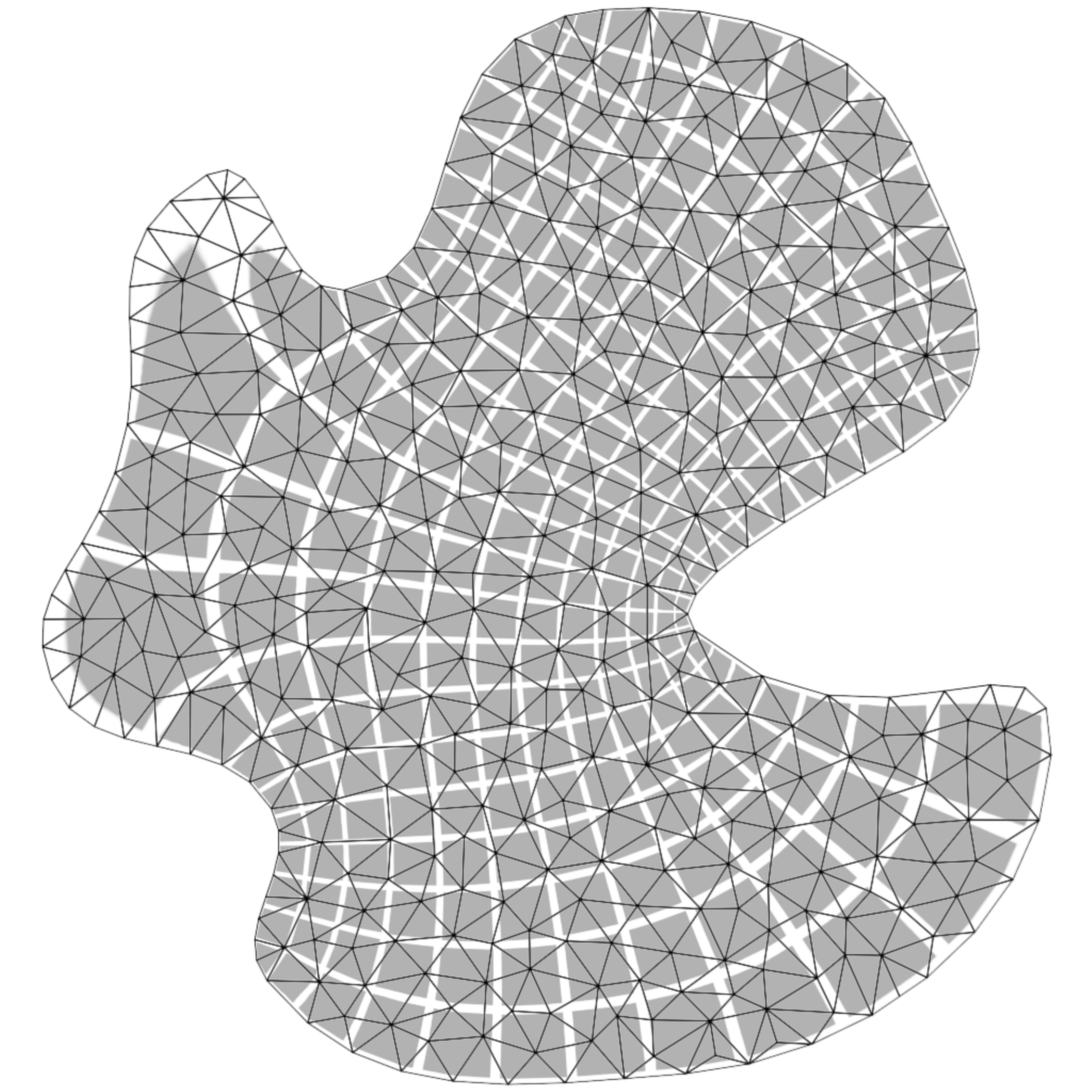}%
  \includegraphics[width=0.495\textwidth]{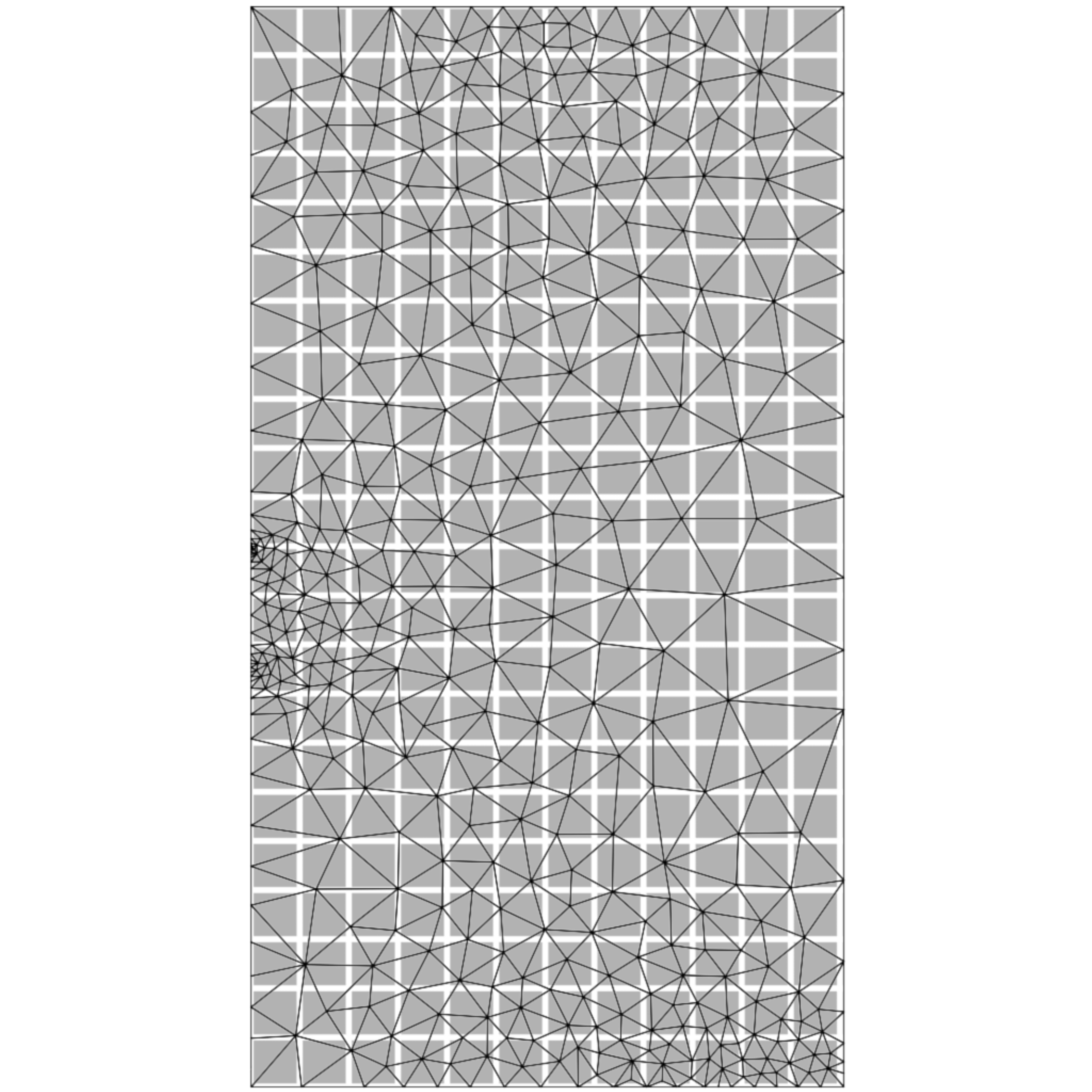}
  \caption{Discrete conformal map to a rectangle.}
  \label{fig:rect_map}
\end{figure}
\begin{figure}
  \centering
  \includegraphics[width=0.45\textwidth]{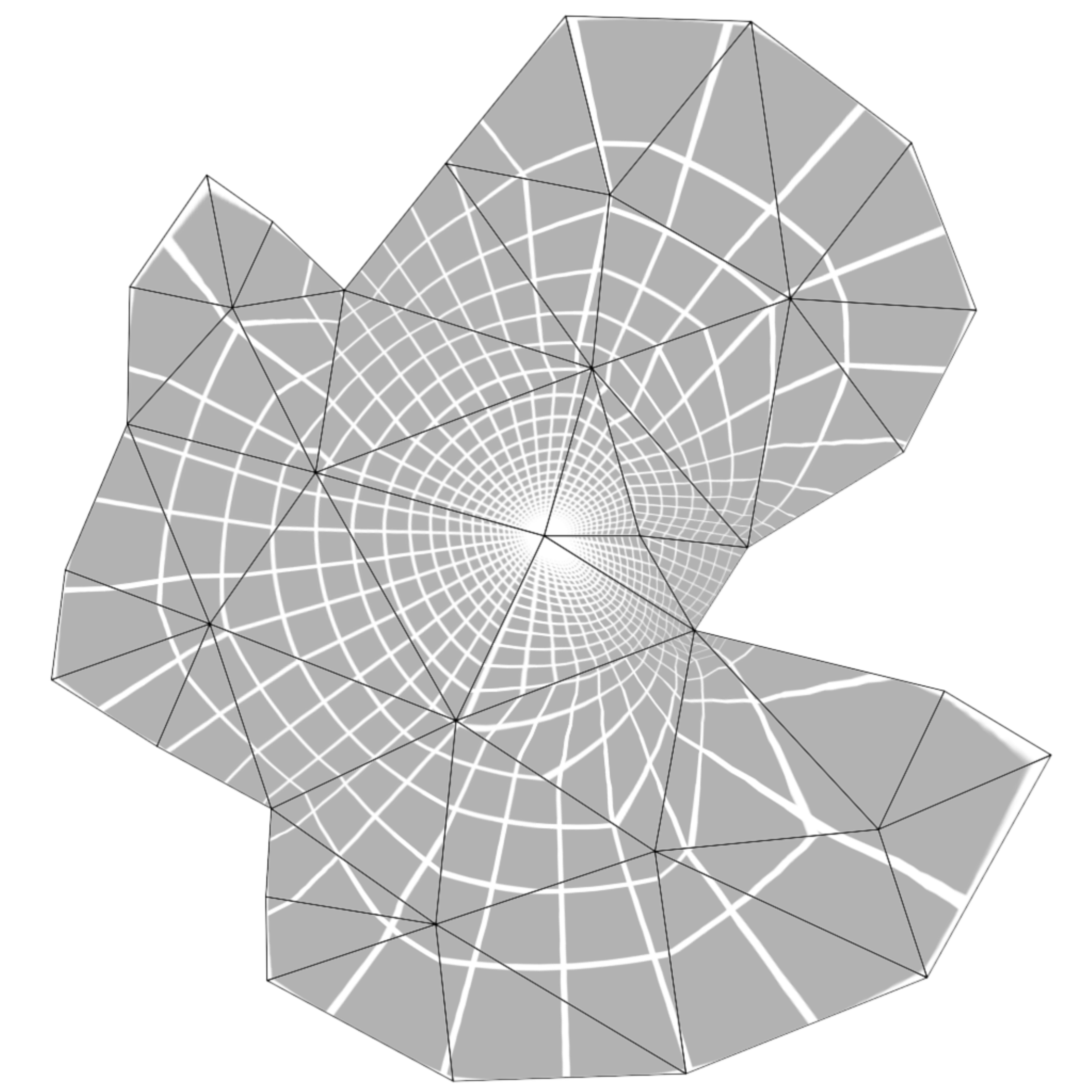}%
  \includegraphics[width=0.45\textwidth]{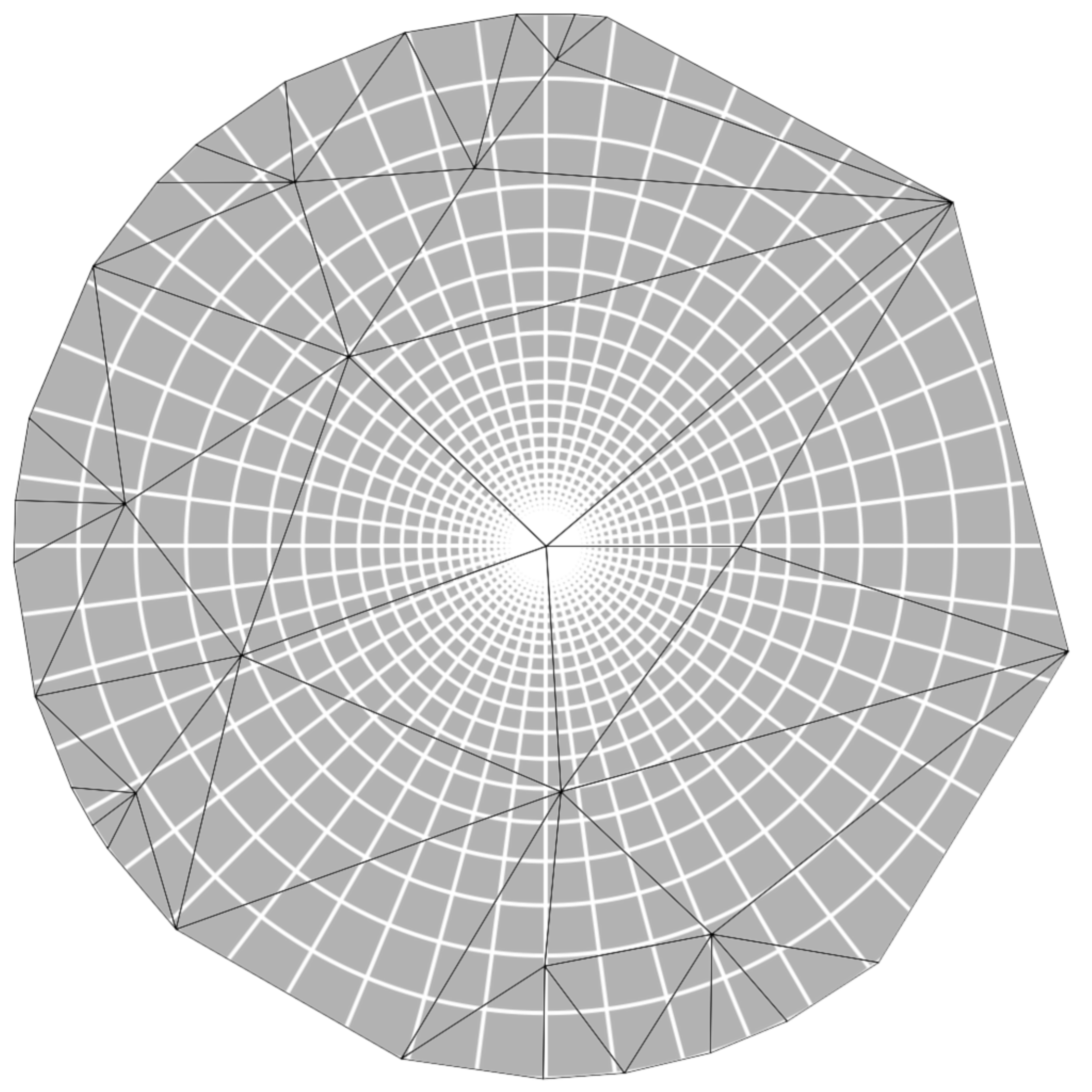}\\
  \includegraphics[width=0.45\textwidth]{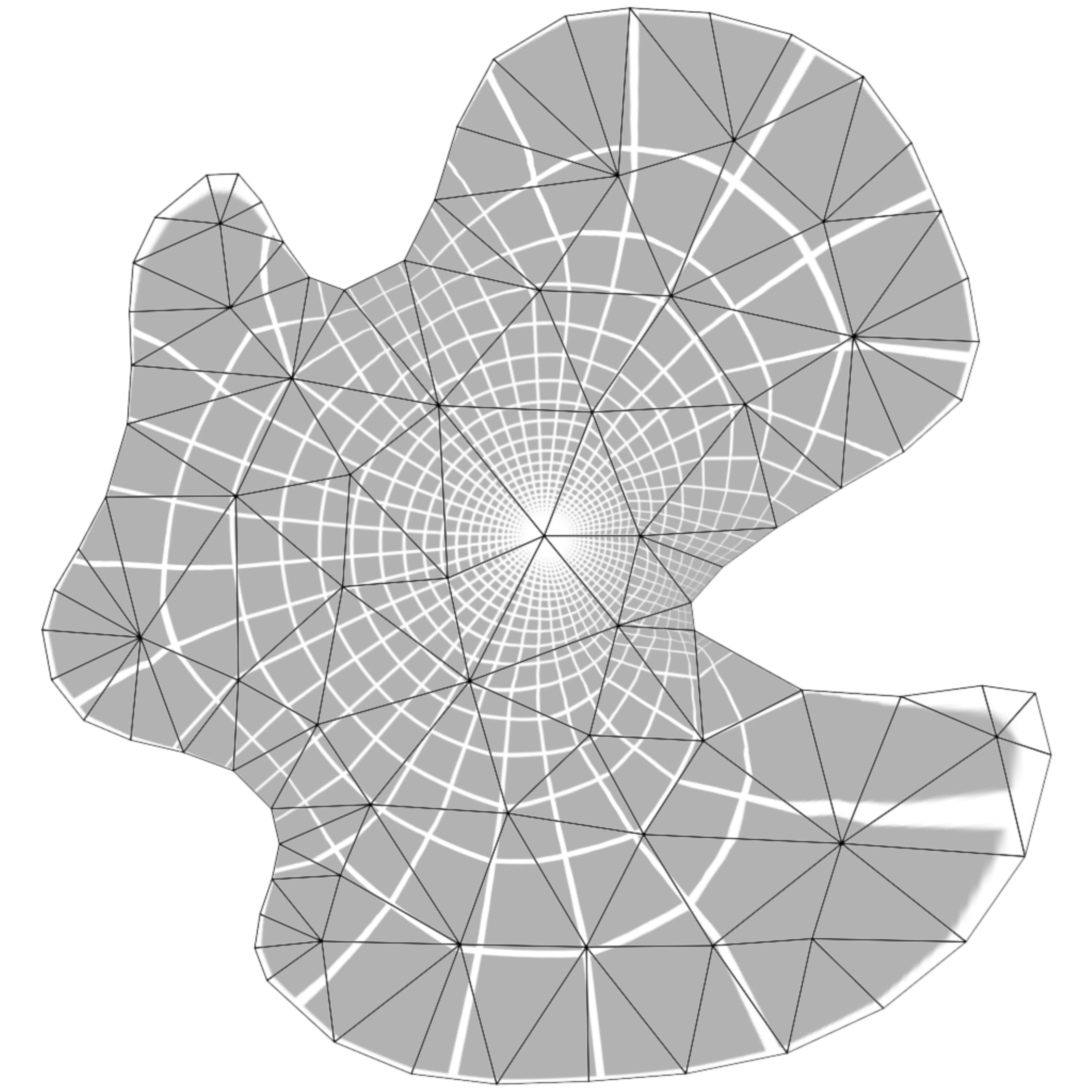}%
  \includegraphics[width=0.45\textwidth]{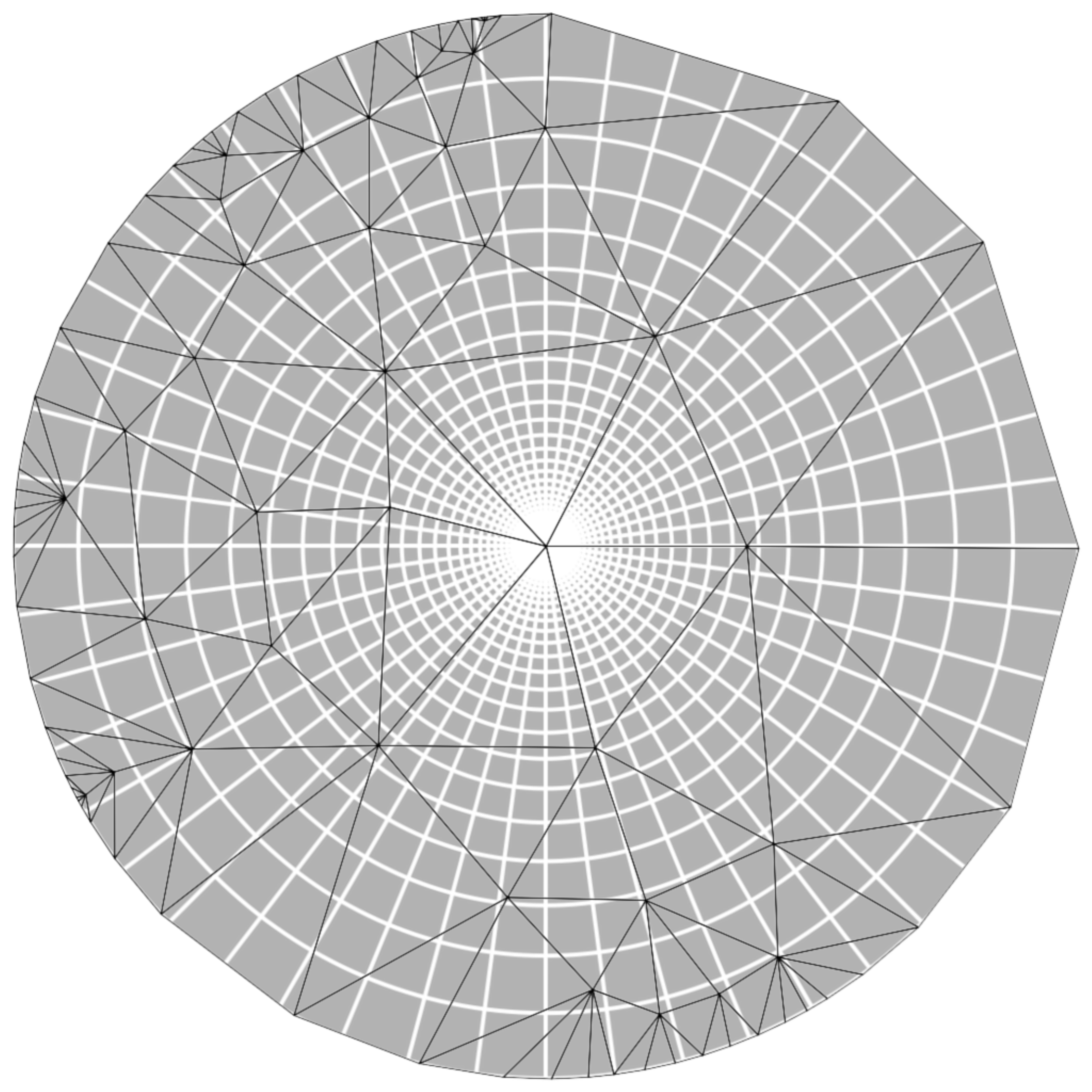}\\
  \includegraphics[width=0.45\textwidth]{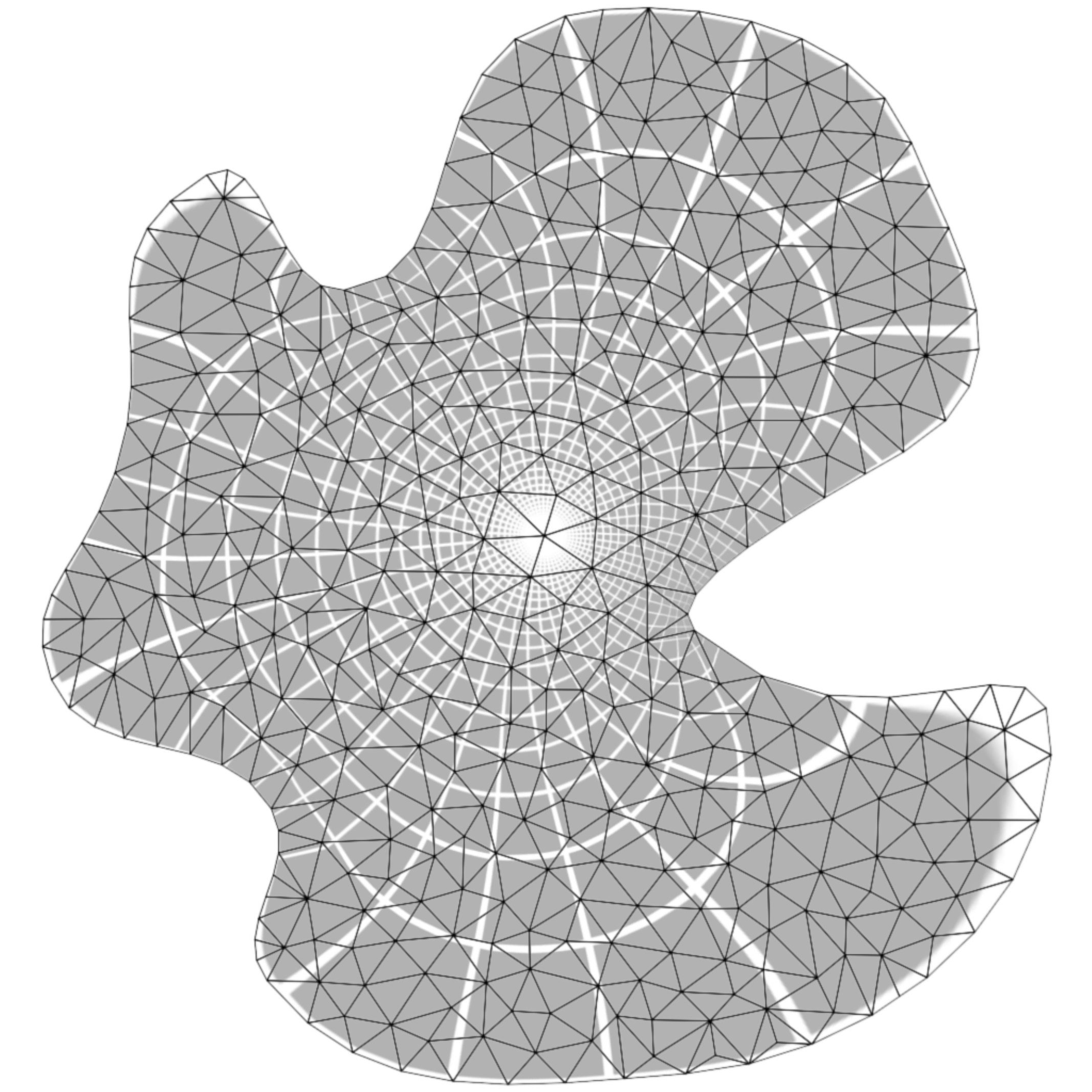}%
  \includegraphics[width=0.45\textwidth]{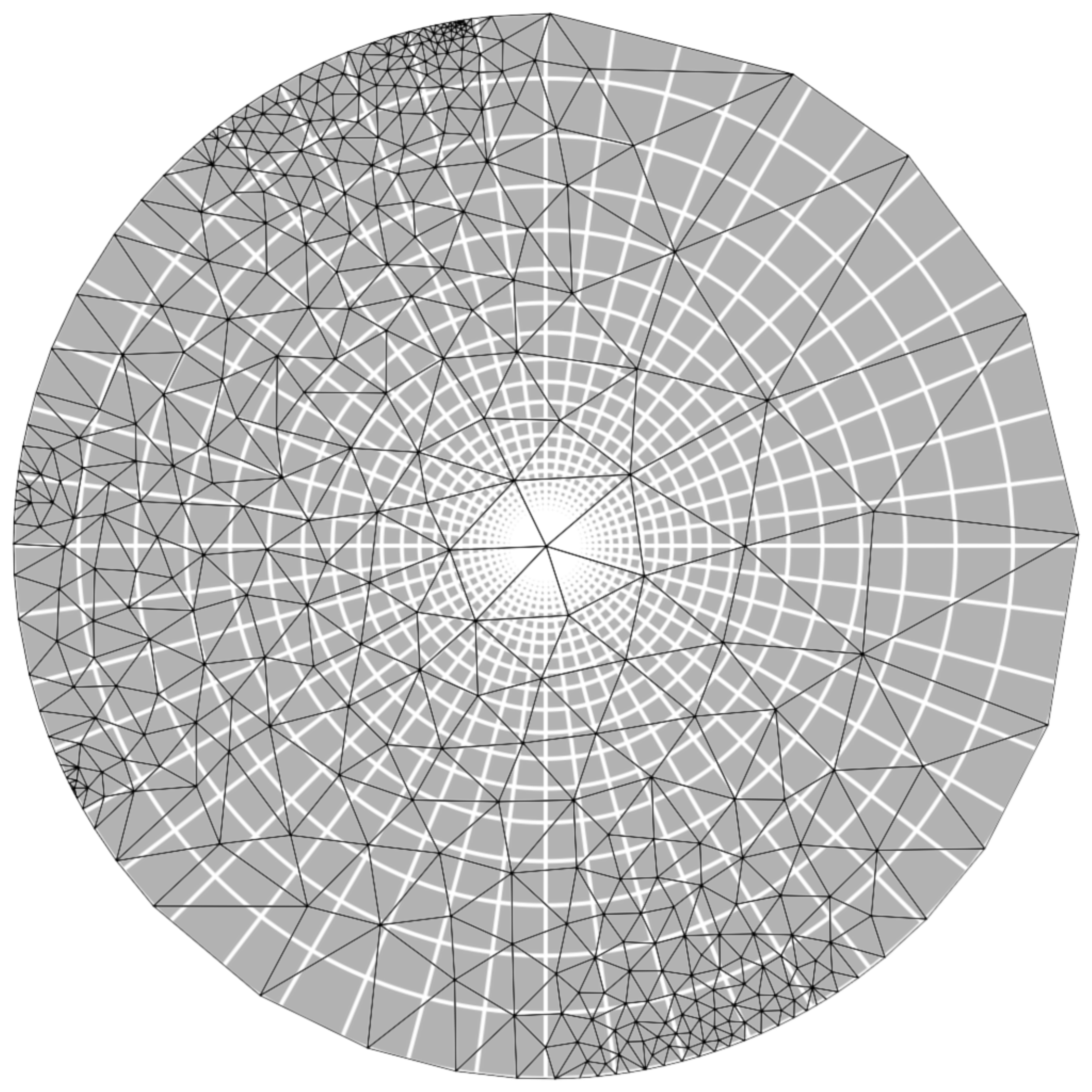}
  \caption{Discrete Riemann maps.}
  \label{fig:discrete_riemann}
\end{figure}%
M\"obius transformations preserve the discrete conformal class
(Section~\ref{sec:Moebius}), and this makes it possible to construct
discrete conformal maps to regions bounded by circular polygons
(Section~\ref{sec:disk}), discrete analogs of the classical
Riemann maps.

The first variational principle
(Section~\ref{sec:variational_principle_1}) involves a function of the
(logarithmic) scale factors $u$. The second variational principle
(Section~\ref{sec:variational_principle_2}) involves a function of the
triangle angles. The two variational principles are Legendre duals in
a precise way, but we do not dwell on this point. The corresponding
variational principles of the classical smooth theory are discussed in
Appendix~\ref{sec:smooth}.

There are clear signs in Sections~\ref{sec:discrete_conformal}
and~\ref{sec:variational_principles} that indicate a connection with
hyperbolic geometry: the appearance of Milnor's Lobachevsky function
$\ML(x)$, the fact that the second variational principle is almost the
same as Rivin's variational principle for ideal hyperbolic polyhedra
with prescribed dihedral angles~\cite{rivin_euclidean_1994}, and the
definition of discrete conformal maps in terms of circumcircle
preserving piecewise projective functions
(Section~\ref{sec:discrete_conformal_maps}). This connection with two-
and three-dimensional hyperbolic geometry is the topic of
Section~\ref{sec:ideal_polyhedra}. Reversing a construction of
Penner~\cite{penner_decorated_1987} \cite{epstein_euclidean_1988}, we
equip a triangulated piecewise euclidean surface with a canonical
hyperbolic metric with cusps. Discrete conformal maps are precisely
the isometries with respect to this hyperbolic metric (Section~\ref{sec:hyp_struct_on_euc_triang}). The logarithmic
edge lengths $\lambda$
(Section~\ref{sec:discrete_conformal_equivalence}) and the
length-cross-ratios that characterize a discrete conformal class
(Section~\ref{sec:length-cross-ratio}) are Penner coordinates and
shear coordinates, respectively, of the corresponding hyperbolic
surface (Section~\ref{sec:penner_and_shear}). The problem of
flattening a triangulation discretely conformally is equivalent to
constructing an ideal hyperbolic polyhedron with prescribed intrinsic
metric
(Section~\ref{sec:ideal_hyp_poly_with_prescribed_intrinsic_metric}). With
this interpretation of discrete conformality in terms of
three-dimensional hyperbolic geometry, the two variational principles
of Section~\ref{sec:variational_principles} are seen to derive from
Schl\"afli's differential volume formula and Milnor's equation for the
volume of an ideal tetrahedron (Section~\ref{sec:hyperbolic_volume}).

Once this connection between discrete conformality and hyperbolic
polyhedra is established, it is straightforward to obtain a modified
version of discrete conformality that pertains to triangulations
composed of hyperbolic triangles instead of euclidean ones
(Section~\ref{sec:dconf_hyperbolic}). This is the theory of discrete
conformal uniformization of triangulated higher genus surfaces over
the hyperbolic plane. It has been applied, for example, for the
hyperbolization of euclidean
ornaments~\cite{von_gagern_hyperbolization_2009}. (It is equally
straightforward to obtain a corresponding theory for spherical
triangulations, but the functions involved in the corresponding
variational principles are not convex. We do not pursue this branch of
the theory here.)

The connection with hyperbolic polyhedra entrains a connection between
the discrete notion of conformality considered here and circle
patterns, another discretization of the same concept. Thurston
introduced patterns of circles as an elementary geometric
visualization of hyperbolic polyhedra~\cite[Chapter
13]{thurston_geometry_????}. He rediscovered Koebe's circle packing
theorem~\cite{koebe_kontaktprobleme_1936} and showed that it followed
from Andreev's work on hyperbolic
polyhedra~\cite{andreev_convex_1970-1} \cite{andreev_convex_1970}, see
also~\cite{roeder_andreevs_2007}. Thurston's conjecture that circle
packings could be used to approximate the classical Riemann map, which
was later proved by Rodin and Sullivan~\cite{rodin_convergence_1987},
set off a flurry of research that lead to a full-fledged theory of
discrete analytic functions and conformal maps based on packings and
patterns of circles~\cite{stephenson_introduction_2005}. (The circle
packing version of Luo's ``combinatorial Yamabe flow'' is the
``combinatorial Ricci flow'' of Chow and
Luo~\cite{chow_combinatorial_2003} \cite{gu_computational_2008}.) The
relationship between these two theories of discrete conformality is
now clear: The circle packing theory deals with hyperbolic polyhedra
with \emph{prescribed dihedral angles} and the notion of discrete
conformality considered here deals with hyperbolic polyhedra with
\emph{prescribed metric}. In Section~\ref{sec:circle_patterns} of the
appendix we discuss the relationship between the variational
principles for discrete conformal maps
(Section~\ref{sec:variational_principles}) and two variational
principles for circle patterns. One is due to
Rivin~\cite{rivin_euclidean_1994} (see also the recent survey article
by Futer and Gu\'eritaud~\cite{futer11:_from}, which provides a wealth
of material that is otherwise difficult to find), and the other is
again related to it by the same sort of singular Legendre
duality~\cite{bobenko04:_variat_princ_for_circl_patter}. Variational
principles for circle patterns are important in discrete differential
geometry in particular for constructing discrete minimal
surfaces~\cite{bobenko06:_minim_surfac_from_circl_patter}.
Instead of triangulations one can consider meshes composed of polygons
that are inscribed in circles (Section~\ref{sec:circular_meshes}), and
we consider the problem to map multiply connected domains to domains
bounded by polygons inscribed in circles, a discrete version of circle
domains (Section~\ref{sec:circle_domains}).

Two important questions are not addressed in this paper. The first is
the question of convergence. Of course we do believe that (under not
too restrictive assumptions that have yet to be worked out) discrete
conformal maps approximate conformal maps if the triangulation is fine
enough. Figure~\ref{fig:discrete_riemann} clearly suggests that a
version of the Rodin--Sullivan theorem~\cite{rodin_convergence_1987}
also holds in this case. But all this has yet to be proved. 

The other question concerns the solvability of the discrete conformal
mapping problems of Section~\ref{sec:mapping_problems}. A solution may
not exist due to violated triangle inequalities. Fairly obvious
necessary conditions and how they relate to properties of the function
$E_{\T,\Theta,\lambda}$ appearing in the first variational principle
are discussed in Appendix~\ref{sec:nec_cond_exist}. In the numerous
numerical experiments that we have made, we have observed that a
solution exists if the necessary conditions are satisfied, no
triangles are almost degenerate to begin with, and the triangulation
is not too coarse. But to find necessary and sufficient conditions for
solvability seems to be an intractable problem in this setting. After
all, this would amount to giving necessary and sufficient conditions
for the existence of a (not necessarily convex) ideal hyperbolic
polyhedron with prescribed intrinsic metric and prescribed
combinatorial type. The way out is to restrict oneself to convex
polyhedra while widening the concept of discrete conformal map to
allow for combinatorial changes
(Section~\ref{sec:hyp_struct_on_euc_triang}). Rivin proved that any
hyperbolic metric with cusps on the sphere is realized by a unique
ideal polyhedron~\cite{rivin_intrinsic_1994}. This translates into an
existence statement for discrete conformal maps. 
(Conversely, this suggests a variational proof of Rivin's theorem
very similar to the recent constructive
proof~\cite{bobenko_alexandrovs_2008} of Alexandrov's classical
polyhedral realization theorem~\cite{alexandrov05:_convex_polyh}.)

Previous versions of this article have been available as preprint
\href{http://arxiv.org/abs/1005.2698}{arXiv:1005.2698} since May
2010. For the published version, the text has been restructured
according to the suggestions of the referee. The mathematical content
has not changed.

\section{Discrete conformal equivalence and maps}
\label{sec:discrete_conformal}

\subsection{Discrete conformal equivalence}
\label{sec:discrete_conformal_equivalence}

A \emph{surface} is a connected $2$-dimensional manifold, possibly
with boundary. A \emph{surface triangulation}, or \emph{triangulation}
for short, is a surface that is a CW complex whose faces ($2$-cells)
are triangles which are glued edge-to-edge. We will denote the sets of
vertices ($0$-cells), edges ($1$-cells), and faces of a
triangulation~$\T$ by $V_{\T}$, $E_{\T}$, and $T_{\T}$, and we will
often drop the subscript $\T$ if the triangulation is clear from the
context. We will also write $A_{\T}$ for the set of triangle angles,
where angles means corners, or triangle-vertex incidences, not angle
measures. 

A \emph{euclidean surface triangulation}, or \emph{euclidean
  triangulation} for short, is a surface triangulation equipped with a
metric so that $\T\setminus V_{\T}$ is locally isometric to the
euclidean plane, or half-plane if there is boundary, and the edges are
geodesic segments. In other words, a euclidean surface triangulation
is a surface consisting of euclidean triangles that are glued
edge-to-edge. At the vertices, the metric may have cone-like
singularities.

A euclidean triangulation is uniquely determined by a triangulation
$\T$ and a function $\ell:E_{\T}\rightarrow\R_{>0}$
assigning a length to every edge in such a way that the triangle
inequalities are satisfied for every triangle in $T_{\T}$. We
call such a positive function $\ell$ on the edges that satisfies all
triangle inequalities a \emph{discrete metric} on $\T$, and we
denote the resulting euclidean triangulation by $(\T, \ell)$.

In this paper, we will assume for simplicity that the triangulations
are simplicial complexes. This means that a triangle may not be glued
to itself at a vertex or along an edge, and the intersection of two
triangles is either empty or it consists of one vertex or one
edge. This restrictions to simplicial complexes allows us to use
simple notation: we will denote by $ij$ the edge with vertices $i$ and
$j$, by $ijk$ the triangle with vertices $i$, $j$, and $k$, and by $
\begin{smallmatrix}
  i\\jk
\end{smallmatrix}
$ 
the corner at vertex
$i$ in triangle $ijk$. If $f,g,h$, and $\phi$ are functions on $V$,
$E$, $T$, and $A$, respectively, we will write $f_{i}$, $g_{ij}$,
$h_{ijk}$, and $\phi_{jk}^{i}$ for $f(i)$, $g(ij)$, $h(ijk)$, and
$\phi(
\begin{smallmatrix}
  i\\jk
\end{smallmatrix}
)$.  But while this restriction to simplicial complexes is
notationally very convenient, it is \emph{a priori} uncalled
for. There are a few exceptions, like
Sections~\ref{sec:sphere} and~\ref{sec:disk} on mapping to the sphere and disk,
but in general the domain of validity of the theory presented here
extends beyond the simplicial case.

The vector spaces of real-valued functions on the sets of vertices,
edges, and angles will be denoted by $\R^{V}$, $\R^{E}$, and $\R^{A}$,
respectively.

\begin{definition}[Luo \cite{luo_combinatorial_2004}]
  \label{def:d_conf_equiv}
  Two combinatorially equivalent euclidean triangulations,
  $(\T, \ell)$ and $(\T, \tilde\ell)$, are
  \emph{discretely conformally equivalent} if the discrete metrics
  $\ell$ and $\tilde\ell$ are related by
  \begin{equation}
    \label{eq:tilde_ell}
    \tilde\ell_{ij}=e^{\frac{1}{2}(u_i+u_j)}\ell_{ij}
  \end{equation}
  for some $u\in \R^{V}$. This defines an equivalence
  relation on the set of discrete metrics on $\T$, of which an
  equivalence classes is called a \emph{discrete conformal class} of
  discrete metrics, or a \emph{discrete conformal structure} on
  $\T$.
\end{definition}

Instead of the edge lengths $\ell$ we will often use the logarithmic
lengths
\begin{equation}
  \label{eq:lambda}
  \lambda = 2\log\ell.
\end{equation}
(The reason for the factor of $2$ will become apparent
in Section~\ref{sec:ideal_polyhedra}.) In terms of these logarithmic lengths,
relation~\eqref{eq:tilde_ell} between $\ell$ and $\tilde\ell$ becomes
linear:
\begin{equation}
  \label{eq:tilde_lambda}
  \tilde\lambda_{ij}=\lambda_{ij} + u_{i} + u_{j}.
\end{equation}

\begin{remark}[Dimension of ``discrete Teichm\"uller space'']
  \label{rem:dim_discrete_Teich}
  The set of all discrete metrics on a triangulation $\T$ is a
  manifold whose dimension is the number of edges, $|E|$. This
  manifold of metrics is fibered by the discrete conformal classes,
  each of which is a submanifold of dimension $|V|$, the number of
  vertices. The corresponding ``discrete Teichm\"uller space'', i.e.,
  the manifold of discrete conformal classes, has dimension
  $|E|-|V|$. If $\T$ triangulates a closed surface of genus $g$, one
  obtains $|E|-|V|=6g-6+2|V|$, which is also the dimension of
  $\Teich_{g,|V|}$, the Teichm\"uller space of a genus $g$ Riemann
  surfaces with $|V|$ punctures. This is no coincidence. The discrete
  conformal classes actually correspond to points in the Teichm\"uller
  space $\Teich_{g,|V|}$ (see
  Section~\ref{sec:hyp_struct_on_euc_triang}).
\end{remark}

\subsection{The two most simple cases}
\label{sec:two_triv_examples}

\noindent (1)\; If the triangulation $\T$ consists of a single
triangle $ijk$, then any two euclidean triangulations $(\T, \ell)$ and
$(\T, \tilde\ell)$ are discretely conformally equivalent, because the
three equations
\begin{equation*}
  \tilde\ell_{ij}=e^{\frac{1}{2}(u_i+u_j)}\ell_{ij},\qquad
  \tilde\ell_{jk}=e^{\frac{1}{2}(u_j+u_k)}\ell_{jk},\qquad
  \tilde\ell_{ki}=e^{\frac{1}{2}(u_k+u_i)}\ell_{ki}
\end{equation*}
always have a unique solution for $u_{i}$, $u_{j}$ and $u_{k}$:
\begin{equation*}
  e^{u_{i}}=
  \frac{\tilde\ell_{ij}\ell_{jk}\tilde\ell_{ki}}
  {\ell_{ij}\tilde\ell_{jk}\ell_{ki}}, 
  \quad\ldots
\end{equation*}

\noindent (2)\; Now let $\T$ be the triangulation consisting of two
triangles $ijk$ and $ilj$ glued along edge $ij$ as shown in
Figure~\ref{fig:lcr}, and let $\ell$ and $\tilde\ell$ be two discrete
metrics on $\T$.
\begin{figure}
  \centering
  \includegraphics{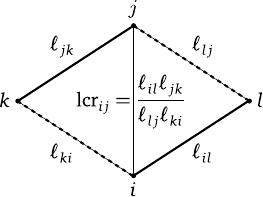}
  \caption{The length-cross-ratio on edge $ij$. The lengths of the
    bold solid and bold dashed edges appear in the numerator and
    denominator, respectively.}
  \label{fig:lcr}
\end{figure}
What is the condition for $(\T, \ell)$ and $(\T,
\tilde\ell)$ to be discretely conformally equivalent? For each
triangle considered separately, the corresponding
equations~\eqref{eq:tilde_ell} determine unique solutions for the
values of $u$ on its vertices. For each of the common vertices $i$ and
$j$ one obtains two values and the necessary and sufficient
condition for discrete conformal equivalence is that they are equal,
which is equivalent to the condition
\begin{equation*}
\frac{\ell_{il}\ell_{jk}}{\ell_{lj}\ell_{ki}}=
\frac{\tilde\ell_{il}\tilde\ell_{jk}}{\tilde\ell_{lj}\tilde\ell_{ki}}\,.
\end{equation*}

\subsection{Length-cross-ratios}
\label{sec:length-cross-ratio}

The simple reasoning of Section~\ref{sec:two_triv_examples} extends to
the general case: Let~$\T$ be any triangulation, and let $\ell$ and
$\tilde\ell$ be two discrete metrics on $\T$. For each triangle $ijk$
of~$\T$, considered separately, equations~\eqref{eq:tilde_ell}
determine unique values for $u$. Thus, for each vertex $i\in V$, one
obtains one value for $u_{i}$ per adjacent triangle. These values are
in general different. They agree for each vertex if and only if the
discrete metrics $\ell$ and $\tilde\ell$ are discretely conformally
equivalent. Since the vertex links are connected, it suffices to
consider values obtained from adjacent triangles. This leads to
Proposition~\ref{prop:conf_equiv_in_terms_of_lcr} below, where the
condition for discrete conformal equivalence is given in terms of the
so-called length-cross-ratios:

\begin{definition}
  For each interior edge $ij$ between triangles $ijk$ and $ilj$ as in
  Figure~\ref{fig:lcr}, define the \emph{length-cross-ratio} induced
  by $\ell$ to be
  \begin{equation}
    \label{eq:lcr}
    \lcr_{ij}=\frac{\ell_{il}\ell_{jk}}{\ell_{lj}\ell_{ki}}.
  \end{equation}
\end{definition}

This definition implicitly assumes that an orientation of the
triangulated surface has been chosen. The other choice of orientation
leads to reciprocal values for the length-cross-ratios. (For
non-orientable surfaces, the length-cross-ratios are well defined on
the interior edges of the oriented double cover.)

If the quadrilateral $iljk$ is embedded in $\C$, then the
length-cross-ratio $\lcr_{ij}$ is just the absolute value of the
complex cross ratio of the vertex positions $z_{i}, z_{l}, z_{j},
z_{k}$,
\begin{equation*}
  \ccr(z_{1},z_{2},z_{3},z_{4})=
  \frac{(z_{1}-z_{2})(z_{3}-z_{4})}{(z_{2}-z_{3})(z_{4}-z_{1})}\,.
\end{equation*}

Discretely conformally equivalent metrics $\ell$, $\tilde\ell$ induce
the same length-cross-ratios, because the scale factors $e^{u/2}$
cancel. By the reasoning above, the converse is also true.

\begin{proposition}
  \label{prop:conf_equiv_in_terms_of_lcr}
  Two euclidean triangulations $(\T,\ell)$ and $(\T,\tilde\ell)$ are
  discretely conformally equivalent if and only if for each interior
  edge $ij\in E_{\T}$, the induced length-cross-ratios are
  equal: $\lcr_{ij}=\widetilde{\lcr}_{ij}$.
\end{proposition}

\subsection{The product of length-cross-ratios around a vertex}
\label{sec:product_of_length_cross_ratios_around_a_vertex}

Let us denote the sets of interior edges and interior vertices by
$\Eint$ and $\Vint$, respectively. Which functions
$\Eint\rightarrow\R_{>0}$ can arise as length-cross-ratios? A
necessary condition is that the product of length-cross-ratios on the
edges around an interior vertex is~$1$, because all lengths $\ell$
cancel.
\begin{equation}
  \label{eq:lcr_product}
  \text{For all }i\in\Vint:\quad
  \prod_{j:ij\in E}\lcr_{ij} = 1\,. 
\end{equation}
If we ignore the triangle inequalities, this condition is also
sufficient:
\begin{proposition}
  \label{prop:ell_from_lcr}
  Let $\lcr:\Eint\rightarrow\R_{>0}$ be any positive function on the
  set of interior edges. There exists a positive function
  $\ell:E\rightarrow\R_{>0}$ on the set of edges
  satisfying~\eqref{eq:lcr} for every interior edge $ij$, if and only
  if condition~\eqref{eq:lcr_product} holds.
\end{proposition}

\begin{proof}
  It remains to show that if $\lcr\in(\R_{>0})^{\Eint}$ satisfies
  condition~\eqref{eq:lcr_product}, then the system of
  equations~\eqref{eq:lcr} has a solution. In fact, we will explicitly
  construct such a solution. To this end, we introduce auxiliary
  parameters $c$, which are defined on the set of angles $A$ of the
  triangulation: Given $\ell\in(\R_{>0})^{E}$, define
  $c\in(\R_{>0})^{A}$ by
  \begin{equation}
    \label{eq:aux_param}
    c^{i}_{jk}=\frac{\ell_{jk}}{\ell_{ij}\ell_{ki}}\,,
  \end{equation}
  see Figure~\ref{fig:aux_param}.
  \begin{figure}
    \centering
    \includegraphics{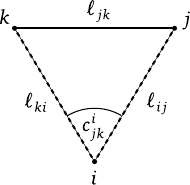}
    \caption{The parameters $c_{jk}^{i}$, defined on the set of
      triangle angles $A$.}
    \label{fig:aux_param}
  \end{figure}
  In terms of these parameters, the length-cross-ratios induced by
  $\ell$ are
  \begin{equation}
    \label{eq:lcr_aux_param}
    \lcr_{ij} = \frac{c^{i}_{jk}}{c^{i}_{lj}}\,,
  \end{equation}
  where $l$, $j$, $k$ occur in the link of $i$ in this cyclic order,
  as in Figure~\ref{fig:lcr}. (For a geometric interpretation of the
  parameters~$c_{jk}^{i}$ in terms of hyperbolic geometry, see
  Section~\ref{sec:decorated_ideal_triangs_and_tets}.)
  
  Now suppose $\lcr\in(\R_{>0})^{\Eint}$ satisfies
  condition~\eqref{eq:lcr_product}. Then it is easy to find a solution
  $c\in(\R_{>0})^{A}$ of equations~\eqref{eq:lcr_aux_param}, because
  each equation involves only two values of $c$ on consecutive angles
  at the same vertex. So one can freely choose one $c$-value per
  vertex and successively calculate the values on neighboring angles
  around the same vertex by multiplying (or dividing) with the values
  of $\lcr$ on the edges in between.

  Next, solve equations~\eqref{eq:aux_param} for $\ell$, where $c$ is
  the solution to equations~\eqref{eq:lcr_aux_param} just
  constructed. This is also easy: The length of an edge $ij$ is
  determined by the values of $c$ on the two adjacent angles on either
  side,
  \begin{equation*}
    \ell_{ij}=(c^{i}_{jk}c^{j}_{ki})^{-\frac{1}{2}}.
  \end{equation*}
  (Check that the two $c$-values on the other side give the same
  value.)  Thus we have constructed a function $\ell\in(\R_{>0})^{E}$
  satisfying equations~\eqref{eq:lcr} for the given function
  $\lcr\in(\R_{>0})^{\Eint}$.
\end{proof}

\subsection{M\"obius invariance of discrete conformal structures}
\label{sec:Moebius}

The group of M\"obius transformations of
$\widehat{\R^{n}}=\R^{n}\cup\{\infty\}$ is the group generated by
inversions in spheres. (Planes are considered spheres through
$\infty$.) The group of M\"obius transformations is also generated by
the similarity transformations (which fix $\infty$), and inversion in
the unit sphere. M\"obius transformations are conformal, and a famous
theorem of Liouville says that for $n>2$, any conformal map of a domain
$U\subset\R^{n}$ is the restriction of a M\"obius transformation.

Let $\T$ be a triangulation and let $\|\cdot\|$ denote the euclidean
norm on $\R^{n}$, $n\geq 2$. Suppose $v:V_{\T}\rightarrow\R^{n}$ maps
the vertices of each triangle to three affinely independent
points. Then $v$ induces a discrete metric
$\ell_{ij}=\|v_{i}-v_{j}\|$. Two maps $v,\tilde
v:V\rightarrow\R^{n}\subset\widehat{\R^{n}}$ are \emph{related by a
  M\"obius transformation} if there is a M\"obius transformation $T$
such that $\tilde v = T\circ v$.

\begin{proposition}
  If two maps $V_{\T}\rightarrow\R^{n}$ are related by a
  M\"obius transformation, then the induced discrete metrics are
  discretely conformally equivalent.
\end{proposition}

\begin{proof}
  The claim is obvious if the relating M\"obius transformation is a
  similarity transformation. For inversion in the unit sphere,
  $x\mapsto\frac{1}{\|x\|^{2}}\,x$, it follows from the identity
  \begin{equation*}
    \Big\|
    \frac{1}{\|p\|^{2}}\,p - \frac{1}{\|q\|^{2}}\,q
    \Big\|
    =\frac{1}{\|p\|\,\|q\|}\,\|p-q\|\,.
  \end{equation*}
\end{proof}

\begin{remark}
  For $n=2$ there is an obvious alternative argument involving the
  complex cross ratio. One can extend this argument to $n>2$. The only
  difficulty is to define a complex cross-ratio for four points in
  $\R^{n}$ if $n>2$, such that it is invariant under M\"obius
  transformations. Such a cross-ratio can be defined up to complex
  conjugation by identifying a 2-sphere through the four points
  conformally with the extended complex plane $\hat\C$. This involves
  several choices: a choice of 2-sphere if the four points are
  cocircular, a choice of orientation of the $2$-sphere, and choice of
  orientation preserving conformal map to $\hat\C$. Only the choice of
  orientation makes a difference, the two choices leading to conjugate
  values for the cross-ratio. The length-cross-ratio is the absolute
  value of this complex cross-ratio, so the ambiguity with respect to
  complex conjugation does not matter.
\end{remark}

\subsection{Discrete conformal maps}
\label{sec:discrete_conformal_maps}

\begin{figure}
  \centering
  \includegraphics[height=0.36\textwidth]{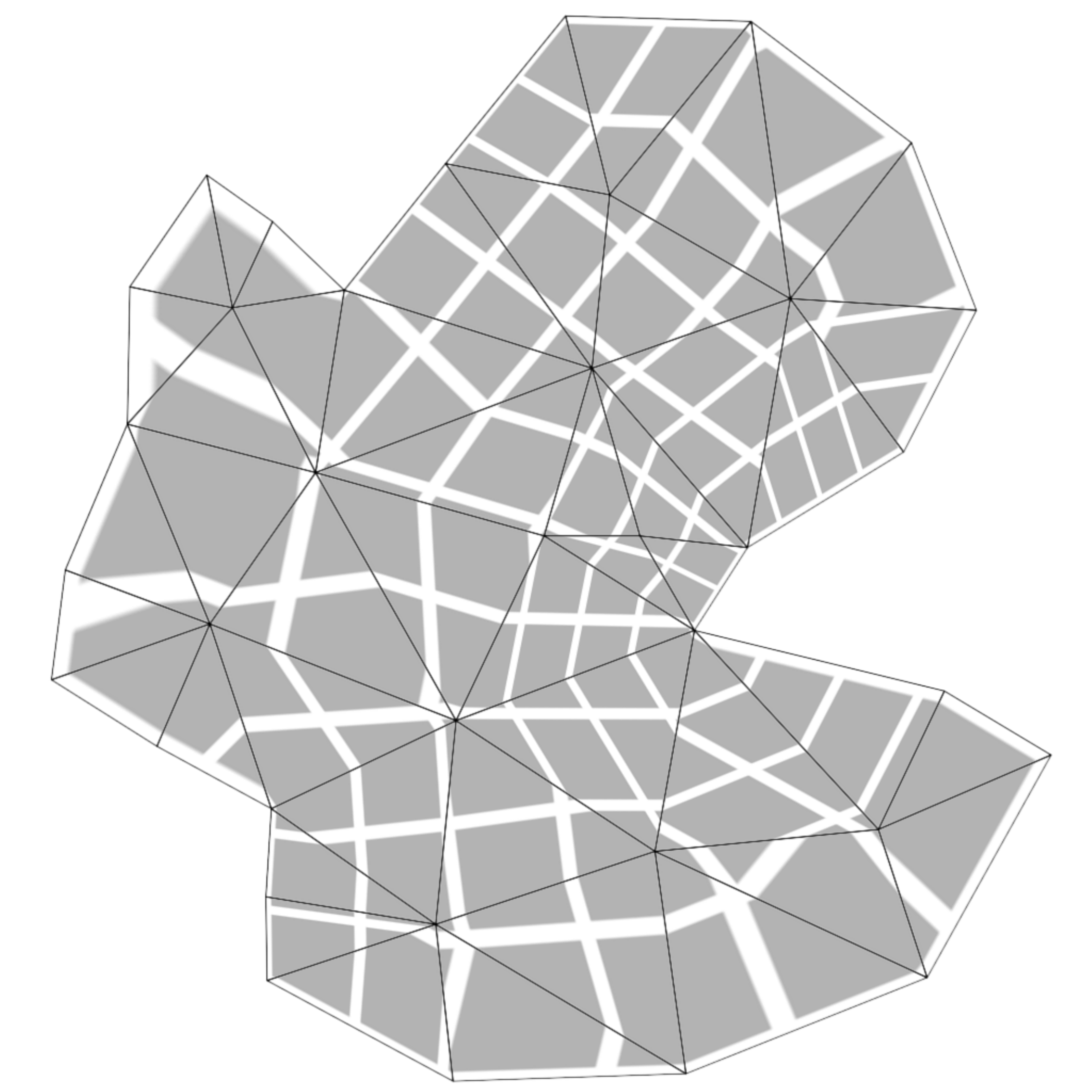}%
  \includegraphics[height=0.36\textwidth]{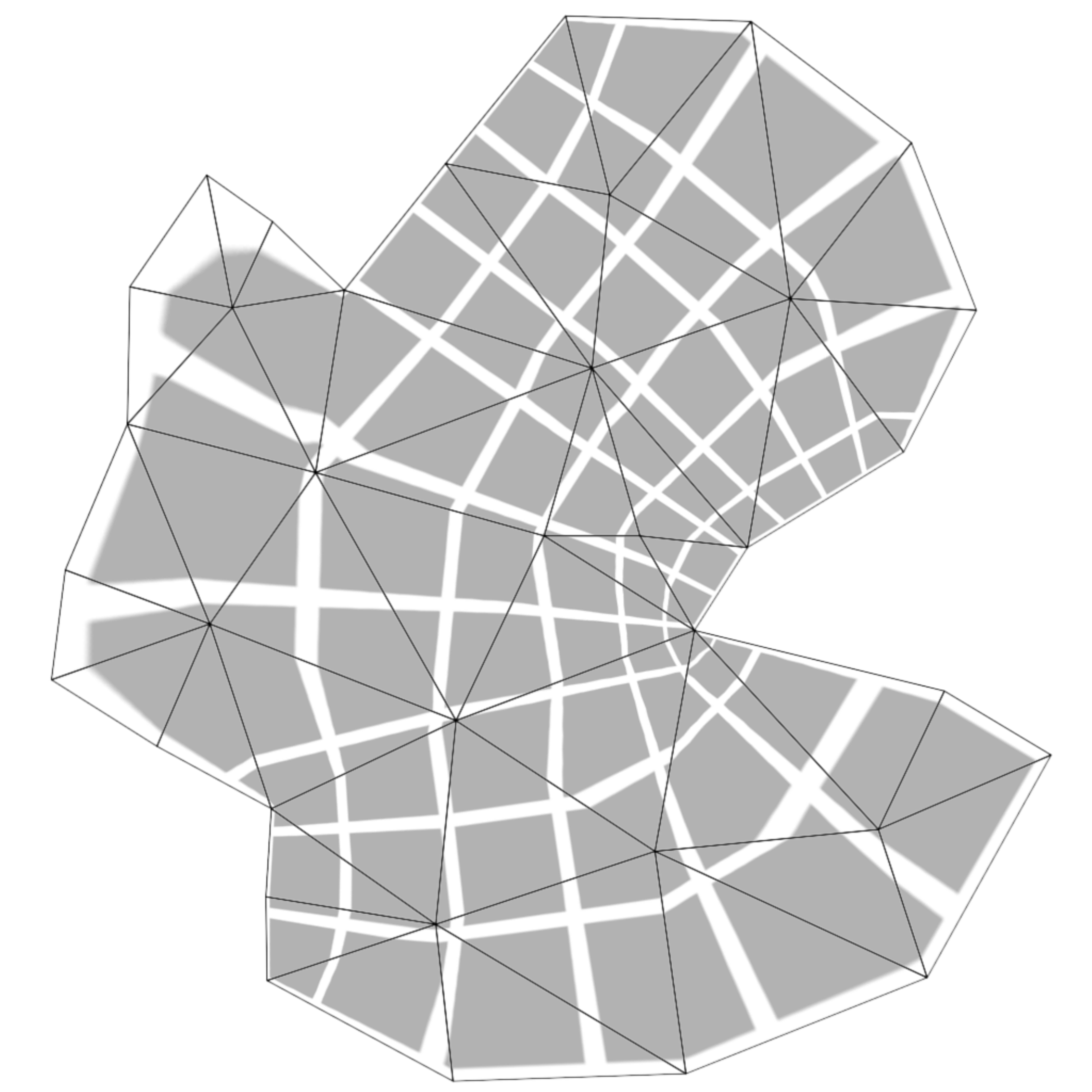}\quad
  \includegraphics[height=0.36\textwidth]{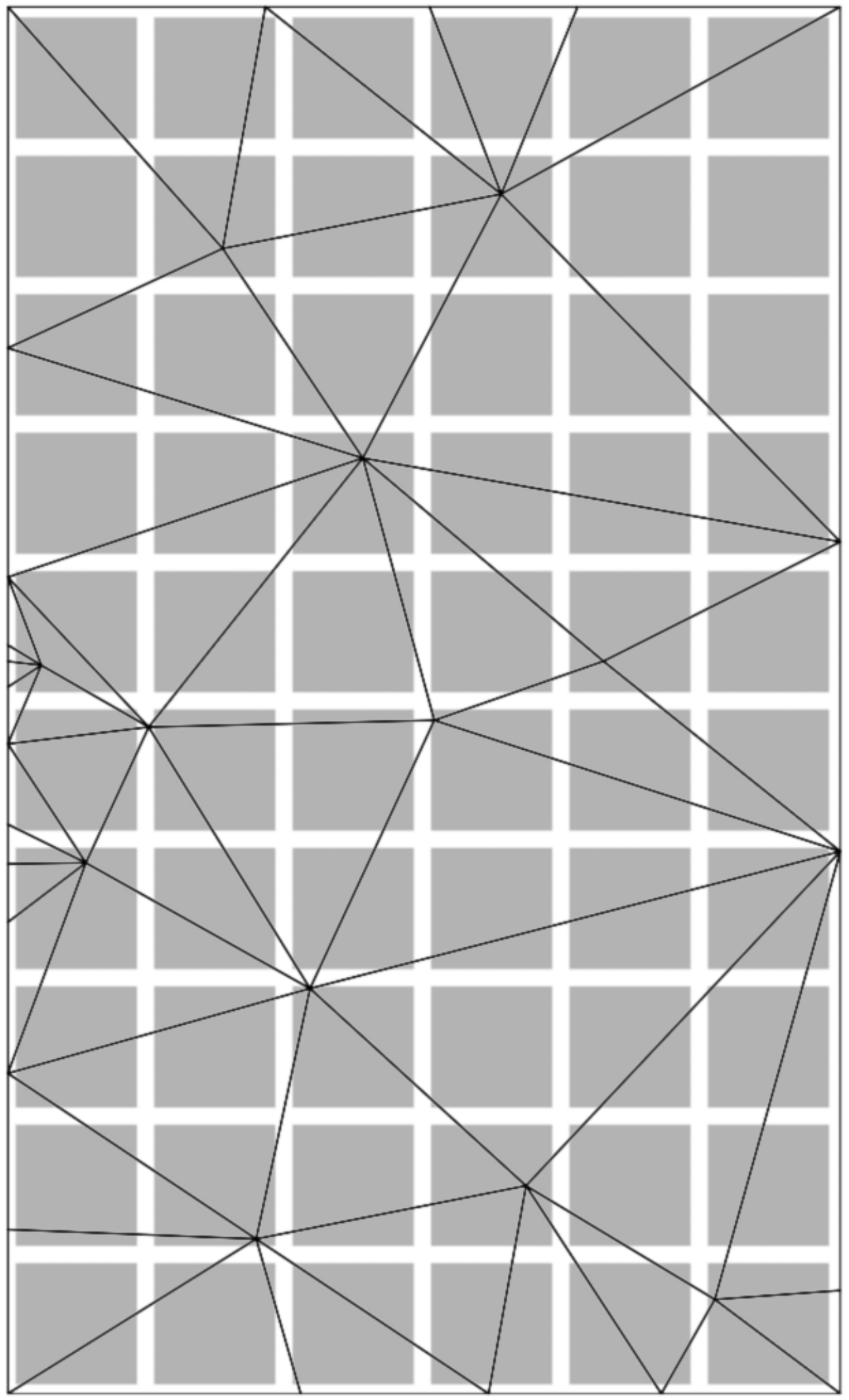}%
  \caption{A coarsely triangulated domain in the plane \emph{(middle)}
    is mapped to a rectangle \emph{(right)} by a discrete conformal
    map (see Definition~\ref{def:dconfmap}). Instead of using
    circumcircle preserving piecewise projective interpolation, one
    can also interpolate linearly in each triangle. The result
    \emph{(left)} looks noticeably ``less smooth''.}
  \label{fig:pl_and_pp}
\end{figure}

So far we have only talked about discrete conformal equivalence. This
section deals with the matching notion of discrete conformal maps (see
Figure~\ref{fig:pl_and_pp}).

For any two euclidean triangles (with labeled vertices to indicate
which vertices should be mapped to which), there is a unique
projective map that maps one triangle onto the other and the
circumcircle of one onto the circumcircle of the other (see
Lemma~\ref{lem:mu_i} below). Let us call this map the
\emph{circumcircle preserving projective map} between the two
triangles.

\begin{definition}
  \label{def:dconfmap}
  A \emph{discrete conformal map} from one euclidean triangulation
  $(\T, \ell)$ to a combinatorially equivalent euclidean triangulation
  $(\T, \tilde\ell)$ is a homeomorphism whose restriction to every
  triangle is the circumcircle preserving projective map
  onto the corresponding image triangle. 
\end{definition}

Consider two combinatorially equivalent euclidean triangulations,
$(\T, \ell)$ and $(\T, \tilde\ell)$. For each individual triangle of
$(\T, \ell)$, there is a circumcircle preserving projective map to the
corresponding triangle of $(\T, \tilde\ell)$. But these maps do in
general not fit together continuously across edges. However, they do
fit together, forming a discrete conformal map, precisely if the
euclidean triangulations are discretely conformally equivalent:

\begin{theorem}
  \label{thm:dconfmap}
  The following two statements are equivalent:
  \begin{compactenum}[(i)]
  \item $(\T, \ell)$ and $(\T, \tilde\ell)$ are discretely conformally
    equivalent.
  \item There exists a discrete conformal map 
    $(\T, \ell)\rightarrow(\T, \tilde\ell)$.
  \end{compactenum}
\end{theorem}

The rest of this section is concerned with the proof of
Theorem~\ref{thm:dconfmap}. It follows easily from
Lemma~\ref{lem:mu_i} below, which provides an analytic description of
the circumcircle preserving projective map between two individual
triangles.

Consider two triangles $\Delta$ and $\tilde\Delta$ in the euclidean
plane, and let $(x_{i},y_{i})$ and $(\tilde x_{i},\tilde y_{i})$,
$i\in\{1,2,3\}$, be the coordinates of their vertices in a Cartesian
coordinate system. Let $\ell_{ij}$ and $\tilde \ell_{ij}$ be
the side lengths,
\begin{equation*}
  \ell_{ij}^{2}=(x_{i}-x_{j})^{2}+(y_{i}-y_{j})^{2},
\end{equation*}
and similarly for $\tilde \ell_{ij}$. Consider the euclidean plane as
embedded in the projective plane $\RP^{2}$ and let
$v_{i}=(x_{i},y_{i},1)$ and $\tilde v_{i}=(\tilde x_{i},\tilde
y_{i},1)$ be the homogeneous coordinate vectors of the vertices,
normalized so that the last coordinate is $1$. Then the projective
maps $f:\RP^{2}\rightarrow\RP^{2}$ that map $\Delta$ to $\tilde\Delta$
correspond via $f([v])=[F(v)]$ to the linear maps
$F:\R^{3}\rightarrow\R^{3}$ of homogeneous coordinates that satisfy
\begin{equation}
  \label{eq:F_of_v_i}
  F(v_{i})=\mu_{i}\tilde v_{i}
\end{equation}
for some ``weights'' $\mu_{i}\in\R\setminus\{0\}$.

\begin{lemma}
  \label{lem:mu_i}
  The projective map $f:[v]\mapsto[F(v)]$ maps the
  circumcircle of $\Delta$ to the circumcircle of $\tilde\Delta$ if and
  only if
  \begin{equation}
    \label{eq:mu_i}
    (\mu_{1},\mu_{2},\mu_{3})=\mu\,(e^{-u_{1}},e^{-u_{2}},e^{-u_{3}})
  \end{equation}
  where $u_{1}, u_{2}, u_{3}$ are the logarithmic scale factors
  satisfying the three equations~\eqref{eq:tilde_ell} for a single
  triangle and $\mu\in\R\setminus\{0\}$ is an arbitrary factor.
\end{lemma}

\begin{proof}[Proof of Lemma~\ref{lem:mu_i}]
  The circumcircle of $\Delta$ is
  \begin{equation*}
    \big\{[v]\in\RP^{2}\;\big|\;q(v)=0\big\},
  \end{equation*}
  where $q$ is the quadratic form
  \begin{equation*}
    q(x,y,z)=x^{2}+y^{2}+2axz+2byz+cz^{2}
  \end{equation*}
  with $a,b,c\in\R$ uniquely determined by the condition that
  \begin{equation}
    \label{eq:q_condition}
    q(v_{1})=q(v_{2})=q(v_{3})=0.
  \end{equation}
  In the same fashion, let the quadratic form describing the
  circumcircle of $\tilde\Delta$ be
  \begin{equation*}
    \tilde q(x,y,z)=x^{2}+y^{2}+2\tilde axz+2\tilde byz+\tilde cz^{2}
  \end{equation*}
  so that 
  \begin{equation}
    \label{eq:tilde_q_condition}
    \tilde q(\tilde v_{1})=\tilde q(\tilde v_{2})=\tilde q(\tilde v_{3})=0.
  \end{equation}
  We will also denote by $q$ and $\tilde q$ the corresponding
  symmetric bilinear forms:
  \begin{equation*}
    q(v)=q(v,v)\quad\text{and}\quad\tilde q(v)=\tilde q(v,v).
  \end{equation*}

  The projective map $f$ maps circumcircle to circumcircle if and only
  if $q$ and the pull-back $F^{*}\tilde q$ are linearly
  dependent. That is, if and only if
  \begin{equation*}
    \mu^{2}\,q(v,w) = \tilde q(F(v), F(w))
  \end{equation*}
  for all $v,w\in\R^{3}$ for some $\mu\in\R$. Since
  $v_{1},v_{2},v_{3}$ is a basis of $\R^{3}$ and because of
  equations~\eqref{eq:q_condition} and~\eqref{eq:tilde_q_condition},
  this is the case if and only if
  \begin{equation}
    \label{eq:q_mu_condition}
    \mu^{2}\,q(v_{i},v_{j})=
    \mu_{i}\mu_{j}\,\tilde q(\tilde v_{i}, \tilde v_{j})
  \end{equation}
  for $i,j\in\{1,2,3\}$, $i\not=j$.
  Now note that
  \begin{equation*}
    \ell_{ij}^{2} = q(v_{i}-v_{j}, v_{i}-v_{j})
    = -2\,q(v_{i}, v_{j}),
  \end{equation*}
  and similarly $\tilde\ell_{ij}^{2} = -2\,\tilde q(\tilde v_{i},
  \tilde v_{j}).$
  So condition~\eqref{eq:q_mu_condition} is equivalent to
  \begin{equation*}
    \mu^{2}\,\ell_{ij}^{2}=
    \mu_{i}\mu_{j}\,\tilde\ell_{ij}^{2}.
  \end{equation*}
  Solve equations~\eqref{eq:tilde_ell} for $u_{i}$ to
  obtain~\eqref{eq:mu_i}. This completes the proof of
  Lemma~\ref{lem:mu_i}.
\end{proof}

To prove Theorem~\ref{thm:dconfmap}, consider two euclidean
triangulations $(\T,\ell)$ and $(T,\tilde\ell)$, and a pair of
adjacent triangles $ijk$ and $jil$ of $T$. Embed the corresponding
euclidean triangles of $(\T,\ell)$ simultaneously isometrically in the
euclidean plane, and do the same for the corresponding two euclidean
triangles of $(T,\tilde\ell)$. Lemma~\ref{lem:mu_i} tells us what the
circumcircle preserving projective maps are, and we might as well
choose $\mu=1$ in both cases. These two maps fit together continuously
along edge $ij$ if and only if the values of $\mu_i=e^{-u_{i}}$ and
$\mu_j=e^{-u_{j}}$ from one triangle are proportional to those of the
other triangle. Since the value of
$e^{(u_{i}+u_{j})/2}=\tilde\ell_{ij}/\ell_{ij}$ is the same for both
triangles, this is the case if and only if the values of $\mu_i$ and
$\mu_j$, hence also those of $u_{i}$ and $u_{j}$, coincide for both
triangles. This holds for all interior edges if and only if
$(\T,\ell)$ and $(T,\tilde\ell)$ are discretely conformally
equivalent. This completes the proof of
Theorem~\ref{thm:dconfmap}.

\section{Discrete conformal mapping problems}
\label{sec:mapping_problems}

\subsection{Prescribing angle sums at vertices}
\label{sec:prescribe_angle_sums}

Consider the following type of discrete conformal mapping problem,
which is a discrete version of the problem considered by
Troyanov~\cite{troyanov86:_les}:

\begin{problem}[prescribed angle sums]
  \label{prob:prescribe_Theta}
  \textbf{\emph{Given}}
  \begin{compactitem}
  \item a surface triangulation~$\T$,
  \item a discrete conformal class $\cclass$ of discrete metrics
    on
    $\T$,
  \item a desired angle sum $\Theta_{i}$ for each vertex~$i\in
    V_{\T}$,
  \end{compactitem}
  \textbf{\emph{find}} a discrete metric $\tilde\ell$ in the conformal
  class $\cclass$ such that the euclidean triangulation $(\T,
  \tilde\ell)$ has angle sum $\Theta_{i}$ around each vertex $i\in
  V_{\T}$.
\end{problem}

If, in particular, the given desired angle sum $\Theta_{i}$ equals
$2\pi$ for every interior vertex~$i$, then
Problem~\ref{prob:prescribe_Theta} asks for a \emph{flat} euclidean
triangulation in the given conformal class which has prescribed angles
at the boundary. A flat and simply connected euclidean triangulation
can be developed in the plane by laying out one triangle after the
other. Thus, Problem~\ref{prob:prescribe_Theta} comprises as a special
case the following problem.

\begin{problem}[planar triangulation with prescribed boundary angles]
  \label{prob:flatten_with_Theta_on_bdy}
  \textbf{\emph{Given}} 
  \begin{compactitem}
  \item A euclidean triangulation~$(\T,\ell)$ that is
    topologically a disc and
  \item a desired angle sum $\Theta_{i}$ for each boundary vertex~$i$,
  \end{compactitem}
  \textbf{\emph{find}} a discretely conformally equivalent planar
  triangulation $(\T,\tilde\ell)$ with the given angle sums at the
  boundary. (The triangulated planar region may overlap with itself.)
\end{problem}

We also consider a more general type of problem than
Problem~\ref{prob:prescribe_Theta}. Suppose the discrete conformal
class $\cclass$ is given in the form of a representative metric
$\ell\in(\R_{>0})^{E}$. For some vertices $i$ we may prescribe the
(logarithmic) scale factor $u_{i}$ instead of the angle sum~$\Theta_{i}$:

\begin{problem}[prescribed angles sums and fixed scale factors]
  \label{prob:general}
  \textbf{\emph{Given}} 
  \begin{compactitem}
  \item a triangulation~$\T$,
  \item a function
    $\ell\in(\R_{>0})^{E}$ representing a conformal class,
  \item a partition $V=V_{0}\,\dot\cup\, V_{1}$ of the vertex set,
  \item a prescribed logarithmic scale factor $u_i\in\R$ for each
    vertex $i\in V_{0}$
  \item a prescribed angle sum $\Theta_{i}$ for each vertex $i\in
    V_{1}$
  \end{compactitem}
  \textbf{\emph{find}} logarithmic scale factors $u_{i}\in\R$ for the
  remaining vertices $i\in V_{1}$ so that $\tilde\ell$ determined by
  equations~\eqref{eq:tilde_ell} is a discrete metric and $(\T,\tilde\ell)$
  has the prescribed angle sum $\Theta_{i}$ around each vertex $i\in
  V_{1}$.
\end{problem}

For $V_{0}=\emptyset$, $V_{1}=V$, this is just
Problem~\ref{prob:prescribe_Theta}. (If the conformal class $\cclass$
is given in the form of length-cross-ratios
$\lcr\in(\R_{>0})^{\Eint}$, one can obtain a representative
$\ell\in(\R_{>0})^{E}$ using the method described in the constructive
proof of Proposition~\ref{prop:ell_from_lcr}.) 

Note that any instance of Problem~\ref{prob:general} can be reduced to
the special case where $u_{i}=0$ is prescribed for $i\in V_{0}$:
Simply apply first a discrete conformal change of
metric~\eqref{eq:tilde_ell} with the arbitrary prescribed $u_{i}$ for
$i\in V_{0}$.

Analytically, Problem~\ref{prob:general} amounts to solving a system
of nonlinear equations. For the unknown logarithmic scale factors
$u_{i}$ ($i\in V_{1}$), one has to solve the system of angle-sum
equations
\begin{equation}
  \label{eq:angle_sum_equation}
  \sum_{jk:ijk\in T} \tilde\alpha_{jk}^{i} = \Theta_{i},
\end{equation}
with one equation for each vertex $i\in V_{1}$. Here,
$\tilde\alpha_{jk}^{i}$ is the angle at $i$ in triangle $ijk$ of
$(\T,\tilde\ell)$. The angles $\tilde\alpha$ are nonlinear functions
of the new lengths $\tilde\ell$. They can be obtained by invoking, for
example, the cosine rule or the half-angle formula
\begin{equation}
  \label{eq:half_angle_formula}
  \tan\Bigg(\frac{\alpha_{jk}^{i}}{2}\Bigg) =
  \sqrt{
    \frac{
      (-\ell_{ij} + \ell_{jk} + \ell_{ki})(\ell_{ij} + \ell_{jk} - \ell_{ki})
    }{
      (\ell_{ij} - \ell_{jk} + \ell_{ki})(\ell_{ij} + \ell_{jk} + \ell_{ki})
    }
  }\;.
\end{equation}
(Tilde marks over $\alpha$ and $\ell$ have been omitted in this equation to
avoid visual clutter.)

\begin{theorem}
  \label{thm:dconfmap_unique}
  If Problem~\ref{prob:general} has a solution, then the solution is
  unique if $V_{0}\not=\emptyset$ (i.e., at least one scale factor is
  fixed) and unique up to scale if $V_{0}=\emptyset$. The solution can
  be found by minimizing a convex function.
\end{theorem}

\begin{proof}
  This follows from Propositions~\ref{prop:first_variational},
  \ref{prop:E_convex}, and~\ref{prop:E_extend}.
\end{proof}

\begin{corollary}
  If a solution to Problem~\ref{prob:prescribe_Theta} (or
  \ref{prob:flatten_with_Theta_on_bdy}) exists, it is unique up to
  scale, and it can be found by minimizing a convex function.
\end{corollary}

\begin{remark}
  An important special case of Problem~\ref{prob:general} is the
  following: prescribe the angle sum $\Theta_{i}=2\pi$ for interior
  vertices, and $u_{i}$ on the boundary. This is analogous to the
  following boundary value problem of the smooth theory: Given a
  smooth $2$-manifold with boundary $M$ equipped with a Riemannian
  metric $g$, find a conformally equivalent flat Riemannian metric
  $e^{2u}g$ with prescribed $u|_{\partial M}$. Suppose we measure the
  relative distortion of a conformally equivalent Riemannian metric by
  the Dirichlet energy of~$u$, $D(u) = \frac{1}{2}\int_{M}du\wedge{*du}$. Then the conformally equivalent flat Riemannian metrics
  with least distortion are those with $u|_{\partial M}=\const$ Thus,
  up to scale there is a unique least distortion solution obtained by
  setting $u|_{\partial
    M}=0$. \cite[Appendix~E]{springborn_conformal_2008}
\end{remark}

\subsection{Mapping to the sphere}
\label{sec:sphere}

If one can solve Problem~\ref{prob:general}, one can also find
discrete conformal maps from euclidean triangulations that are
topological spheres to polyhedra with vertices on the unit sphere, and
from euclidean triangulations that are topological disks to planar
triangulations with boundary vertices on the unit circle
(Section~\ref{sec:disk}).

Suppose $(\T,\ell)$ is a euclidean triangulation that is topologically
a sphere. To map it to a polyhedron with vertices on the unit sphere,
proceed as follows:

\begin{compactenum}[1.]
\item Choose a vertex $k$ and apply a discrete conformal change of
  metric~\eqref{eq:tilde_ell} so that afterwards all edges incident
  with $k$ have the same length, say $\tilde\ell_{ki}=1$ for all
  neighbors $i$ of $k$. For example, let $e^{u_{i}/2}=\ell_{ki}^{-1}$ if 
  $i$ is a neighbor of $k$ and $1$ otherwise.
\item Let $\T'$ be $\T$ minus the open star of $k$. This is
  topologically a closed disk.
\item Solve Problem~\ref{prob:general} for $T'$ with prescribed
  $\Theta_{i}=2\pi$ for interior vertices $i$ and prescribed $u_{i}=0$
  for boundary vertices. (Suppose a solution exists.) The result is a
  planar triangulation.
\item Map the vertices of this planar triangulation to the unit sphere
  by stereographic projection. Add another vertex (the image of the
  removed vertex $k$) on the sphere at the center of the stereographic
  projection. Build a geometric simplicial complex using these
  vertices and the combinatorics of $\T$.
\end{compactenum}

\begin{proposition}
  The result of this procedure is a polyhedron with vertices on the
  sphere that is discretely conformally equivalent to $(\T,\ell)$. (It
  may not be convex. It is also possible that the planar triangulation
  obtained in step three overlaps with itself. In this case the star
  of $k$ in the image polyhedron is not embedded.)
\end{proposition}

\begin{proof}
  After Step~1, the length-cross-ratio for an edge $ki$ incident with
  $k$ is the quotient of the lengths of two consecutive edges $mi$,
  $ij$ in the boundary of $T'$. This is not changed in Step~$3$
  because $u=0$ on the boundary. Further, the length-cross-ratio for
  an edge $ij$ opposite $k$ as in Figure~\ref{fig:lcr} is then the
  quotient $\ell_{il}/\ell_{lj}$. This is also not changed in Step~$3$
  because $u=0$ on the boundary. Now imagine that before Step~4 you
  reinsert $k$ at $\infty$ in the plane, which you identify with the
  (extended) complex plane. Then the absolute values of the complex
  cross-ratios for all edges are the same as in $(T,\ell)$.
\end{proof}

\begin{remark}
  The method presented here is a variation of a method
  described~\cite{springborn_conformal_2008}. The old version requires
  an input triangulation that is immersed in some $\R^{n}$ with
  straight edges.
\end{remark}

\subsection{Mapping to the disk}
\label{sec:disk}

Suppose $(\T,\ell)$ is a euclidean triangulation that is topologically
a closed disk. To map it to a triangulated circular polygon, proceed
as follows (see Figure~\ref{fig:map_to_disk_steps}):
\begin{figure}
  \centering
  \includegraphics[width=0.7\textwidth]{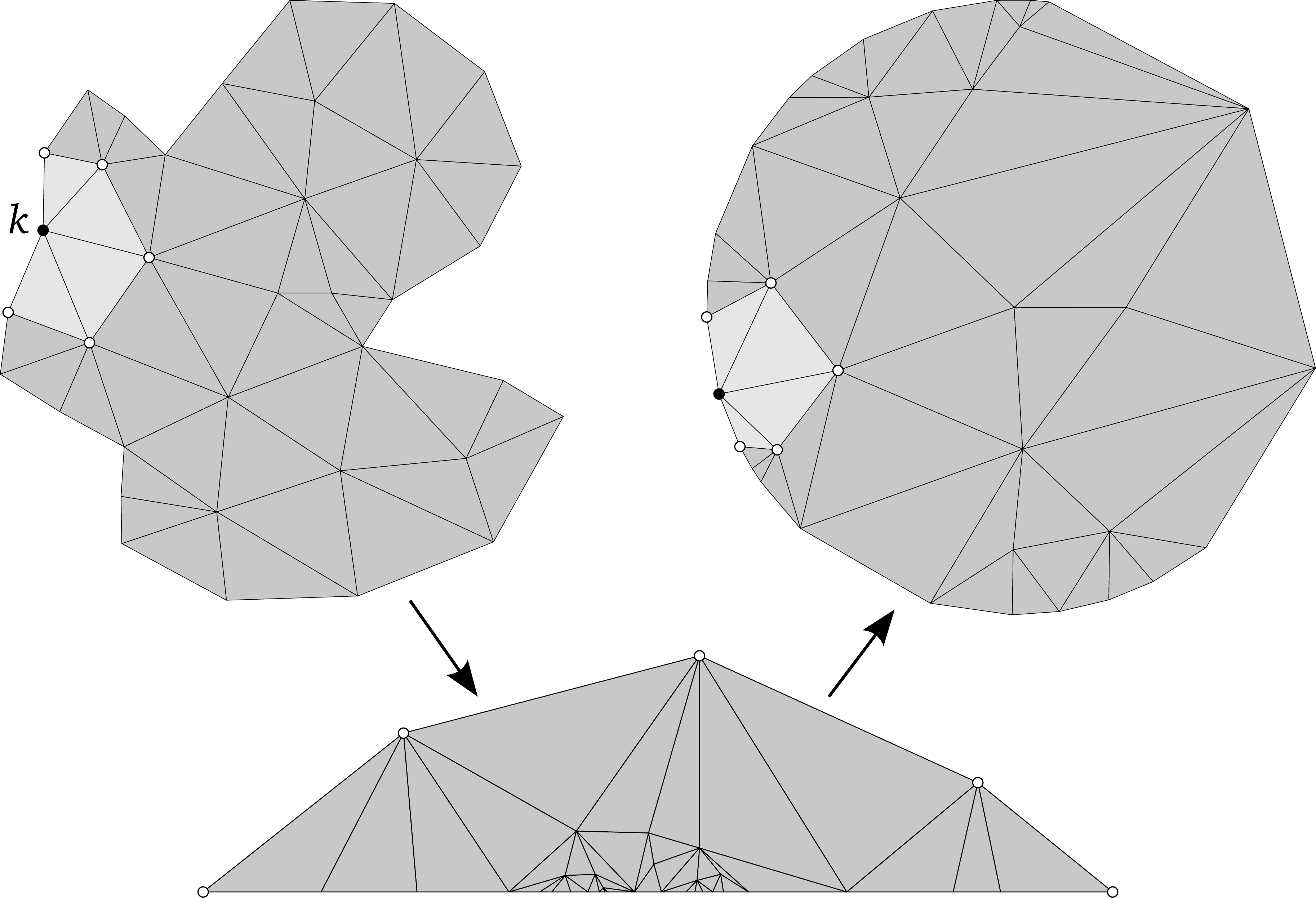}
  \caption{Mapping to the disk.}
  \label{fig:map_to_disk_steps}
\end{figure}

\begin{compactenum}[1.]
\item Choose a boundary vertex $k$ and apply a discrete conformal change of
  metric~\eqref{eq:tilde_ell} so that afterwards all edges incident
  with $k$ have the same length, say $\tilde\ell_{ki}=1$ for all
  neighbors $i$ of $k$. For example, let $e^{u_{i}/2}=\ell_{ki}^{-1}$ if 
  $i$ is a neighbor of $k$ and $1$ otherwise.
\item Let $\T'$ be $\T$ minus the open star of $k$. Suppose this is
  topologically still a closed disk.
\item Solve Problem~\ref{prob:general} for $\T'$, with prescribed
  $\Theta_{i}=2\pi$ for interior vertices of $\T'$, $\Theta_{i}=\pi$
  for boundary vertices of $T'$ that are not neighbors of $k$ in $\T$,
  and prescribed $u_{i}=0$ for the neighbors of $k$ in $\T$. (Suppose
  a solution exists.) The result is a planar triangulation. All
  boundary edges except the neighbors of $k$ in $\T$ are contained in
  one straight line.
\item Apply a M\"obius transformation to the vertices that maps this
  straight line to a circle and the other vertices inside this
  circle. Reinsert $k$ at the image point of $\infty$ under this
  M\"obius transformation. 
\end{compactenum}

\begin{proposition}
  The result of this procedure is a planar triangulation that is
  discretely conformally equivalent to $(\T,\ell)$ and has a boundary
  polygon that is inscribed in a circle.
\end{proposition}

\noindent%
We omit the proof because no new ideas are needed.

\begin{remark}
  Note that for Problem~\ref{prob:general} in step~3 to be solvable,
  the triangulation $\T$ should not have any ears (i.e., triangles on
  the boundary that are attached by one edge only). Prescribing a
  total angle of $\pi$ at boundary vertices forces such triangles to
  degenerate.
\end{remark}

\section{Two variational principles}
\label{sec:variational_principles}

\subsection{The first variational principle}
\label{sec:variational_principle_1}

The system of nonlinear equations described in the previous section
turns out to be variational. Solutions of the conformal mapping
problems correspond to the critical points of the function
$E_{\T,\Theta,\lambda}$ defined as follows. The precise statement of
this first variational principle is
Proposition~\ref{prop:first_variational}.

Let $\T$ be a surface triangulation, $\Theta\in\R^{V}$, and
$\lambda\in\R^{E}$. For now (we will later extend the domain of
definition to $\R^{V}$) define the real valued function
$E_{\T,\Theta,\lambda}(u)$ on the open subset of $\R^{V}$ containing
all $u$ such that $\tilde\ell$ determined by~\eqref{eq:tilde_ell} is a
discrete metric (that is, satisfies the triangle inequalities):
\begin{equation}
  \label{eq:E_of_lambda}
  \begin{split}
    E_{\T,\Theta,\lambda}(u)=\sum_{ijk\in T} \Big(& 
    \tilde\alpha_{ij}^{k}\tilde\lambda_{ij} +
    \tilde\alpha_{jk}^{i}\tilde\lambda_{jk} + 
    \tilde\alpha_{ki}^{j}\tilde\lambda_{ki} +
    2\ML(\tilde\alpha_{ij}^{k}) +
    2\ML(\tilde\alpha_{jk}^{i}) +
    2\ML(\tilde\alpha_{ki}^{j}) \\
    &- \tfrac{\pi}{2}(\tilde\lambda_{ij}+\tilde\lambda_{jk}+\tilde\lambda_{ki})
    \Big)
    +\sum_{i\in V}\Theta_{i} u_{i}\,.
  \end{split}
\end{equation}
The first sum is taken over all triangles, $\tilde\alpha_{jk}^{i}$
denotes the angle at vertex~$i$ in triangle~$ijk$ with side lengths
$\tilde\ell=e^{\tilde\lambda/2}$, 
\begin{equation*}
  \tag{\ref{eq:tilde_lambda}}
  \tilde\lambda_{ij}=\lambda_{ij} + u_{i} + u_{j},
\end{equation*}
and $\ML(x)$ is
Milnor's Lobachevsky function,
\begin{equation}
  \label{eq:ML}
  \ML(x) = -\int_{0}^{x}\log\big|2\sin(t)\big|\,dt.
\end{equation}
\begin{figure}
  \centering
  \includegraphics[width=0.4\linewidth]{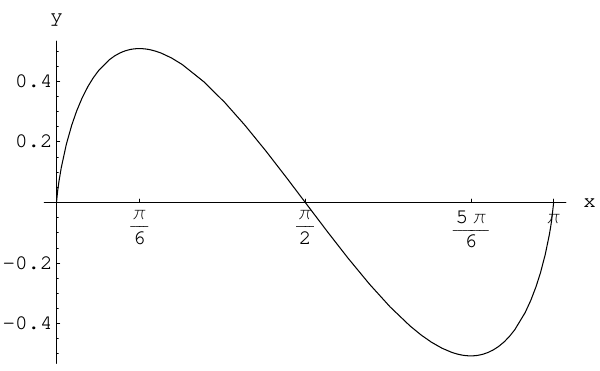}
  \caption{Graph of Milnor's Lobachevsky function, $y=\ML(x)$. The
    function is $\pi$-periodic, odd, and smooth except at $x\in
    \pi\,\Z$, where its tangents are vertical.}
  \label{fig:lobachevskyplot}
\end{figure}%
(Figure~\ref{fig:lobachevskyplot} shows a graph of this function.) The
second sum is taken over all vertices. It is linear in $u$.

\begin{remark}
  The notation $\ML(x)$, using a letter from the Cyrillic alphabet,
  and the name ``Lobachevsky function'' are due to
  Milnor\,\cite{milnor_hyperbolic_1982}\,\cite{milnor_to_1994}. Lobachevsky
  used a slightly different function which is also known as the
  Lobachevsky function and often denoted $L(x)$. To distinguish these
  two functions, we call $\ML(x)$ Milnor's Lobachevsky function. It is
  almost the same as Clausen's integral
  \cite{lewin_polylogarithms_1981},
  $\operatorname{Cl}_{2}(x)=\frac{1}{2}\ML(2x)$.
\end{remark}

\begin{proposition}[First derivative]
  \label{prop:grad_E}
  The partial derivative of $E_{\T,\Theta,\lambda}$ with respect to
  $u_{i}$ is
  \begin{equation*}
    \frac{\partial}{\partial u_{i}}\, E_{\T,\Theta,\lambda} = 
    \Theta_{i}-\sum_{jk:ijk\in T} \tilde\alpha_{jk}^{i}\,,
  \end{equation*}
  where the sum is taken over all angles around vertex $i$.
\end{proposition}
\begin{proof}
  This follows from equation~\eqref{eq:E_with_f} and
  Proposition~\ref{prop:grad_f}.
\end{proof}

\begin{proposition}[First variational principle]
  \label{prop:first_variational}
  Let $\cclass$ be a discrete conformal class with representative
  metric $\ell=e^{\lambda/2}$, and let
  $\tilde\ell=e^{\tilde\lambda/2}$ where $\tilde\lambda$ is the
  function of $u$ defined by equations~\eqref{eq:tilde_lambda}. Then
  \begin{compactitem}[$\bullet$]
  \item $\tilde\ell$ solves Problem~\ref{prob:prescribe_Theta} if and
    only if $u$ is a critical point of $E_{\T,\Theta,\lambda}$,
  \item $\tilde\ell$ solves Problem~\ref{prob:general} if and only if
    $u$ is a critical point of $E_{\T,\Theta,\lambda}$ with fixed
    $u_{i}$ for $i\in V_{0}$. (In this case, the values of
    $\Theta$ for $i\in V_{0}$ are irrelevant.)
  \end{compactitem}
\end{proposition}

\begin{proof}
  This follows immediately from Proposition~\ref{prop:grad_E} because
  $\frac{\partial}{\partial u_{i}}\, E_{\T,\Theta,\lambda} = 0$ is
  equivalent to the angle sum condition~\eqref{eq:angle_sum_equation}.
\end{proof}

\begin{proposition}[Local convexity]
  \label{prop:E_convex}
  The function $E_{\T,\Theta,\lambda}$ is locally convex, that is, its
  second derivative 
  $\sum\frac{\partial^{2}E_{\T,\Theta,\lambda}}{\partial u_{i}\,\partial
  u_{j}}\,du_{i}\,du_{j}$ 
  is a positive
  semidefinite quadratic form. The kernel is $1$-dimensional and
  consists of the constants in $\R^{V}$.
\end{proposition}

\begin{proof}
  This follows from equation~\eqref{eq:E_with_f} and
  Proposition~\ref{prop:f_convex}.
\end{proof}

\begin{proposition}[Extension]
  \label{prop:E_extend}
  The function $E_{\T,\Theta,\lambda}$ can be extended to a convex
  continuously differentiable function on $\R^{V}$.
\end{proposition}

\begin{proof}
  This follows from equation~\eqref{eq:E_with_f} and
  Proposition~\ref{prop:f_extend}.
\end{proof}

In fact, one has an explicit formula for the second derivative of
$E_{\T,\Theta,\lambda}$. This is helpful from the practical point of
view, because it allows one to use more powerful algorithms to
minimize $E_{\T,\Theta,\lambda}$ and thus solve the discrete conformal
mapping problems. It is also interesting from the theoretical point of
view, because the second derivative of $E_{\T,\Theta,\lambda}$ at $u$
is the well known finite-element approximation of the Dirichlet energy
(the cotan-formula) for a triangulation with edge lengths
$\tilde\ell$~\cite{duffin_distributed_1959}
\cite{pinkall_computing_1993}:

\begin{proposition}[Second derivative]
  \label{prop:hess_E}
  The second derivative of $E_{\T,\Theta,\lambda}$ at $u$ is
  \begin{equation*}
    \sum_{i,j\in V}\frac{\partial^{2}E_{\T,\Theta,\lambda}}{\partial
      u_{i}\partial u_{j}} 
    \,du_{i}\,du_{j}
    = \frac{1}{2}\,\sum_{ij\in E}w_{ij}(u)(du_{i}-du_{j})^{2},
  \end{equation*}
  where
  $w_{ij}(u)=\frac{1}{2}(\cot\tilde\alpha_{ij}^{k}+\cot\tilde\alpha_{ij}^{l})$
  if $ij$ is an interior edge with opposite vertices $k$ and $l$ and
  $w_{ij}(u)=\frac{1}{2}\cot\tilde\alpha_{ij}^{k}$ if $ij$ is a
  boundary edge with opposite vertex $k$. This assumes all triangle
  inequalities are satisfied. If triangle inequalities are violated,
  the cotangent terms for the corresponding triangles have to be
  replaced with $0$.
\end{proposition}

\begin{proof}
  This follows from equation~\eqref{eq:E_with_f} and
  Proposition~\ref{prop:hess_f}.
\end{proof}

\subsection{A peculiar triangle function}
\label{sec:function_f}

\begin{figure}
  \centering
  \includegraphics{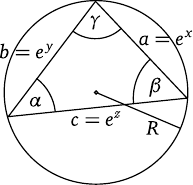}
  \caption{Triangle with sides $a=e^{x}$, $b=e^{y}$, $c=e^{z}$ and
    angles $\alpha$, $\beta$, $\gamma$. The radius of the circumcircle
    is
    $R=\frac{a}{2\sin\alpha}=\frac{b}{2\sin\beta}=\frac{c}{2\sin\gamma}$.}
  \label{fig:triangle_f_xyz}
\end{figure}
Consider the function 
\begin{equation}
  \label{eq:f}
  f(x,y,z) = \alpha x + \beta y + \gamma z + 
  \ML(\alpha) + \ML(\beta) + \ML(\gamma),
\end{equation}
where $\alpha$, $\beta$, and $\gamma$ are the angles in a euclidean
triangle with sides $a=e^{x}$, $b=e^{y}$, and $c=e^{z}$ as shown in
Figure~\ref{fig:triangle_f_xyz}. Such a triangle exists if and only if
the triangle inequalities are satisfied. So $f$ is (for now) only
defined on the set
\begin{equation}
  \label{eq:amoeba}
  \begin{split}
    \mathcal{A} = \big\{(x,y,z)\in\R^{3}\,\big|\,
    -e^{x}+e^{y}+e^{z}&>0,\\
    e^{x}-e^{y}+e^{z}&>0,\\ 
    e^{x}+e^{y}-e^{z}&>0\big\}.
  \end{split}
\end{equation}
Note that the function $f(x,y,z)$ is affine in the $(1,1,1)$-direction:
\begin{equation}
  \label{eq:f_scaling_behavior}
  f(x+h,y+h,z+h)=f(x,y,z) + \pi h.
\end{equation}
This equation remains valid after the extension of $f$ to $\R^{3}$
described in Proposition~\ref{prop:f_extend}. We will use it in
Appendix~\ref{sec:nec_cond_exist} to prove
Proposition~\ref{prop:scale_invariance}.

The function $f(x,y,z)$ is the fundamental building block of
$E_{\T,\Theta,\lambda}(u)$ since
\begin{equation}
\label{eq:E_with_f}
  E_{\T,\Theta,\lambda}(u) = \sum_{ijk\in T}
  \Big(
  2f\big(\tfrac{\tilde\lambda_{ij}}{2}, 
  \tfrac{\tilde\lambda_{jk}}{2}, 
  \tfrac{\tilde\lambda_{ki}}{2}\big)
  - \tfrac{\pi}{2}
  \big(\tilde\lambda_{ij} + \tilde\lambda_{jk} + \tilde\lambda_{ki}\big)
  \Big)
  + \sum_{i\in V}\Theta_i u_{i}\,,
\end{equation}
and Propositions~\ref{prop:grad_E}, \ref{prop:E_convex},
\ref{prop:E_extend}, and~\ref{prop:hess_E} follow from
corresponding statements regarding $f(x,y,z)$.

\begin{proposition}[First derivative]
\label{prop:grad_f}
The partial derivatives of $f$ are 
\begin{equation*}
  \frac{\partial f}{\partial x} = \alpha,\qquad
  \frac{\partial f}{\partial y} = \beta,\qquad
  \frac{\partial f}{\partial z} = \gamma.\qquad
\end{equation*}
\end{proposition}

\begin{proof}
  Using $\ML'(x)=-\log|2\sin(x)|$ we obtain from~\eqref{eq:f} that
  \begin{multline*}
    \frac{\partial f}{\partial x} = \alpha 
    + \big(x - \log(2\sin\alpha)\big)\,\frac{\partial\alpha}{\partial x}\,
    + \big(y - \log(2\sin\beta)\big)\,\frac{\partial\beta}{\partial x}\,
    + \big(z - \log(2\sin\gamma)\big)\,\frac{\partial\gamma}{\partial x}\,.
  \end{multline*}
  Since 
  \begin{equation*}
    x - \log(2\sin\alpha) = 
    y - \log(2\sin\beta) = 
    z - \log(2\sin\gamma) = \log R,
  \end{equation*}
  where $R$ is the radius of the circumcircle, and since
  \begin{equation*}
    \frac{\partial\alpha}{\partial x} + 
    \frac{\partial\beta}{\partial x} + 
    \frac{\partial\gamma}{\partial x} = 0
  \end{equation*}
  (because $\alpha+\beta+\gamma=\pi$), this implies $\frac{\partial
    f}{\partial x} = \alpha$.
\end{proof}

\begin{remark}
  All closed one-forms of the form
  $\sum_{i=1}^{3}f(\alpha_{i})\,dg(\ell_{i})$, where $\ell_{1}$,
  $\ell_{2}$, $\ell_{3}$, and $\alpha_{1}$, $\alpha_{2}$, $\alpha_{3}$
  are the sides and angles of a (euclidean, hyperbolic, or spherical)
  triangle, have been classified by Luo~\cite{luo_rigidity_2010}, see
  also~\cite{dai_variational_2008}. In the euclidean case, they are
  the one-forms
  $w_{s}=\sum_{i=1}^{3}(\int^{\alpha_{i}}\sin^{s}t\,dt)\,d\ell_{i}/\ell_{i}^{s+1}$.
  Thus, the function $f$ is the integral of $w_{0}$.
\end{remark}

\begin{proposition}[Second derivative]
  \label{prop:hess_f}
  The second derivative of $f$ is
  \begin{multline}
    \label{eq:hess_f}
    \Big(
    \begin{smallmatrix}
      dx \\ dy \\ dz
    \end{smallmatrix}
    \Big)^{T}
    \left(
    \begin{smallmatrix}
      \frac{\partial^{2} f}{\partial x^{2}}       &
      \frac{\partial^{2} f}{\partial x\partial y} &
      \frac{\partial^{2} f}{\partial x\partial z} \\
      \frac{\partial^{2} f}{\partial y\partial x} &
      \frac{\partial^{2} f}{\partial y^{2}}       &
      \frac{\partial^{2} f}{\partial y\partial z} \\
      \frac{\partial^{2} f}{\partial z\partial x} &
      \frac{\partial^{2} f}{\partial z\partial y} &
      \frac{\partial^{2} f}{\partial z^{2}}
    \end{smallmatrix}
    \right)
    \Big(
    \begin{smallmatrix}
      dx\\dy\\dz
    \end{smallmatrix}
    \Big)
    =\cot\alpha\,(dy-dz)^{2}+\cot\beta\,(dz-dx)^{2}\\[-3ex]
    +\cot\gamma\,(dx-dy)^{2}
  \end{multline}
\end{proposition}

\begin{proof}
  By Proposition~\ref{prop:grad_f},
  \begin{equation*}
    \Big(
    \begin{smallmatrix}
      dx \\ dy \\ dz
    \end{smallmatrix}
    \Big)^{T}
    \left(
    \begin{smallmatrix}
      \frac{\partial^{2} f}{\partial x^{2}}       &
      \frac{\partial^{2} f}{\partial x\partial y} &
      \frac{\partial^{2} f}{\partial x\partial z} \\
      \frac{\partial^{2} f}{\partial y\partial x} &
      \frac{\partial^{2} f}{\partial y^{2}}       &
      \frac{\partial^{2} f}{\partial y\partial z} \\
      \frac{\partial^{2} f}{\partial z\partial x} &
      \frac{\partial^{2} f}{\partial z\partial y} &
      \frac{\partial^{2} f}{\partial z^{2}}
    \end{smallmatrix}
    \right)
    = \big( d\alpha \quad d\beta \quad d\gamma \big),
  \end{equation*}
  so the left-hand side of equation~\eqref{eq:hess_f} equals
  \begin{equation*}
    d\alpha\,dx+d\beta\,dy+d\gamma\,dz.  
  \end{equation*}
  We will show that
  \begin{equation}
    \label{eq:dalpha}
    d\alpha = \cot\gamma\,(dx - dy) + \cot\beta\,(dx - dz)\,.
  \end{equation}
  This and the analogous equations for $d\beta$ and $d\gamma$ imply
  \begin{equation*}
    d\alpha\,dx + d\beta\,dy + d\gamma\,dz =
    \cot\alpha\,(dy-dz)^{2} + \cot\beta\,(dz-dx)^{2} +
    \cot\gamma\,(dx-dy)^{2},
  \end{equation*}
  and hence equation~\eqref{eq:hess_f}.

  To derive equation~\eqref{eq:dalpha}, differentiate the cosine rule
  \begin{equation*}
    2bc\cos\alpha = b^{2} + c^{2} - a^{2}
  \end{equation*}
  to get
  \begin{equation*}
    -2bc\sin\alpha\,d\alpha + 2bc\cos\alpha\big(dy+dz)
    = 2b^{2}\,dy + 2c^{2}\,dz - 2a^{2}\,dx.
  \end{equation*}
  Apply the cosine rule three more times to get
  \begin{equation*}
    \begin{split}
      2bc\sin\alpha\,d\alpha 
      &= (b^{2}-c^{2}+a^{2})\,(dx-dy) +
      (-b^{2}+c^{2}+a^{2})\,(dx-dz) \\
      &=2ab\cos\gamma\,(dx-dy)+2ac\cos\beta\,(dx-dz)\,.
    \end{split}
  \end{equation*}
  Divide through by $2bc\sin\alpha$ and apply the sine rule to obtain
  equation~\eqref{eq:dalpha}.
\end{proof}

\begin{proposition}[Local convexity]
  \label{prop:f_convex}
  The function $f$ is locally convex, that is, the second
  derivative~\eqref{eq:hess_f} is a positive semidefinite quadratic
  form. Its kernel is one-dimensional and spanned by $(1,1,1)\in\R^{3}$.
\end{proposition}

\begin{proof}
  Writing $(dy-dz)$ as $((dy-dx)-(dx-dz))$ we obtain
  \begin{multline*}
    \cot\alpha\,(dy-dz)^{2} + \cot\beta\,(dz-dx)^{2} +
    \cot\gamma\,(dx-dy)^{2} \\
    = (\cot\alpha+\cot\beta)(dx-dz)^{2}
    + (\cot\alpha+\cot\gamma)(dx-dy)^{2}\\
    - 2\cot\alpha(dx-dy)(dx-dz)
  \end{multline*}
  Thus, in terms of $(dx-dz)$ and $(dx-dy)$, the matrix of this
  quadratic form is 
  \begin{equation*}
    M = 
    \begin{pmatrix}
      \cot\alpha+\cot\beta & -\cot\alpha \\
      -\cot\alpha          & \cot\alpha+\cot\gamma
    \end{pmatrix}
  \end{equation*}
  We proceed as in~\cite[Section~2]{rivin_euclidean_1994}. Using
  $\alpha+\beta+\gamma=\pi$, we obtain
  \begin{equation*}
    M=
    \frac{1}{\sin\alpha\sin\beta\sin\gamma}
    \begin{pmatrix}
      \sin^{2}\gamma & -\cos\alpha\sin\beta\sin\gamma \\
      -\cos\alpha\sin\beta\sin\gamma & \sin^{2}\beta
    \end{pmatrix}
  \end{equation*}
  and $\det M=1$. Since $M_{11}>0$ and $\det M>0$, $M$ is positive
  definite. The claim about the second derivative of $f$ follows.
\end{proof}

\begin{proposition}[Extension]
  \label{prop:f_extend}
  Extend the definition of $f$ from $\mathcal A$ to $\R^{3}$ as follows.
  Define $f(x,y,z)$ by equation~\eqref{eq:f} for all $(x, y,
  z)\in\R^{3}$, where for $(x,y,z)\not\in\mathcal A$ the angles
  $\alpha$, $\beta$, and $\gamma$ are defined to be $\pi$ for the
  angle opposite the side that is too long and $0$ for the other
  two. The so extended function $f:\R^{3}\rightarrow\R$ is
  continuously differentiable and convex.
\end{proposition}

\begin{proof}
  The so-defined functions $\alpha$, $\beta$, $\gamma$ are continuous
  on $\R^{3}$. This implies the continuity of $f$ and, together with
  Proposition~\ref{prop:grad_f}, the continuity of its first
  derivative. Since $f$ is locally convex in $\mathcal A$
  (Proposition~\ref{prop:f_convex}) and linear outside, it is convex.
\end{proof}

Figure~\ref{fig:f_contour_plot} shows contour lines of the extended
function $f(x,y,z)$ in the plane $z=0$ and its graph.
\begin{figure}
  \hfill
  \includegraphics[width=0.3\textwidth]{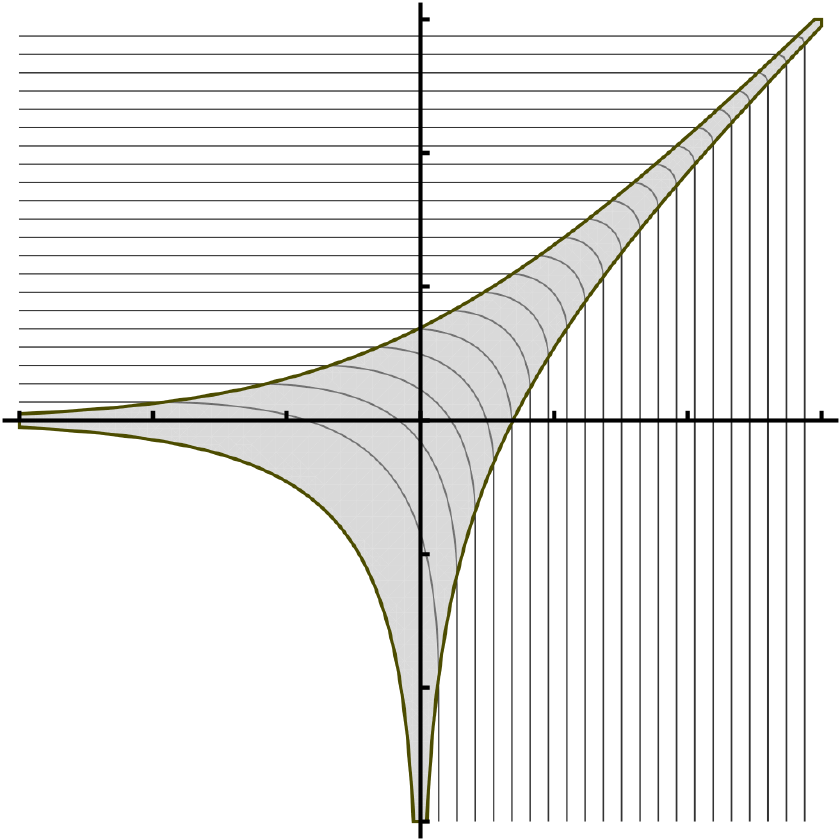}
  \hfill
  \includegraphics[width=0.4\textwidth]{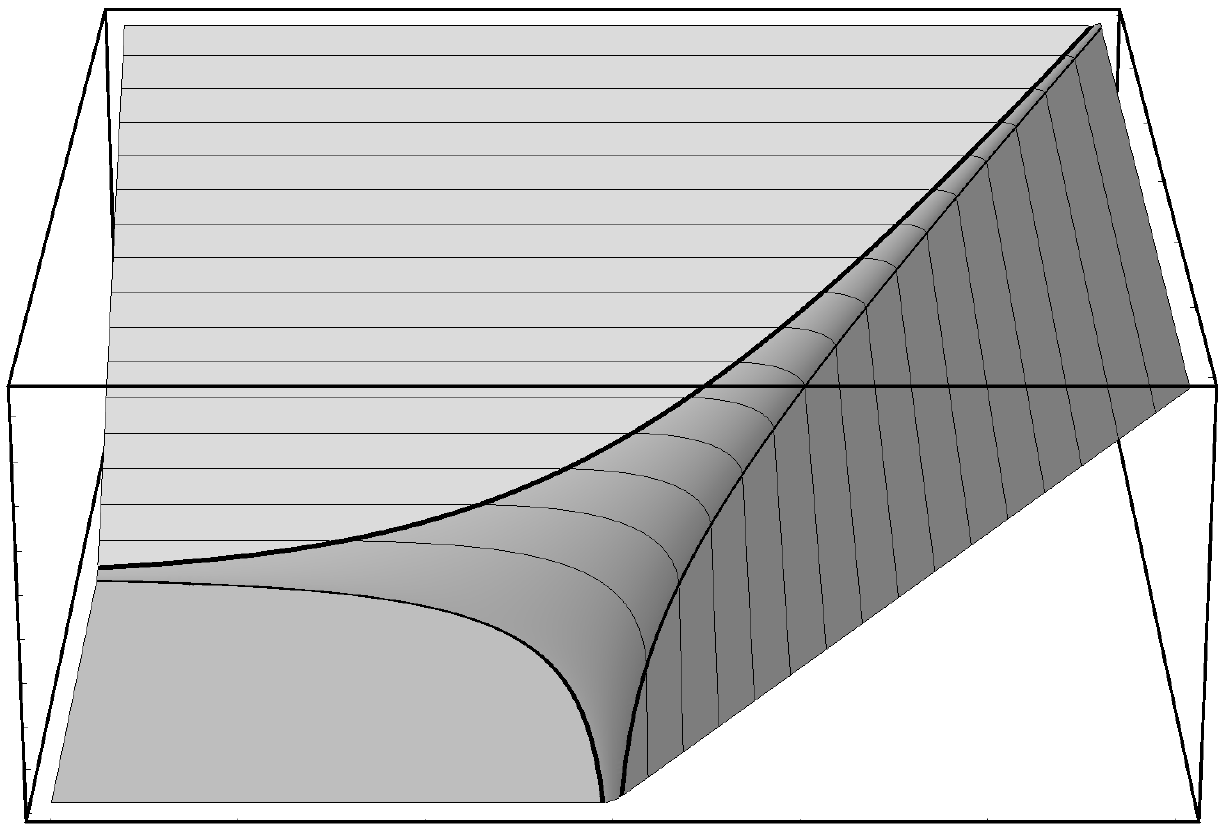}
  \hspace*{\fill}
  \caption{\emph{Left:} Contour plot of $(x,y)\mapsto f(x,y,0)$. The
    intersection of the domain $\mathcal A$ (see
    equation~\eqref{eq:amoeba}) with the $xy$-plane is
    shaded. \emph{Right:} Graph of the same function, also showing
    contour lines.}
  \label{fig:f_contour_plot}
\end{figure}
We will need the following estimate in the proof of
Proposition~\ref{prop:E_tends_to_infty}.

\begin{proposition}[Estimate]
  \label{prop:estimate}
  $f(x,y,z) \geq \pi\,\max\{x,y,z\}$
\end{proposition}

\begin{proof}
  The inequality follows from Proposition~\ref{prop:f_extend} and the
  following two observations. First, the condition of convexity,
  \begin{align*}
    f\big((1-t)p_{1}+tp_{2})\leq(1-t)f(p_{1})+tf(p_{2})\quad&\text{if}\quad 0\leq t\leq 1,\\
    \intertext{is equivalent to}
    f\big((1-t)p_{1}+tp_{2})\geq(1-t)f(p_{1})+tf(p_{2})\quad&\text{if}\quad t\leq 0\quad\text{or}\quad t\geq 1.
  \end{align*}
  Second, for fixed $y$ and $z$, if $x$ is greater than some constant,
  then $f(x,y,z)=\pi x$. 

  Together, they imply $f(x,y,z)\geq\pi x$. Equally,
  $f(x,y,z)\geq\pi y$ and $f(x,y,z)\geq\pi z$.
\end{proof}

\begin{remark}[Amoebas and Ronkin functions]
  In fact, $\mathcal A$ is an amoeba, the extended $f$ is a Ronkin
  function, and the convexity of $f$ that we have proved by elementary
  means follows also from a general theorem of Passare and
  Rullg{\aa}rd~\cite{passare_amoebas_2004}, which says that a Ronkin
  function is convex. Amoebas were introduced by
  Gelfand, Kapranov and
  Zelevinsky~\cite{gelfand_discriminants_1994}. The amoeba $\mathcal
  A_{p}$ of a complex polynomial $p(z_{1},\ldots,z_{n})$ with $n$
  indeterminates is defined as the domain in $\R^{n}$ that is the
  image of the set of zeros of $p$ under the map
  $(z_{1},\ldots,z_{n})\mapsto (\log|z_{1}|,\ldots,\log|z_{n}|)$. So
  the domain $\mathcal A$ defined by equation~\eqref{eq:amoeba} is the
  amoeba of the linear polynomial $z_{1}+z_{2}+z_{3}$. The Ronkin
  function of a polynomial $p$ is defined as the function
  $N_{p}:\R^{n}\rightarrow\R$,
  \begin{equation*}
    N_{p}(x_{1},\ldots,x_{n})=\frac{1}{(2\pi i)^{n}}
    \int_{S^{1}(e^{x_{1}})\times\cdots\times S^{1}(e^{x_{n}})}
    \log\big|p(z_{1},\ldots,z_{n})\big|\,
    \frac{dz_{1}}{z_{1}}\wedge\ldots\wedge\frac{dz_{n}}{z_{n}}\,,
  \end{equation*}
  where $S^{1}(r)$ is the circle in $\C$ around $0$ with radius
  $r$. As it turns out, 
  \begin{equation*}
    \label{eq:f_as_ronkin_function}
    f(x_{1},x_{2},x_{3}) = \pi\,N_{z_{1}+z_{2}+z_{3}}(x_{1},x_{2},x_{3}).
  \end{equation*}
  We will not spoil the reader's fun by presenting a proof here. The
  same Ronkin function also appears in the work of Kenyon, Okounkov,
  and Sheffield on the dimer
  model~\cite{kenyon_planar_2006}~\cite{kenyon_dimers_2006} as the
  Legendre dual of a ``surface tension'' in a variational principle
  governing the limit shape of random surfaces. (See in particular
  Kenyon's survey article on dimers~\cite{kenyon09:_lectur} and
  Mikhalkin's survey article on
  amoebas~\cite{mikhalkin_amoebas_2004}.) Whether or how this is
  related to the variational principles discussed in this paper is
  unclear.
\end{remark}

\subsection{The second variational principle}
\label{sec:variational_principle_2}

The second variational principle has angles as variables. It is based
on the two elementary observations that, first, the sine theorem lets
us express the length-cross-ratios in terms of angles,
\begin{equation}
  \label{eq:lcr_with_angles}
  \lcr_{ij}=
  \frac{\sin(\alpha_{il}^{j})\sin(\alpha_{jk}^{i})}
  {\sin(\alpha_{lj}^{i})\sin(\alpha_{ki}^{j})}
\end{equation}
(see Figure~\ref{fig:lcr_with_angles}),
\begin{figure}
  \centering
  \includegraphics{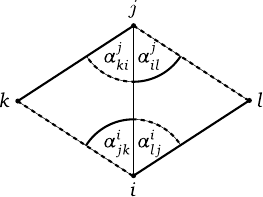}
  \caption{The sine theorem lets us express the length-cross-ratios
    (see Figure~\ref{fig:lcr}) in
    terms angles.}
  \label{fig:lcr_with_angles}
\end{figure}
and that, second, if we know the angles in a euclidean triangulation,
then we can (again using the sine theorem) reconstruct the lengths up
to a global scale factor.

For a triangulated surface $\T$ and $\lambda\in\R^{E}$, define
\begin{gather}
  \notag
  S_{\T,\lambda}:\R^{A}\longrightarrow\R,\\
  \label{eq:widehat_S}
  S_{\T,\lambda}(\alpha)=\sum_{ijk\in T}
  \Big(2V(\alpha_{ij}^{k},\alpha_{jk}^{i},\alpha_{ki}^{j})
  +\alpha_{ij}^{k}\lambda_{ij}
  +\alpha_{jk}^{i}\lambda_{jk}
  +\alpha_{ki}^{j}\lambda_{ki}
  \Big),
\end{gather}
where
\begin{equation}
  \label{eq:V}
  V(\alpha,\beta,\gamma)=
  \ML(\alpha)+\ML(\beta)+\ML(\gamma).
\end{equation}

\begin{remark}
  The function $S_{\T,\lambda}$ is (up to an irrelevant additive
  constant) equal to Rivin's function $\mathcal{V}_{S}$ defined
  in~\cite[Section~7]{rivin_euclidean_1994}. But the variational
  principles considered here
  (Propositions~\ref{prop:second_variational_1}
  and~\ref{prop:second_variational_2}) are different. Rivin's
  variational principle has an additional constraint: Only such
  variations are allowed that fix, for each edge, the sum of opposite
  angles.
\end{remark}

\begin{proposition}[Rivin~\cite{rivin_euclidean_1994}]
  The function $V$ is strictly concave on the domain 
  \begin{equation*}
    \big\{(\alpha,\beta,\gamma)\in(\R_{>0})^{3}
    \,\big|\,\alpha+\beta+\gamma=\pi\big\}.
  \end{equation*}
  So $S_{\T,\lambda}$ is also strictly concave on the domain of
  positive angle assignments that sum to $\pi$ in each triangle.
\end{proposition}

\begin{proposition}[Second variational principle I]
  \label{prop:second_variational_1}
  Let $\cclass$ be a discrete conformal class on $\T$ with
  representative $\ell=e^{\lambda/2}\in\R^{E}$, let $\Theta\in\R^{V}$,
  and define the subset $C_{\Theta}\subseteq\R^{A}$ by
  \begin{equation}
    \begin{split}
      C_{\Theta} = \Big\{\alpha\in\R^{A}\,\Big|\, \alpha>0,\quad%
      &\text{for all $ijk\in T$: }%
      \alpha_{jk}^{i}+\alpha_{ki}^{j}+\alpha_{ij}^{k}=\pi,\\
      &\text{for all $i\in V$: }%
      \sum_{jk:ijk\in T}\alpha_{jk}^{i}=\Theta_{i}
      \Big\}\,.
    \end{split}
  \end{equation}
  Then $\tilde\alpha\in C_{\Theta}$ is the angle function of a solution
  $\tilde\ell=e^{\tilde\lambda/2}$ of
  \emph{Problem~\ref{prob:prescribe_Theta}} if and only if
  $S_{\T,\lambda}(\tilde\alpha)$ is the maximum of the restriction
  $S_{\T,\lambda}\big|_{C_{\Theta}}$\,.
\end{proposition}

\begin{proof}
  Consider the graph $\Gamma$ that is obtained by choosing one point
  in each triangle of $\T$ and connecting it to the vertices of the
  triangle (see Figure~\ref{fig:graph_gamma}, left). The vertex set
  $V_{\Gamma}$ is in one-to-one correspondence with $V_{\T}\cup
  T_{\T}$, and the edge set $E_{\Gamma}$ is in one-to-one
  correspondence with the set of angles $A_{\T}$.
  \begin{figure}
    \hfill%
    \includegraphics{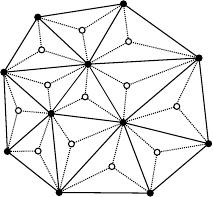}
    \hfill%
    \includegraphics{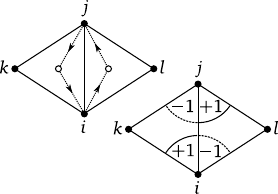}
    \hspace*{\fill}
    \caption{\emph{Left:} Triangulation $\T$ (solid) and the
      corresponding graph $\Gamma$ (dotted). \emph{Right:} The cycle
      of edges of $\Gamma$ (dashed) that corresponds to an interior
      edge $ij\in E_{\T}$, and the corresponding tangent vector
      $v_{ij}\in\R^{A}$ to $C_{\Theta}$.}
    \label{fig:graph_gamma}
  \end{figure}
  The tangent space to $C_{\Theta}\subseteq\R^{A}$, which consists of
  those vectors in $\R^{A}$ that sum to $0$ in each triangle and
  around each vertex, is in one-to-one correspondence with the space
  of closed edge chains of $\Gamma$ with coefficients in~$\R$.
  
  First, assume that $\tilde\alpha$ is a critical point of
  $S_{\T,\lambda}|_{C}$. Suppose $ij \in E_{\T}$ is an interior edge
  and consider the cycle of edges of $\Gamma$ shown on the right in
  Figure~\ref{fig:graph_gamma}. The corresponding tangent vector to
  $C_{\Theta}$ in $\R^{A}$ is
  \begin{equation}
    \label{eq:zero_homological_tangent_vector}
    \frac{\partial}{\partial\alpha_{il}^{j}}
    -\frac{\partial}{\partial\alpha_{lj}^{i}}
    +\frac{\partial}{\partial\alpha_{jk}^{i}}
    -\frac{\partial}{\partial\alpha_{ki}^{j}}\,,
  \end{equation}
  and
  \begin{multline*}
    \Big(
    \tfrac{\partial}{\partial\alpha_{il}^{j}}
    -\tfrac{\partial}{\partial\alpha_{lj}^{i}}
    +\tfrac{\partial}{\partial\alpha_{jk}^{i}}
    -\tfrac{\partial}{\partial\alpha_{ki}^{j}}
    \Big)
    S_{\T,\lambda}(\tilde\alpha) =
    -2\log
    \Bigg(\frac{\sin(\tilde\alpha_{il}^{j})\sin(\tilde\alpha_{jk}^{i})}
    {\sin(\tilde\alpha_{lj}^{i})\sin(\tilde\alpha_{ki}^{j})}\Bigg)\\
    + \lambda_{il}-\lambda_{lj}+\lambda_{jk}-\lambda_{ki}.
  \end{multline*}
  Provided that $\tilde\alpha$ is in fact the system of angles of a
  discrete metric $\tilde\ell=e^{\tilde\lambda/2}$, this implies that
  $\tilde\ell$ and $\ell$ are discretely conformally equivalent. It
  remains to show that $\tilde\alpha$ is indeed the system of angles
  of a discrete metric. Construct such a metric as follows: Pick one
  edge $ij\in E_{\T}$ and choose an arbitrary value for
  $\tilde\lambda_{ij}$. To define $\tilde\lambda_{lk}$ for any other
  edge $lk\in E_{\T}$, connect it to $ij$ by an edge-connected
  sequence of triangles as shown in Figure~\ref{fig:triangle_chain},
  \begin{figure}
    \centering
    \includegraphics[scale=1]{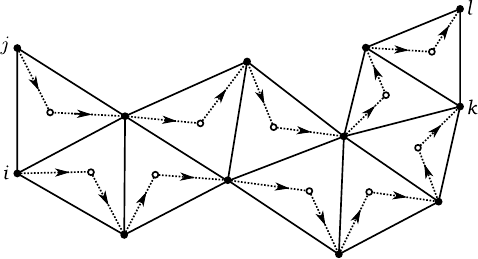}
    \caption{Edge-connected chain of triangles from edge $ij$ to edge
      $kl$ of $\T$. The corresponding chain $\gamma$ of $\Gamma$-edges
      (dotted) consists of the edges of $\Gamma$ opposite the initial
      edge $ij$, the terminal edge $kl$, and the intermediate
      connecting edges of the triangle chain.}
    \label{fig:triangle_chain}
  \end{figure}
  let $\gamma$ be the chain of edges of the graph $\Gamma$ as
  indicated in the figure. Let $w_{\gamma}$ be the corresponding
  vector in $\R^{A}$ and define
  \begin{equation*}
    \tilde\lambda_{kl}=\tilde\lambda_{ij}
    + dS_{\T,\lambda}
    \raisebox{-.6ex}{$\Big|_{\mathrlap{\tilde\alpha}}$}(w_{\gamma})
    + \lambda_{kl} - \lambda_{ij}.
  \end{equation*}
  The value of $\tilde\lambda_{kl}$ obtained in this way is
  independent of the choice of triangle chain: Another triangle chain
  connecting $ij$ to $kl$ leads to an edge-chain $\gamma'$ such that
  $\gamma'-\gamma$ is a closed edge-chain so that
  $w_{\gamma'}-w_{\gamma}\in\R^{A}$ is tangent to
  $C_{\Theta}$. Further, $\tilde\ell=e^{\tilde\lambda/2}$ is a
  discrete metric with angles~$\tilde\alpha$: If $ij$ and $kl$ belong
  to the same triangle (that is, if the triangle chain consists of
  only one triangle) then this follows from the sine rule. The general
  case follows by induction over the length of $\gamma$. So
  $\tilde\ell$ is a solution of Problem~\ref{prob:prescribe_Theta}.

  The converse implication (solution of
  Problem~\ref{prob:prescribe_Theta} implies critical point) follows
  from the fact that the cycle space of $\Gamma$ is spanned by the
  cycles corresponding to interior edges of $\T$ as shown in
  Figure~\ref{fig:graph_gamma} (right) together with the cycles in
  $\Gamma$ corresponding to edge-connected triangle sequences as shown
  in Figure~\ref{fig:triangle_chain} but closed.
\end{proof}

Solutions to the more general Problem~\ref{prob:general} (with
$u|_{V_{0}} = 0$) are also in one-to-one correspondence with critical
points of~$S_{\T,\lambda}$. The only difference is that the
angle sums are not constrained for vertices in $V_{0}$:

\begin{proposition}[Second variational principle II]
  \label{prop:second_variational_2}
  Let $\cclass$ be a discrete conformal class on $\T$ with
  representative $\ell=e^{\lambda/2}\in\R^{E}$, let $V=V_{0}\dot
  \cup V_{1}$ be a partition of $V$, let $\Theta\in\R^{V_{1}}$, and
  define $C_{\Theta}'\subseteq\R^{A}$ by
  \begin{equation}
    \begin{split}
      C_{\Theta}' = \Big\{\alpha\in\R^{A}\,\Big|\, \alpha>0,\quad%
      &\text{for all $ijk\in T$: }%
      \alpha_{jk}^{i}+\alpha_{ki}^{j}+\alpha_{ij}^{k}=\pi,\\
      &\text{for all $i\in V_{1}$: }%
      \sum_{jk:ijk\in T}\alpha_{jk}^{i}=\Theta_{i}
      \Big\}\,.
    \end{split}
  \end{equation}
  Then $\tilde\alpha\in C_{\Theta}'$ is the angle function of a solution of
  \emph{Problem~\ref{prob:general}} with fixed $u|_{V_{0}} = 0$ if and
  only if $S_{\T,\lambda}(\tilde\alpha)$ is the maximum of the
  restriction $S_{\T,\lambda}\big|_{C_{\Theta}'}$\,.
\end{proposition}

We omit the proof because no essential new ideas are necessary beyond
those used in the proof of
Proposition~\ref{prop:second_variational_1}.

\begin{remark}
  It also makes sense to consider critical points of $S_{\T,\lambda}$
  under variations of the type shown in Figure~\ref{fig:graph_gamma}
  alone, disallowing variations corresponding to homologically
  non-trivial cycles in $\Gamma$. These correspond to discretely
  conformally equivalent \emph{similarity structures}, that is,
  ``metrics'' which may have global scaling holonomy.
\end{remark}

\section{The other side of the theory: Interpretation in terms of
  hyperbolic geometry}
\label{sec:ideal_polyhedra}

\subsection{Hyperbolic structure on a euclidean triangulation}
\label{sec:hyp_struct_on_euc_triang}

This section deals with the inverse of a construction of
Penner~\cite{penner_decorated_1987} \cite{epstein_euclidean_1988},
which equips a hyperbolic manifold with cusps with a piecewise
euclidean metric. Here, we construct a natural hyperbolic metric with
cusps on any euclidean triangulation.

Consider a euclidean triangle with its circumcircle. If we interpret
the interior of the circumcircle as a hyperbolic plane in the Klein
model, then the euclidean triangle becomes an ideal hyperbolic
triangle, that is, a hyperbolic triangle with vertices at
infinity. This construction equips any euclidean triangle (minus its
vertices) with a hyperbolic metric. If it is performed on all
triangles of a euclidean triangulation $(\T,\ell)$, then the
hyperbolic metrics induced on the individual triangles fit together so
$\T\setminus V$ is equipped with a hyperbolic metric with cusps at the
vertices. Thus, $\T$ becomes an ideal triangulation of a hyperbolic
surface with cusps.

\begin{remark}
  We will see in Section~\ref{sec:penner_and_shear} that
  $\lambda_{ij}$ and $\log\lcr_{ij}$ are Penner coordinates and shear
  coordinates for this hyperbolic surfaces. It follows that the above
  construction yields the same surface as a construction described (in
  terms of length-cross-ratios) by
  Rivin~\cite[Section~7]{rivin_combinatorial_2003}.
\end{remark}

\begin{theorem}
  \label{thm:hyperbolic_isometry}
  Two euclidean triangulations $(\T,\ell)$ and $(\T,\tilde\ell)$ with
  the same combinatorics are discretely conformally equivalent if and
  only if the hyperbolic metrics with cusps induced by the
  circumcircles are isometric. Discrete conformal maps are
  isometries with respect to the induced hyperbolic metrics.
\end{theorem}

\begin{proof}
  This follows immediately from Theorem~\ref{thm:dconfmap}
  (Section~\ref{sec:discrete_conformal_maps}), because the projective
  circumcircle preserving maps between triangles are precisely the
  hyperbolic isometries.
\end{proof}

\begin{remark}
  Each discrete conformal structure on $\T$ corresponds therefore to a
  point in the classical Teichm\"uller space $\mathcal T_{g,n}$ of a
  punctured surface. This explains the dimensional agreement observed
  in Remark~\ref{rem:dim_discrete_Teich}.
\end{remark}
Theorem~\ref{thm:hyperbolic_isometry} also suggests a way to extend
the concepts of discrete conformal equivalence and discrete conformal
maps to triangulations which are not combinatorially equivalent:

\begin{definition}
  Two euclidean triangulations $(\T,\ell)$ and $(\tilde\T,\tilde\ell)$,
  which need not be combinatorially equivalent, are \emph{discretely
    conformally equivalent} if they are isometric with respect to the
  induced hyperbolic metrics with cusps. The corresponding isometries
  are called \emph{discrete conformal maps}.
\end{definition}

\subsection{Decorated ideal triangles and tetrahedra}
\label{sec:decorated_ideal_triangs_and_tets}

In this section we review some basic facts about ideal triangles and
tetrahedra that will be needed in subsequent sections.

All ideal hyperbolic triangles are congruent with respect to the group
of hyperbolic isometries. A \emph{decorated ideal triangle} is an
ideal hyperbolic triangle together with a choice of horocycles, one
centered at each vertex (see Figure~\ref{fig:ideal_triangle}).
\begin{figure}
  \centering
  \raisebox{7mm}{\includegraphics{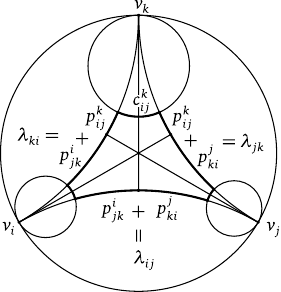}}
  \includegraphics{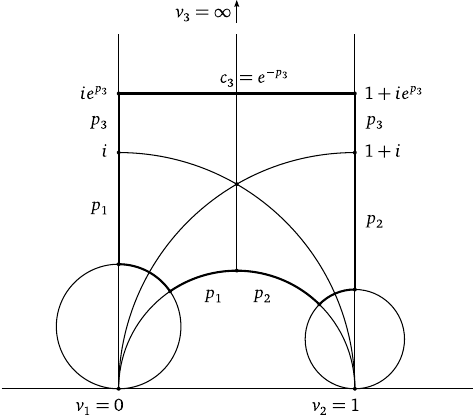}
  \caption{Decorated ideal triangle in the Poincar\'e disk model
    \emph{(left)} and in the half-plane model \emph{(right)}.}
  \label{fig:ideal_triangle}
\end{figure}
We denote by $\lambda_{ij}$ the signed distance between the horocycles
at vertices $i$ and $j$ as measured along the edge $ij$ and taken
negatively if the horocycles intersect. Any triple of real numbers
$(\lambda_{ij}, \lambda_{jk}, \lambda_{ki})\in\R^{3}$ corresponds to
one and only one choice of
horocycles. Figure~\ref{fig:ideal_triangle} shows also the lines
of symmetry of the ideal triangle. (They are its heights as well.) We
denote the signed distances from their base points to the horocycles
by $p_{ij}^{k}$ as shown. Clearly,
\begin{equation*}
    \lambda_{ij}= p_{jk}^{i}+p_{ki}^{j},\qquad
    \lambda_{jk}= p_{ki}^{j}+p_{ij}^{k},\qquad
    \lambda_{ki}= p_{ij}^{k}+p_{jk}^{i},   
\end{equation*}
so
\begin{equation}
  \label{eq:pijk}
  \begin{split}
    p_{ij}^{k}&=\tfrac{1}{2}(-\lambda_{ij}+\lambda_{jk}+\lambda_{ki}),\\
    p_{jk}^{i}&=\tfrac{1}{2}(\lambda_{ij}-\lambda_{jk}+\lambda_{ki}),\\
    p_{ki}^{j}&=\tfrac{1}{2}(\lambda_{ij}+\lambda_{jk}-\lambda_{ki}).
  \end{split}
\end{equation}

\begin{lemma}[Penner~\cite{penner_decorated_1987}]
  \label{lem:c_length}
  The length $c_{ij}^{k}$ of the arc of the horocycle centered at
  $v_{k}$ that is contained in an ideal triangle $v_{i}v_{j}v_{k}$
  as shown in Figure~\ref{fig:ideal_triangle} (left) is
  \begin{equation*}
    c_{ij}^{k}
    =e^{-p_{ij}^{k}}
    =e^{\frac{1}{2}(\lambda_{ij}-\lambda_{jk}-\lambda_{ki})}.
  \end{equation*}
\end{lemma}

\begin{proof}
  See Figure~\ref{fig:ideal_triangle} (right),
  which shows the ideal triangle in the half-plane model. Recall that
  in the half-plane model, the hyperbolic plane is represented by
  $\{z\in\C|\im z > 0\}$ with metric $ds=\tfrac{1}{\im z}|dz|$.
\end{proof}

\begin{remark}
  Together with Proposition~\ref{prop:penner_and_shear} of the next
  section, this provides a geometric interpretation for the auxiliary
  parameters $c^{i}_{jk}$ introduced in
  Section~\ref{sec:product_of_length_cross_ratios_around_a_vertex}.
\end{remark}

Not all ideal tetrahedra are isometric. There is a complex
$1$-parameter family of them, the parameter being the complex
cross-ratio of the vertices in the infinite boundary of hyperbolic
$3$-space. A \emph{decorated ideal tetrahedron} is an ideal hyperbolic
tetrahedron together with a choice of horospheres centered at the
vertices. Figure~\ref{fig:ideal_tet}
\begin{figure}
  \centering
  \includegraphics{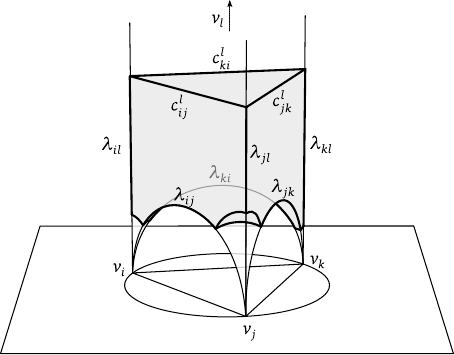}
  \caption{Decorated ideal tetrahedron in the half-space model.}
  \label{fig:ideal_tet}
\end{figure}
shows a decorated ideal tetrahedron, truncated at its horospheres, in
the half-space model. Again, we denote the signed distances between
the horospheres by $\lambda_{ij}$.

The intrinsic geometry of a horosphere in hyperbolic space is
euclidean. So the intersection of the tetrahedron with the horosphere
centered at, say, $v_{l}$ is a euclidean triangle with side lengths
$c_{ij}^{l}$, $c_{jk}^{l}$, $c_{ki}^{l}$ determined by
Lemma~\ref{lem:c_length}. One easily deduces the following lemma.

\begin{lemma}
  \label{lem:cij_triangle_ineq}
  Six real numbers $\lambda_{ij}$, $\lambda_{jk}$, $\lambda_{ki}$,
  $\lambda_{il}$, $\lambda_{jl}$, $\lambda_{kl}$ are the signed
  distances between horospheres of a decorated ideal tetrahedron as
  shown in Figure~\ref{fig:ideal_tet} (which is then unique) if and
  only if $c_{ij}^{l}$, $c_{jk}^{l}$, $c_{ki}^{l}$ determined by
  Lemma~\ref{lem:c_length} satisfy the triangle inequalities.
\end{lemma}

So the six parameters $\lambda$ determine the congruence class of the
ideal tetrahedron ($2$ real parameters) and the choice of horospheres
($4$ parameters). 

Note that the angles of the euclidean triangles in which the
tetrahedron intersects the four horospheres are the dihedral angles of
the tetrahedron. This implies that the dihedral angles sum to
$\pi$ at each vertex, and further, that the dihedral angles at
opposite edges are equal. The space of ideal tetrahedra is therefore
parametrized by three dihedral angles $\alpha_{ij}=\alpha_{kl}$,
$\alpha_{jk}=\alpha_{il}$, $\alpha_{ki}=\alpha_{jl}$ satisfying
$\alpha_{il}+\alpha_{jl}+\alpha_{kl}=\pi$.

\subsection{Penner coordinates and shear coordinates}
\label{sec:penner_and_shear}

In Section~\ref{sec:hyp_struct_on_euc_triang}, we equipped a euclidean
triangulation $(\T,\ell)$ with a hyperbolic cusp metric that turns it
into an ideal hyperbolic triangulation. In this section, we will
identify the logarithmic edge lengths $\lambda$ (see
equation~\eqref{eq:lambda}) with the Penner
coordinates~\cite{penner_decorated_1987} and the logarithmic
length-cross-ratios $\log\lcr$ (see equation~\eqref{eq:lcr}) with the
shear coordinates~\cite{fock_dual_1997}~\cite{thurston_minimal_1998}
for this ideal triangulation. (The
handbook~\cite{papadopoulos_handbook_2007} is a good reference for the
pertinent aspects of Teichm\"uller theory.)

\begin{warning*}
  Our notation differs from Penner's in a potentially confusing
  way. His ``lambda-lengths'' are $\sqrt{2}e^{\lambda/2}=\sqrt{2}\ell$
  in our notation. Our $\lambda$s are the signed hyperbolic distances
  between horocycles.
\end{warning*}

Since the sides of an ideal hyperbolic triangle are complete
geodesics, there is a one-parameter family of ways to glue two sides
together. Penner coordinates and shear coordinates can be seen as two
ways to describe how ideal triangles are glued together along their
edges to form a hyperbolic surface with cusps.

Suppose $\T$ is a triangulated surface and $\lambda\in\R^{E}$. For
each triangle $ijk\in T$, take the decorated ideal triangle with
horocycle distances $\lambda_{ij}$, $\lambda_{jk}$, $\lambda_{ki}$,
and glue them so that the horocycles fit together (see
Figure~\ref{fig:penner_and_shear_coords}, left).
\begin{figure}
  \hfill%
  \includegraphics{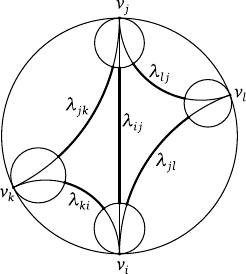}%
  \hfill%
  \includegraphics{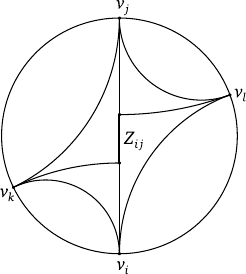}%
  \hspace*{\fill}%
  \caption{Penner coordinates \emph{(left)} and shear coordinates
    \emph{(right)}.}
  \label{fig:penner_and_shear_coords}
\end{figure}
The result is a hyperbolic surface with cusps at the vertices,
together with a particular choice of horocycles centered at the
cusps. In this way, the Penner coordinates $\lambda$ parametrize the
\emph{decorated Teichm\"uller} space, that is, the space of hyperbolic
cusp metrics on a punctured surface (modulo isotopy) with
horocycles centered at the cusps.

The shear coordinates represent another way to prescribe how ideal
triangles are glued, for which no choice of horospheres is
necessary. The shear coordinate $Z$ on an interior edge of an ideal
triangulation is the signed distance of the base points of the heights
from the opposite vertices (see
Figure~\ref{fig:penner_and_shear_coords}, right). The following
relation between Penner coordinates and shear coordinates is well
known.

\begin{lemma}
  \label{lem:penner_and_shear}
  If $\lambda\in\R^{E}$ are the Penner coordinates for an ideal
  triangulation with a particular choice of horocycles, then the shear
  coordinates $Z\in\R^{\Eint}$ are
  \begin{equation*}
    Z_{ij}=\tfrac{1}{2}(\lambda_{il}-\lambda_{lj}+\lambda_{jk}-\lambda_{ki}),
  \end{equation*}
  where $k$ and $l$ are the vertices opposite edge $ij$ as
  in Figure~\ref{fig:penner_and_shear_coords}.
\end{lemma}

\begin{proof}
  The claim follows from $Z_{ij}=p_{ki}^{j}-p_{il}^{j}$ and
  equations~\eqref{eq:pijk}.
\end{proof}

\begin{proposition}
  \label{prop:penner_and_shear}
  Let $(T,\ell)$ be a euclidean triangulation. The shear coordinates
  $Z\in\R^{\Eint}$ for the corresponding ideal triangulation (see
  Section~\ref{sec:hyp_struct_on_euc_triang}) are
  \begin{equation}
    \label{eq:shear_and_lcr}
    Z_{ij}=\log\lcr_{ij}
  \end{equation}
  (see equation~\eqref{eq:lcr}). Thus, for a suitable choice of
  horocycles, the Penner coordinates $\lambda\in\R^{E}$
  are given by equation~\eqref{eq:lambda}.
\end{proposition}

\begin{proof}
  Consider an interior edge $ij\in E$ between triangles $ijk$ and
  $jil$. Without loss of generality, we may assume that the triangles
  have a common circumcircle. For otherwise we can change $\ell$
  discretely conformally so that this holds, and this changes neither
  $\lcr_{ij}$ nor the hyperbolic cusp metric on $T$. We may further
  assume that $ij$ is a diameter of the common circumcircle. For
  otherwise we may apply a projective transformation that maps the
  circle onto itself so that this holds. This is an isometry with
  respect to the hyperbolic metric of the Klein model, and it is a
  discrete conformal map of the quadrilateral formed by the two
  triangles. We arrive at the situation shown in
  Figure~\ref{fig:shear_and_lcr} in the Klein model. 
  \begin{figure}
    \centering
    \includegraphics{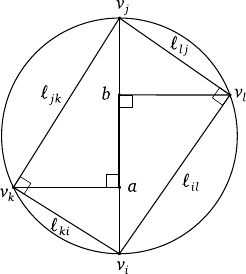}
    \caption{Shear and length-cross-ratio (see proof of
      Proposition~\ref{prop:penner_and_shear}).}
    \label{fig:shear_and_lcr}
  \end{figure}
  The hyperbolic heights are also the euclidean heights, and in the
  hyperbolic metric of the Klein model, the distance between their
  base points $a$ and $b$ is
  \begin{equation*}
    Z_{ij}=\tfrac{1}{2}\log\frac{|av_{j}|\,|bv_{i}|}{|av_{i}|\,|bv_{j}|}\,,
  \end{equation*}
  where $|xy|$ denotes the euclidean distance between $x$ and $y$.
  Since by elementary euclidean geometry 
  \begin{equation*}
    \frac{|av_{j}|}{|av_{i}|}=\frac{{\ell_{jk}}^{2}}{{\ell_{ki}}^{2}}
    \qquad\text{and}\qquad
    \frac{|bv_{i}|}{|bv_{j}|}=\frac{{\ell_{lj}}^{2}}{{\ell_{il}}^{2}}\,,
  \end{equation*}
  this implies
  equation~\eqref{eq:shear_and_lcr}. Now
  Lemma~\ref{lem:penner_and_shear} implies the statement about Penner coordinates.
\end{proof}

\subsection{Ideal hyperbolic polyhedra with prescribed intrinsic
  metric}
\label{sec:ideal_hyp_poly_with_prescribed_intrinsic_metric}

The discrete conformal mapping problems described in
Section~\ref{sec:mapping_problems} are equivalent to problems
involving the polyhedral realization of surfaces with hyperbolic cusp
metrics, like the following.

\begin{problem}
  \label{prob:ideal_polyhedral_realization}
  \emph{\textbf{Given}} an ideal triangulation $\T$ of a punctured
  sphere equipped with a hyperbolic metric with cusps,
  \emph{\textbf{find}} an isometric embedding of $\T$ as ideal
  hyperbolic polyhedron in $H^{3}$. The polyhedron is not required to
  be convex, but it is required that the edges of the polyhedron are
  edges of $\T$.
\end{problem}

\begin{theorem}
  \label{thm:ideal_polyhedral_realization}
  For any vertex $l$ of $\T$,
  Problem~\ref{prob:ideal_polyhedral_realization} has at most one
  solution that is star-shaped with respect to $l$.
\end{theorem}

Problem~\ref{prob:ideal_polyhedral_realization} is equivalent to (a
special case of) Problem~\ref{prob:general}, so
Theorem~\ref{thm:ideal_polyhedral_realization} follows from
Theorem~\ref{thm:dconfmap_unique}. Indeed, to solve
Problem~\ref{prob:ideal_polyhedral_realization} one may proceed as
follows. Let $\lambda\in\R^{E}$ be the Penner coordinates for the
ideal triangulation $\T$, and let $\ell=e^{\lambda/2}$. Choose a
vertex $l$ of $\T$ and let the triangulation $\T'$ be $\T$ minus the
open star of $l$. Solve Problem~\ref{prob:general} for $\T'$,
prescribing $\Theta_{i}=2\pi$ if $i$ is an interior vertex and
$u_{i}=-\lambda_{il}$ if $i$ is a boundary vertex. Suppose a
solution $u\in\R^{V'}$ exists. This leads to a flat triangulation
$(\T',\tilde\ell)$. Suppose further that $(\T',\tilde\ell)$ does not
overlap with itself when developed in the plane. For each triangle
$ijk$ of $\T'$, construct the decorated ideal tetrahedron (see
Section~\ref{sec:decorated_ideal_triangs_and_tets}) with
horosphere-distances $\lambda_{ij}$, $\lambda_{jk}$, $\lambda_{ki}$,
and $-u_{i}$, $-u_{j}$, $-u_{k}$ as shown in
Figure~\ref{fig:ideal_tet_with_u}.
\begin{figure}
  \centering
  \includegraphics{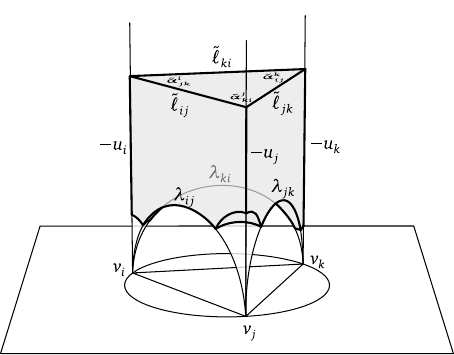}
  \caption{Discretely conformally flattening a euclidean triangulation
    is equivalent to constructing an ideal polyhedron with prescribed
    hyperbolic cusp metric.}
  \label{fig:ideal_tet_with_u}
\end{figure}
They exist by Lemma~\ref{lem:cij_triangle_ineq}, because by
Lemma~\ref{lem:c_length}, the intersection of the ideal tetrahedron
with the horosphere centered at the vertex opposite $ijk$ is the
euclidean triangle with side lengths $\tilde\ell_{ij}$,
$\tilde\ell_{jk}$, $\tilde\ell_{ki}$ as shown in the figure. Hence,
all these ideal tetrahedra fit together to form a solution of
Problem~\ref{prob:ideal_polyhedral_realization} that is star-shaped
with respect to $l$.

Conversely, any solution of
Problem~\ref{prob:ideal_polyhedral_realization} that is star-shaped
with respect to $l$ yields a solution without self-overlap of the
corresponding instance of Problem~\ref{prob:general}.

\begin{remark}
  Note the similarity with the procedure for mapping to a sphere
  described in Section~\ref{sec:sphere}. 
\end{remark}

Numerous variations of Problem~\ref{prob:ideal_polyhedral_realization}
can be treated in similar fashion. We mention only the following.

\begin{problem}
  \label{prob:ideal_torus_realization}
  \emph{\textbf{Given}} an ideal triangulation $T$ of a punctured
  torus equipped with a hyperbolic metric with cusps,
  \emph{\textbf{find}} an isometric embedding of the universal cover
  of $T$ as an ideal polyhedron in $H^{3}$ that is symmetric with
  respect to an action of the fundamental group of $\T$ by parabolic
  isometries. The polyhedron is not required to be convex, but it is
  required that the polyhedron is star-shaped with respect to the
  ideal fixed point of the parabolic isometries and that the edges of
  the polyhedron are edges of $T$.
\end{problem}

\begin{theorem}
  \label{thm:ideal_torus_realization}
  If Problem~\ref{prob:ideal_torus_realization} has a solution,
  it is unique.
\end{theorem}

\subsection{The variational principles and hyperbolic volume}
\label{sec:hyperbolic_volume}

The connection with hyperbolic polyhedra elucidates the nature and
origin of the variational principles for discrete conformal maps
(Proposition~\ref{prop:first_variational}
and~Propositions~\ref{prop:second_variational_1},
\ref{prop:second_variational_2}). In this section, we will indicate
how to derive these variational principles from Milnor's equation for
the volume of an ideal tetrahedron and Schl\"afli's formula.

Milnor~\cite{milnor_hyperbolic_1982}\,\cite{milnor_to_1994} showed
that the volume of an ideal tetrahedron with dihedral angles $\alpha$,
$\beta$, $\gamma$ is $V(\alpha, \beta, \gamma)$ as defined by
equation~\eqref{eq:V}. Schl\"afli's differential volume formula
(more precisely, Milnor's generalization which allows for ideal
vertices~\cite{milnor_schlfli_1994}) says that its derivative is
\begin{equation}
  \label{eq:dV}
  dV = -\frac{1}{2}\sum \lambda_{ij}\,d\alpha_{ij},
\end{equation}
where the sum is taken over the six edges $ij$, $\lambda_{ij}$ is the
signed distance between horospheres centered at the vertices $i$ and
$j$, and $\alpha_{ij}$ is the interior dihedral angle. (The choice of
horospheres does not matter because the dihedral angle sum at a vertex
is constant, see Section~\ref{sec:decorated_ideal_triangs_and_tets}.)

Using the correspondence between ideal tetrahedra and euclidean
triangles, the volume function $V$ can be reinterpreted as a function
of the angles of a euclidean triangle, whose derivatives
$(\frac{\partial}{\partial\alpha}-\frac{\partial}{\partial\beta})V$,
etc., are logarithmic ratios of the sides. This is the essential
property of $V$ used in the second variational principle
(Propositions~\ref{prop:second_variational_1},
\ref{prop:second_variational_2}).

Now define
\begin{equation}
  \label{eq:hat_V}
  \widehat
  V(\lambda_{12},\lambda_{23},\lambda_{31},\lambda_{14},\lambda_{24},\lambda_{34})
  =\tfrac{1}{2}\sum_{ij}\alpha_{ij}\lambda_{ij}
  + V(\alpha_{14},\alpha_{24},\alpha_{34}),
\end{equation}
where the dihedral angles $\alpha_{12}=\alpha_{34}$,
$\alpha_{23}=\alpha_{14}$, $\alpha_{31}=\alpha_{24}$ of the decorated
tetrahedron are considered as functions of the $\lambda_{ij}$. (They
are the angles in a euclidean triangle with side lengths
$e^{(\lambda_{12}-\lambda_{14}-\lambda_{24})/2}$,
$e^{(\lambda_{23}-\lambda_{24}-\lambda_{34})/2}$,
$e^{(\lambda_{31}-\lambda_{34}-\lambda_{14})/2}$, see
Section~\ref{sec:decorated_ideal_triangs_and_tets}.) Then, by
equation~\eqref{eq:dV},
\begin{equation}
  \label{eq:dVhat}
  d\widehat V=\frac{1}{2}\sum\alpha_{ij}\,d\lambda_{ij}.
\end{equation}
This implies Proposition~\ref{prop:grad_E} on the partial derivatives
of $E_{\T,\Theta,\lambda}$, and therefore
Proposition~\ref{prop:first_variational} (the first variational
principle), because using~\eqref{eq:hat_V}
(and~\eqref{eq:tilde_lambda}) we can rewrite
equation~\eqref{eq:E_of_lambda} as
\begin{equation}
  \label{eq:E_of_lambda_with_V}
  \begin{split}
    E_{\T,\Theta,\lambda}(u)= \sum_{ijk\in T} 2\hat
    V(\lambda_{ij},\lambda_{jk},\lambda_{ki},-u_{i},-u_{j},-u_{k})&\\
    -\sum_{ij\in E}\Phi_{ij}\lambda_{ij} +\sum_{i\in
      V}\Theta_{i}u_{i},&
  \end{split}
\end{equation}
where 
\begin{equation}
  \label{eq:theta_ij_for_dconf}
  \Phi_{ij}=
  \begin{cases}
    \pi,\qquad\text{if $ij$ is an interior edge},\\
    \tfrac{\pi}{2},\qquad\text{if $ij$ is a boundary edge}.
  \end{cases}
\end{equation}
(See also Figure~\ref{fig:ideal_tet_with_u}.)

\section{The discrete conformal equivalence of hyperbolic triangulations}
\label{sec:dconf_hyperbolic}

\subsection{Definition and variational principle}
\label{sec:d_conf_equiv_hyp}

In Section~\ref{sec:hyperbolic_volume} we derived the first
variational principle for discrete conformal maps from Milnor's
equation for the volume of an ideal tetrahedron and Schl\"afli's
formula. A straightforward modification of this derivation leads to a
companion theory of discrete conformality for \emph{hyperbolic}
triangulations. This makes it possible, for example, to construct
discretely conformal uniformizations of higher genus surfaces as shown
in Figure~\ref{fig:pretzel}.
\begin{figure}
  \hfill%
  \raisebox{.1\textwidth}{\includegraphics[width=.5\textwidth]{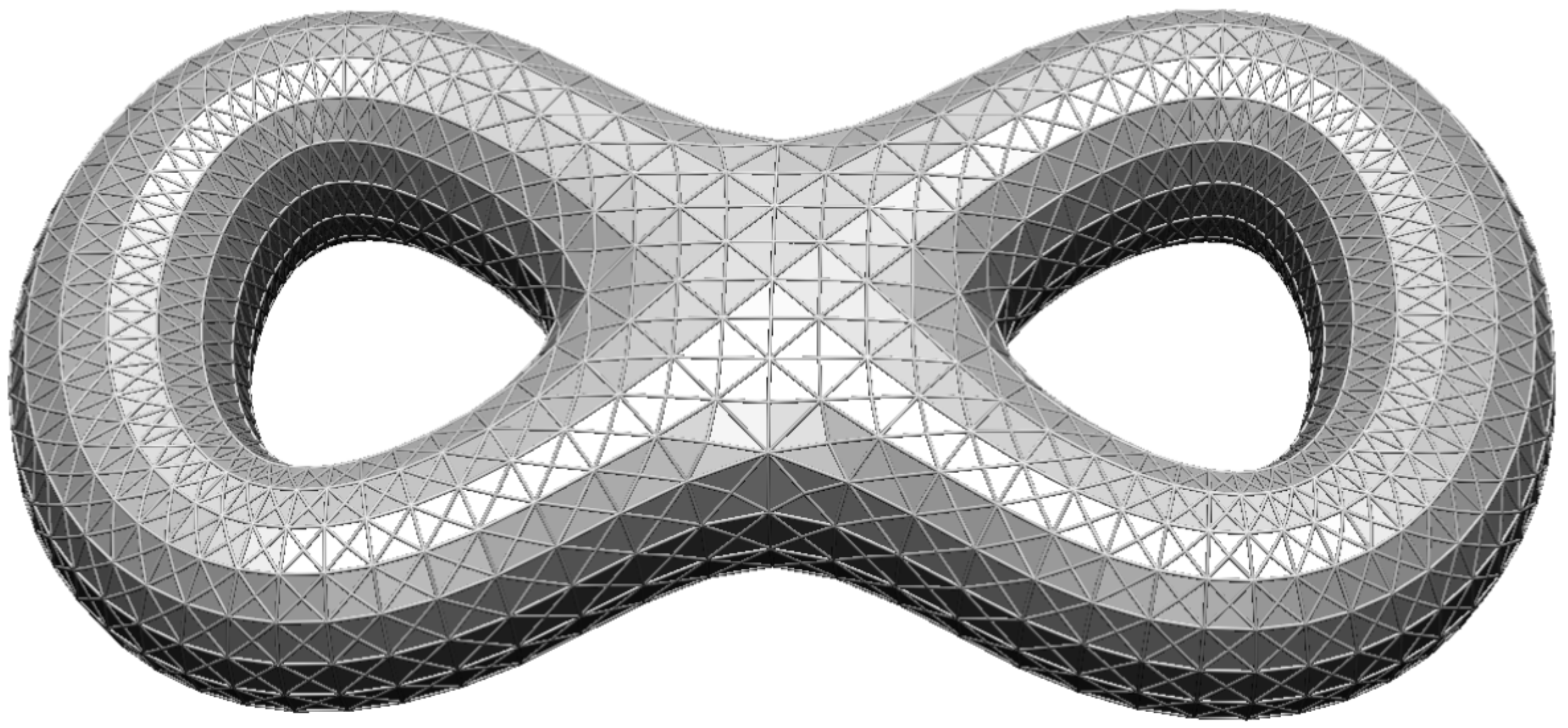}}
  \hfill%
  \includegraphics[width=.45\textwidth]{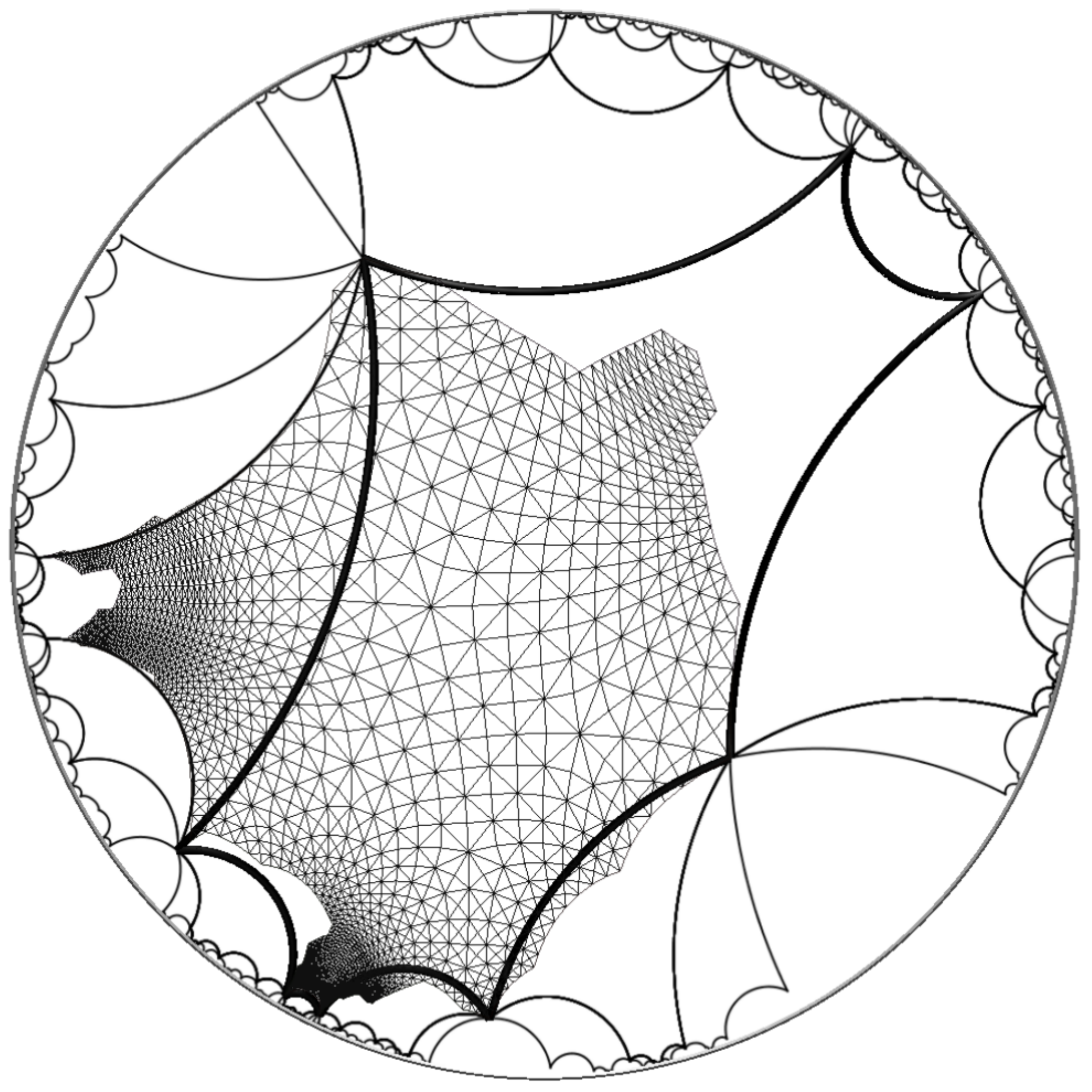}
  \hspace*{\fill}
  \caption{Discretely conformal uniformization of a genus-two surface.}
  \label{fig:pretzel}
\end{figure}
We will present the basic theory in this section, and show how to
derive it by hyperbolic volume considerations in the next. 

Suppose $\T$ is a surface triangulation and $\ell\in\R_{>0}^{E}$ is a
discrete metric, that is, a real valued function on the set of edges
that satisfies all triangle inequalities. Then there is not only a
euclidean triangulation $(\T,\ell)$ with these edge lengths. One can
equally construct hyperbolic triangles $ijk$ with hyperbolic side
lengths $\ell_{ij}$, $\ell_{jk}$, $\ell_{ki}$ and glue them
together. The result is a \emph{hyperbolic triangulation} which we
denote by $(\T,\ell)_{h}$.

\begin{definition}
  \label{def:d_conf_hyp}
  Two combinatorially equivalent hyperbolic triangulations,
  $(T,\ell)_{h}$ and $(\T,\tilde\ell)_{h}$, are \emph{discretely
    conformally equivalent} if the discrete metrics $\ell$ and
  $\tilde\ell$ are related by
  \begin{equation}
    \label{eq:tilde_ell_hyp}
    \sinh\frac{\tilde\ell_{ij}}{2}=
    e^{\frac{1}{2}(u_{i}+u_{j})}\,\sinh\frac{\ell_{ij}}{2}
  \end{equation}
  for some function $u:V\rightarrow\R$.
\end{definition}

Thus, in the hyperbolic version of the theory
equation~\eqref{eq:lambda} is replaced by
\begin{equation}
  \label{eq:lambda_hyp}
  \lambda = 2\log\sinh\frac{\ell}{2},
\end{equation}
so that in terms of $\lambda$ and $\tilde\lambda$, the
relation~\eqref{eq:tilde_ell_hyp} is again equivalent
to~\eqref{eq:tilde_lambda}.  The role of $E_{\T,\Theta,\lambda}(u)$ is
played by the function
\begin{equation}
  \label{eq:Ehyp}
  \Ehyp_{\T,\Theta,\lambda}(u) = \sum_{ijk\in T}
  2\Vhathyp(\lambda_{ij},\lambda_{jk},\lambda_{ki},-u_i,-u_j,-u_k,)
  +\sum_{i\in V}\Theta_iu_i\,,
\end{equation}
where
\begin{multline}
  \label{eq:Vhathyp}
  2\Vhathyp(\lambda_{12},\lambda_{23},\lambda_{31}, 
  \lambda_{1}, \lambda_{2}, \lambda_{3}) =\\
  \alpha_{1}\lambda_{1}
  +\alpha_{2}\lambda_{2}
  +\alpha_{3}\lambda_{3}
  +\alpha_{12}\lambda_{12}
  +\alpha_{23}\lambda_{23}
  +\alpha_{31}\lambda_{31}\\
  +\ML(\alpha_{1})
  +\ML(\alpha_{2})
  +\ML(\alpha_{3})
  +\ML(\alpha_{12})
  +\ML(\alpha_{23})
  +\ML(\alpha_{31})\\
  +\ML\Big(
  \tfrac{1}{2} \big(\pi-\alpha_{1}-\alpha_{2}-\alpha_{3}\big)
  \Big),
\end{multline}
and $\alpha_{1}$, $\alpha_{2}$, $\alpha_{3}$ are the angles in a
hyperbolic triangle with side lengths 
\begin{equation}
  \label{eq:tilde_ell_with_arsinh}
  \begin{split}
    \tilde\ell_{23}=
    2\arsinh
    \big(e^{\frac{1}{2}(\lambda_{23}-\lambda_{2}-\lambda_{3})}\big),\\
    \tilde\ell_{31}=
    2\arsinh 
    \big(e^{\frac{1}{2}(\lambda_{31}-\lambda_{3}-\lambda_{1})}\big),\\
    \tilde\ell_{12}=
    2\arsinh 
    \big(e^{\frac{1}{2}(\lambda_{12}-\lambda_{1}-\lambda_{2})}\big),
  \end{split}
\end{equation}
and
\begin{equation}
  \label{eq:alpha_ij}
  \begin{split}
    \alpha_{23} & = \tfrac{1}{2} \big(
    \pi + \alpha_{1} - \alpha_{2} - \alpha_{3} \big),\\ 
    \alpha_{31} & = \tfrac{1}{2} \big(
    \pi - \alpha_{1} + \alpha_{2} - \alpha_{3} \big),\\
    \alpha_{12} & = \tfrac{1}{2} \big(
    \pi - \alpha_{1} - \alpha_{2} + \alpha_{3} \big).
  \end{split}
\end{equation}
Thus, $\Vhathyp$ is defined only on the domain where
$\tilde\ell_{12}$, $\tilde\ell_{23}$, $\tilde\ell_{31}$ satisfy the
triangle inequalities. However, exactly as in the case of
$E_{\T,\Theta,\lambda}(u)$, we can extend the domain of definition of
$\Ehyp_{\T,\Theta,\lambda}(u)$ to the whole of $\R^{V}$:

\begin{proposition}
  \label{prop:E_extend_hyp}
  Extend the domain of definition of $\Ehyp_{\T,\Theta,\lambda}(u)$ to
  $\R^{V}$ by declaring the angles in ``broken'' triangles to be
  $0$, $0$, $\pi$, respectively. The resulting function is
  continuously differentiable on $\R^{V}$.
\end{proposition}

\begin{proof}
  See Section~\ref{sec:derive_by_volume}.
\end{proof}

\begin{remark}
  To compute the angles $\alpha$, $\beta$, $\gamma$ in a hyperbolic
  triangle with side lengths $a$, $b$, $c$, one can use, for example,
  the hyperbolic cosine rule or the hyperbolic half-angle formula
  \begin{equation*}
    \tan\Bigg(\frac{\alpha}{2}\Bigg)=
    \sqrt{\frac{
        \sinh\big((a-b+c)/2\big)
        \sinh\big((a+b-c)/2\big)
      }{
        \sinh\big((-a+b+c)/2\big)
        \sinh\big((a+b+c)/2\big)
      }}\;.
  \end{equation*}
\end{remark}

\begin{proposition}
  \label{prop:grad_Ehyp}
  Let $\ell\in\R^{E}$, let $\lambda$ be defined by
  equation~\eqref{eq:lambda_hyp}, and suppose $u\in\R^{V}$ is in the
  domain where $\tilde\ell$ defined by
  equation~\eqref{eq:tilde_ell_hyp} satisfies all triangle
  inequalities. Then the partial derivative of
  $\Ehyp_{\T,\Theta,\lambda}$ with respect to $u_{i}$ is
  \begin{equation*}
    \frac{\partial}{\partial u_{i}}\, \Ehyp_{\T,\Theta,\lambda} = 
    \Theta_{i}-\sum_{jk:ijk\in T} \tilde\alpha_{jk}^{i}\,,
  \end{equation*}
  where $\tilde\alpha$ are the angles in the hyperbolic triangulation
  $(\T,\tilde\ell)_{h}$, and the sum is taken over all angles around vertex $i$.
\end{proposition}

\begin{proof}
  See Section~\ref{sec:derive_by_volume}.
\end{proof}

\begin{proposition}
  \label{prop:E_convex_hyp}
  The function $\Ehyp_{\T,\Theta,\lambda}(u)$ is convex on $\R^{V}$ and
  locally strictly convex in the domain where $\tilde\ell$ defined by
  equation~\eqref{eq:tilde_ell_hyp} satisfies all triangle
  inequalities. 
\end{proposition}

\begin{proof}
  See Section~\ref{sec:derive_by_volume}.
\end{proof}

Consider the discrete conformal mapping problems for hyperbolic
triangulations that are analogous to those for euclidean
triangulations described in
Section~\ref{sec:mapping_problems}. Propositions~\ref{prop:grad_Ehyp}
and~\ref{prop:E_convex_hyp} imply the following hyperbolic version of
Theorem~\ref{thm:dconfmap_unique}.

\begin{theorem}
  If the discrete mapping problems for hyperbolic triangulations have a
  solution, it is unique and can be found by minimizing
  $\Ehyp_{\T,\Theta,\lambda}(u)$.
\end{theorem}

\noindent%
The following relatively simple explicit formula for the
second derivative facilitates the numerical minimization of
$\Ehyp_{\T,\Theta,\lambda}$.

\begin{proposition}
  \label{prop:hess_Ehyp}
  The second derivative of $\Ehyp_{\T,\Theta,\lambda}$ at $u$ is
  \begin{multline}
    \sum_{i,j\in V}\frac{\partial^{2}\Ehyp_{\T,\Theta,\lambda}}{\partial
      u_{i}\partial u_{j}} 
    \,du_{i}\,du_{j}
    = \\
    \frac{1}{2}\,\sum_{ij\in E}w_{ij}(u)\Big(
    (du_{i}-du_{j})^{2}
    +\tanh^{2}\big(\tfrac{\tilde\ell_{ij}}{2}\big)
    (du_{i}+du_{j})^{2}
    \Big),
  \end{multline}
  with $\tilde\ell$ defined by equation~\eqref{eq:tilde_ell_hyp}
  and
  \begin{equation}
    w_{ij}(u)=\frac{1}{2}\Big(
    \cot\big(
    \tfrac{1}{2}(
    \pi-\tilde\alpha_{jk}^{i}-\tilde\alpha_{ki}^{j}+\tilde\alpha_{ij}^{k})
    \big)
    +\cot\big(
    \tfrac{1}{2}(
    \pi-\tilde\alpha_{il}^{j}-\tilde\alpha_{lj}^{i}+\tilde\alpha_{ji}^{l})
    \big)
    \Big)
  \end{equation}
  for interior edges $ij$ with opposite vertices $k$ and $l$ if
  $\tilde\ell$ satisfies the triangle inequalities for $ijk$ and $jil$
  (so that the corresponding angles $\tilde\alpha$ are positive and
  smaller than $\pi$). If $ij$ is a boundary edge, there is only one
  cotangent term. For ``broken'' triangles, replace the three
  corresponding cotangent terms with $0$.
\end{proposition}

\noindent%
We omit the proof, which consists of a lengthy but elementary
calculation.

\begin{remark}
  \label{rem:euclidean_and_hyperbolic}
  When are a euclidean and a hyperbolic triangulation discretely
  conformally equivalent?
  We propose the following definition:
  A euclidean triangulation $(\T,\ell)$ and a hyperbolic triangulation
  $(\T,\tilde\ell)_{h}$ are \emph{discretely conformally equivalent}
  if $\ell$ and $\tilde\ell$ are related by
  \begin{equation}
    \label{eq:tilde_ell_euc_hyp}
    \sinh\frac{\tilde\ell_{ij}}{2}=e^{\frac{1}{2}(u_{i}+u_{j})}\ell_{ij}
  \end{equation}
  for some function $u\in\R^{V}$.

This is based on the following interpretation of
equations~\eqref{eq:lambda_hyp} and~\eqref{eq:tilde_lambda}. Consider
the hyperboloid model of the hyperbolic plane, 
$
  H^{2}=\big\{x\in\R^{2,1}\,\big|\,\langle x,x\rangle=-1,\;x_{3}>0\big\},
$
where $\langle\cdot,\cdot\rangle$ denotes the indefinite scalar
product
$
  \langle x,y\rangle=x_{1}y_{1}+x_{2}y_{2}-x_{3}y_{3},
$
and the hyperbolic distance $d_{h}(x,y)$ between two points $x,y\in
H^{2}$ satisfies
\begin{equation*}
  \cosh d_{h}(x,y) = -\langle x,y\rangle.
\end{equation*}
This implies
\begin{equation*}
  \|x-y\|_{h}=2\sinh\Big(\tfrac{1}{2}\,d_{h}(x,y)\Big),
\end{equation*}
where $\|v\|_{h}=\sqrt{\langle v,v\rangle}$. 
To every  hyperbolic triangle in $H^{2}$ with sides of length
$\ell_{12}$, $\ell_{23}$, $\ell_{31}$, there corresponds a
\emph{secant triangle} in $\R^{2,1}$ whose sides are the straight line
segments in $\R^{2,1}$ connecting the vertices. Their lengths, as
measured by $\|\cdot\|_{h}$, are therefore $2\sinh(\ell_{ij}/2)$.

Note that the following statements are equivalent: 

\begin{compactenum}[(i)]
\item The restriction of the indefinite scalar product
  $\langle\cdot,\cdot\rangle$ of $\R^{2,1}$ to the affine plane of the
  secant triangle is positive definite and therefore induces a
  euclidean metric on that plane. 
\item The side lengths of the secant
  triangle satisfy the triangle inequalities. 
\item The circumcircle
  of the hyperbolic triangle is a proper circle. (The circumcircle of
  a hyperbolic triangle is either a proper circle or a horocycle or a
  curve of constant distance from a geodesic.)
\end{compactenum}
(The analogous statement for the
secant triangles of decorated ideal hyperbolic triangles is Penner's
Lemma~2.2 \cite{penner_decorated_1987}.)

Now let $(\T,\ell)_{h}$ and $(\T,\tilde\ell)_{h}$ be two hyperbolic
triangulations. The edge lengths of the secant triangles of
$(\T,\ell)_{h}$ are $2e^{\lambda/2}$ with $\lambda$ defined by
equation~\eqref{eq:lambda_hyp}, and similarly for
$(\T,\tilde\ell)_{h}$. Now $(\T,\ell)_{h}$ and $(\T,\tilde\ell)_{h}$
are discretely conformally equivalent if and only if
$\ell'=e^{\lambda/2}$ and $\tilde \ell'=e^{\tilde\lambda/2}$ are
related by equation~\eqref{eq:tilde_ell}, that is, related like
discrete metrics of discretely conformally equivalent euclidean
triangulations.
\end{remark}

\subsection{Derivation by volume considerations}
\label{sec:derive_by_volume}

The theory of discrete conformal equivalence for hyperbolic
triangulations is based on volume considerations for the type of
polyhedron shown in Figure~\ref{fig:ideal_halfprism}.
\begin{figure}
  \centering
  \includegraphics{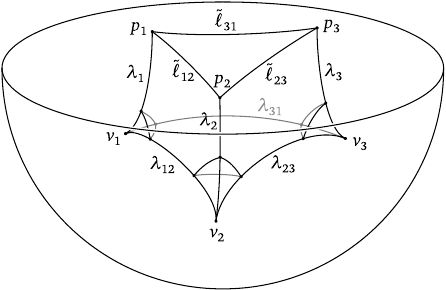}
  \caption{The polyhedral building block (shown in the Poincar\'e ball
    model) used to derive the theory of discrete conformal equivalence
    of hyperbolic triangulations.}
  \label{fig:ideal_halfprism}
\end{figure}
From the vertices $p_{1}$, $p_{2}$, $p_{3}$ of a triangle in
hyperbolic $3$-space, three rays run orthogonally to the plane of
the triangle until they intersect the infinite boundary in the ideal
points $v_{1}$, $v_{2}$, $v_{3}$. The convex hull of these six points
is a prism with three ideal vertices and right dihedral angles at the
base triangle $p_{1}p_{2}p_{3}$. Let the dihedral angles at the three rays from $p_{1}$,
$p_{2}$, $p_{3}$ be $\alpha_{1}$, $\alpha_{2}$, $\alpha_{3}$.  Since
the dihedral angles sum to $\pi$ at the ideal vertices, the dihedral
angles $\alpha_{12}$, $\alpha_{23}$, $\alpha_{31}$ at edges
$v_{1}v_{2}$, $v_{2}v_{3}$, $v_{3}v_{1}$ satisfy
equations~\eqref{eq:alpha_ij}.  Let $\tilde\ell_{ij}$ be the lengths
of the finite edges, and let $\lambda_{i}$ and $\lambda_{ij}$ be the
lengths of the infinite edges truncated at some horospheres centered
at the ideal vertices $v_{i}$, as shown in
Figure~\ref{fig:ideal_halfprism}.

\begin{lemma}[Leibon~\cite{leibon_characterizingdelaunay_2002}]
  \label{lem:edge_lengths}
  The (truncated) edge lengths of the prism shown in
  Figure~\ref{fig:ideal_halfprism} are related by
  equations~\eqref{eq:tilde_ell_with_arsinh}.
\end{lemma}

\begin{proof}
  We consider the case when $\lambda_{1}=\lambda_{2}=\lambda_{3}=0$
  (that is, when the truncating horospheres touch the base plane in
  $p_{1}$, $p_{2}$, $p_{3}$), from which the general case follows
  easily. Figure~\ref{fig:proof_lem_edge_lengths_and_angle_of_parallelity}
  (left)
  \begin{figure}
    \centering
    \hfill
    \includegraphics{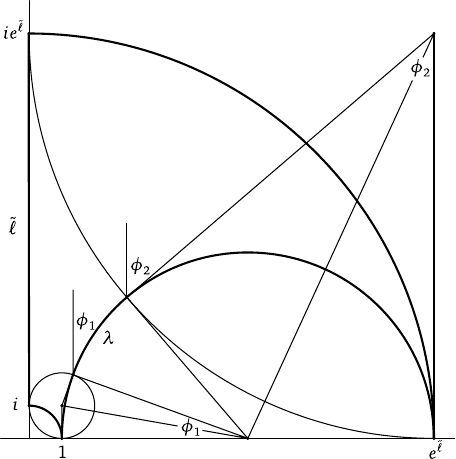}
    \hfill
    \raisebox{1\baselineskip}{\includegraphics{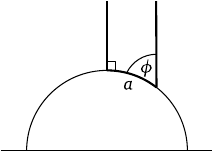}}
    \hspace*{\fill}
    \caption{\emph{Left:} Proof of Lemma~\ref{lem:edge_lengths}.
    \emph{Right:} The angle $\phi$ and the side length $a$ in a right angled
      hyperbolic triangle with an ideal vertex satisfy the equation
      $a=\log\cot(\phi/2)$.}
    \label{fig:proof_lem_edge_lengths_and_angle_of_parallelity}
  \end{figure}
 shows one of the
  side quadrilaterals of the prism in the half-plane model. 
  We will show that
  \begin{equation}
    \label{eq:lambda_in_proof_lem_edge_lengts}
    \lambda = 2\log\sinh\big(\frac{\tilde\ell}{2}\big),
  \end{equation}
  which proves this special case. We have
  \begin{equation*}
    \phi_{1}=2\arccot\Big(\frac{e^{\tilde\ell}-1}{2}\Big),\qquad
    \phi_{2}=2\arccot\Big(\frac{2e^{\tilde\ell}}{e^{\tilde\ell}-1}\Big),
  \end{equation*}
  The equation for the ``angle of parallelity'' (see
  Figure~\ref{fig:proof_lem_edge_lengths_and_angle_of_parallelity}
  (right))
  implies that
  \begin{equation*}
    \lambda=\log\cot(\phi_{1}/2)-\log\cot(\phi_{2}/2),
  \end{equation*}
  and hence equation~\eqref{eq:lambda_in_proof_lem_edge_lengts} holds.
\end{proof}

The volume of the polyhedron shown in Figure~\ref{fig:ideal_halfprism}
is
\begin{multline}
  \label{eq:Vhyp}
  \Vhyp(\alpha_{1},\alpha_{2},\alpha_{3})=\frac{1}{2}\Big(
  \ML(\alpha_{1}) + \ML(\alpha_{2}) + \ML(\alpha_{3})
  + \ML(\alpha_{12}) + \ML(\alpha_{23}) + \ML(\alpha_{31})\\
  +\ML\Big(
  \tfrac{1}{2}\big(\pi-\alpha_{1}-\alpha_{2}-\alpha_{3}\big)
  \Big)\Big).
\end{multline}
This was shown by Leibon~\cite{leibon_characterizingdelaunay_2002},
who also showed that the volume function $\Vhyp$ is strictly concave
on its domain of definition,
\begin{equation*}
\big\{
(\alpha_{1},\alpha_{2},\alpha_{3})\in\R^{3}
\,\big|\,
\alpha_{1}>0,\;
\alpha_{2}>0,\;
\alpha_{3}>0,\;
\alpha_{1}+\alpha_{2}+\alpha_{3}<\pi
\big\}.
\end{equation*}
By Schl\"afli's formula, 
\begin{equation*}
  \label{eq:dVhyp}
  d\Vhyp=-\tfrac{1}{2}\big(
  \lambda_{1}\,d\alpha_{1}+
  \lambda_{2}\,d\alpha_{2}+
  \lambda_{3}\,d\alpha_{3}+
  \lambda_{12}\,d\alpha_{12}+
  \lambda_{23}\,d\alpha_{23}+
  \lambda_{31}\,d\alpha_{31}\big).
\end{equation*}
(The choice of horospheres does not matter because the angle sum at
the ideal vertices is constant. Also note that the lengths
$\tilde\ell_{ij}$ of the finite edges do not appear in the equation
because their dihedral angles are constant.)

The function $\Vhathyp$ defined by
equation~\eqref{eq:Vhathyp} is
\begin{multline*}
  \Vhathyp(\lambda_{12},\lambda_{23},\lambda_{31}, 
  \lambda_{1}, \lambda_{2}, \lambda_{3})=\\
  \tfrac{1}{2}(
  \alpha_{1}\lambda_{1}
  +\alpha_{2}\lambda_{2}
  +\alpha_{3}\lambda_{3}
  +\alpha_{12}\lambda_{12}
  +\alpha_{23}\lambda_{23}
  +\alpha_{31}\lambda_{31})
  + \Vhyp(\alpha_{1},\alpha_{2},\alpha_{3}),
\end{multline*}
so that 
\begin{equation*}
  d\Vhathyp=\tfrac{1}{2}\big(
  \alpha_{1}\,d\lambda_{1}+
  \alpha_{2}\,d\lambda_{2}+
  \alpha_{3}\,d\lambda_{3}+
  \alpha_{12}\,d\lambda_{12}+
  \alpha_{23}\,d\lambda_{23}+
  \alpha_{31}\,d\lambda_{31}\big).
\end{equation*}
From this one obtains Proposition~\ref{prop:grad_Ehyp} on the partial
derivatives of $\Ehyp_{\T,\Theta,\lambda}$. By extending $\Vhathyp$
using essentially the same argument as in the proof of
Proposition~\ref{prop:E_extend}, one obtains
Proposition~\ref{prop:E_extend_hyp} on the extension of
$\Ehyp_{\T,\Theta,\lambda}$. To prove the convexity of
$\Ehyp_{\T,\Theta,\lambda}$ (Proposition~\ref{prop:E_convex_hyp}),
note that the function
\begin{equation*}
  \Vhathyp(\lambda_{12},\lambda_{23},\lambda_{31}, 
  \lambda_{1}, \lambda_{2}, \lambda_{3})
  -\tfrac{\pi}{4}\big(
  \lambda_{12}+\lambda_{23}+\lambda_{31}\big)
\end{equation*}
really only depends on the three parameters
\begin{equation*}
  \begin{split}
    x_{1} &= \tfrac{1}{4} 
    \big(
    +\lambda_{12} - \lambda_{23} + \lambda_{31} - 2\lambda_{1}
    \big)
    =\frac{\partial\Vhyp}{\partial\alpha_{1}},\\
    x_{2} &= \tfrac{1}{4} 
    \big(
    +\lambda_{12} + \lambda_{23} - \lambda_{31} - 2\lambda_{2}
    \big)
    =\frac{\partial\Vhyp}{\partial\alpha_{2}},\\
    x_{3} &= \tfrac{1}{4} 
    \big(
    - \lambda_{12} + \lambda_{23} + \lambda_{31} - 2\lambda_{3}
    \big)
    =\frac{\partial\Vhyp}{\partial\alpha_{3}},\\
  \end{split}
\end{equation*}
As function of these parameters, it is minus the Legendre transform of the
strictly concave function~$\Vhyp$:
\begin{equation*}
  \Vhathyp
  -\tfrac{\pi}{4}\big(\lambda_{12}+\lambda_{23}+\lambda_{31}\big)
  =
  -\alpha_{1}x_{1}-\alpha_{2}x_{2}-\alpha_{3}x_{3}+\Vhyp. 
\end{equation*}
Hence, $\Vhathyp$ is a locally strictly convex function of $x_{1}$,
$x_{2}$, $x_{3}$, and hence also of $\lambda_{1}$, $\lambda_{2}$,
$\lambda_{3}$, if $\lambda_{12}$, $\lambda_{23}$,
$\lambda_{31}$ are considered constant. The $C^{1}$ extension of
$\Vhathyp$ is linear outside the domain where the triangle
inequalities are satisfied, hence still convex.

\begin{remark}
  In the same way, one can derive a theory of discrete conformal
  equivalence for \emph{spherical} triangulations. In this case, the
  polyhedral building block is a tetrahedron with one finite and three
  ideal vertices. The functions involved in the corresponding
  variational principles are not convex. So in this case, the
  variational principles do not immediately lead to a uniqueness
  theorem, nor to a computational method for discrete conformal maps.
\end{remark}

\begin{appendices}

\section{Necessary conditions for the existence of a solution of the
  discrete conformal mapping problems}
\label{sec:nec_cond_exist}

In this appendix, we will discuss some rather obvious and rather mild
necessary conditions for the solvability of the discrete mapping
problems and how they relate to the behavior of the function
$E_{\T,\Theta,\lambda}(u)$. In short, the conditions are necessary for
the problems to have a solution and sufficient to ensure that
$E_{\T,\Theta,\lambda}(u)$ behaves ``sanely'', so that the following
solvability alternative (see the corollary to
Proposition~\ref{prop:E_tends_to_infty}) holds: Provided that we are
able to find a minimizer of a convex function if it exists, then the
variational principle allows us to either solve a discrete conformal
mapping problem or to ascertain that it is not solvable.

\subsection{The discrete Gauss-Bonnet condition}
\label{sec:discrete_Gauss_Bonnet}

\begin{condition}
  \label{cond:sum_theta_equality}
  \quad$\sum_{i\in V}\Theta_{i} = \pi\,|T|\,$
\end{condition}

\noindent%
If Problem~\ref{prob:prescribe_Theta} has a solution then clearly
Condition~\ref{cond:sum_theta_equality} is satisfied (because the sum
of angle sums around vertices equals the sum of angle sums in
triangles). This is actually a discrete version of the Gauss-Bonnet
formula. If we set $K_{i} = 2\pi-\Theta_{i}$ for interior vertices and
$\kappa_{i}=\pi-\Theta_{i}$ for boundary vertices then
Condition~\ref{cond:sum_theta_equality} is equivalent to
\begin{equation*}
  \sum_{i\in \Vint}K_{i} + \sum_{i\in\Vbdy}\kappa_{i}=2\pi(|T|-|E|+|V|)).
\end{equation*}

\begin{proposition}
  \label{prop:scale_invariance}
  The function $E_{\T,\Theta,\lambda}(u)$ is scale-invariant, that is,
  \begin{equation*}
    E_{\T,\Theta,\lambda}(u + h\,1_{V})=E_{\T,\Theta,\lambda}(u), 
  \end{equation*}
  if and only if Condition~\ref{cond:sum_theta_equality} is satisfied.
\end{proposition}

\begin{proof}
  Adding $h$ to every $u_{i}$ results in an added $2h$ to every
  $\tilde\lambda_{ij}$, see equation \eqref{eq:tilde_lambda}. Using
  equations~\eqref{eq:f_scaling_behavior} and~\eqref{eq:E_with_f}, one
  obtains
  \begin{equation*}
    E_{\T,\Theta,\lambda}(u + h\,1_{V}) = E_{\T,\Theta,\lambda}(u) 
    + h\,\Big(\sum_{i\in V}\Theta_{i}-\pi\,|T|\Big).\qedhere
  \end{equation*}
\end{proof}

\subsection{The solvability alternative}
\label{sec:solvability_alternative}

The following stronger Conditions~\ref{cond:angle_system}
and~\ref{cond:subsets} are also obviously necessary for the existence
of a solution of Problem~\ref{prob:prescribe_Theta}. Moreover, if a
solution to the general Problem~\ref{prob:general} exists (where
$\Theta_{i}$ is prescribed only for $i\in V_{1}$), then positive
$\Theta$-values can be assigned also to the vertices in $V_{0}$ so
that Conditions~\ref{cond:angle_system} and~\ref{cond:subsets} are
satisfied.

\begin{condition}
  \label{cond:angle_system}
  There exists a system of angles $\widehat\alpha>0$, such that
  \begin{equation}
    \label{eq:coherent_angle_sum_triangle}
    \widehat\alpha_{jk}^{i}+\widehat\alpha_{ki}^{j}+\widehat\alpha_{ij}^{k}=\pi 
    \quad\text{for all}\quad ijk\in T,
  \end{equation}
  and
  \begin{equation*}
    \sum_{jk:ijk\in T}\widehat\alpha_{jk}^{i}=\Theta_{i}
    \quad\text{for all}\quad i\in V.
  \end{equation*}
\end{condition}

\begin{condition}
  \label{cond:subsets}
  If $T_{1}$ is any subset of $T$ and $V_{1}\subseteq V$ is the set of
  all vertices of the triangles in $T_{1}$, that is,
  \begin{equation*}
    V_{1}=\bigcup_{ijk\in T_{1}}\{i,j,k\}, 
  \end{equation*}
  then
  \begin{equation*}
    \pi\,\big|T\setminus T_1\big| \geq \sum_{i\in V\setminus V_{1}}\Theta_{i},
  \end{equation*}
  where equality holds if and only if $T_{1}=\emptyset$ or $T_{1}=T$.
\end{condition}

\begin{proposition}
  \label{prop:conditions_equivalent}
  Conditions~\ref{cond:angle_system} and~\ref{cond:subsets} are equivalent.
\end{proposition}

\begin{proof}
  The implication `Condition~\ref{cond:angle_system} $\Rightarrow$
  Condition~\ref{cond:subsets}' is easy to see. Regarding the converse
  implication, Colin~de~Verdi\`ere proves a similar statement using
  the feasible flow theorem~\cite[Section
  7]{colin_de_verdire_un_1991}. It is straightforward to adapt his proof for
  Proposition~\ref{prop:conditions_equivalent}.
\end{proof}

\begin{proposition}
  \label{prop:E_tends_to_infty}
  If Condition~\ref{cond:angle_system} or~\ref{cond:subsets} is
  satisfied (and hence both of them are), then
  \begin{equation*}
    E_{\T,\Theta,\lambda}(u)\longrightarrow\infty
    \quad\text{if}\quad
    \max_{i\in V}u_{i}-\min_{i\in V}u_{i}\longrightarrow\infty\,.
  \end{equation*}
\end{proposition}

\begin{definition}[Reasonably posed mapping problems]
  We say that Problem~\ref{prob:prescribe_Theta} is \emph{reasonably
    posed} if Condition~\ref{cond:angle_system} or
  Condition~\ref{cond:subsets} is satisfied (and hence both of them
  and Condition~\ref{cond:sum_theta_equality} are). We say that
  Problem~\ref{prob:general} (where $\Theta_{i}$ is prescribed only
  for $i\in V_{1}$) is \emph{reasonably posed} if positive
  $\Theta$-values can be assigned also to the vertices in $V_{0}$ so
  that Conditions~\ref{cond:angle_system} or
  Condition~\ref{cond:subsets} are satisfied (and hence both of them
  and Condition~\ref{cond:sum_theta_equality} are).
\end{definition}

\begin{corollary}[Solvability alternative]
  If Problem~\ref{prob:prescribe_Theta} or Problem~\ref{prob:general}
  are reasonably posed, then $E_{\T,\Theta,\lambda}(u)$ (maybe with
  some variables $u_{i}$ fixed) has a minimizer $u_{\textit
    min}$. Either $u_{\textit min}$ is contained in the domain where
  all triangle inequalities are satisfied, in which case it is unique
  (up to an additive constant if no variables are fixed) and
  corresponds to the solution of the discrete conformal mapping
  problem, or it lies outside that domain, in which case the
  corresponding discrete conformal mapping problem does not have a
  solution.
\end{corollary}

\begin{proof}[Proof of Proposition~\ref{prop:E_tends_to_infty}]
  Using the (constant) angles $\widehat\alpha$ we can rewrite the sum
  over vertices on the right-hand side of equation~\eqref{eq:E_with_f}
  as a sum over triangles:
  \begin{equation*}
    \sum_{i\in V}\Theta_iu_{i}=\sum_{ijk\in T}(
    \widehat\alpha_{jk}^{i}u_{i} +
    \widehat\alpha_{ki}^{j}u_{j} +
    \widehat\alpha_{ij}^{k}u_{k})\,.
  \end{equation*}
  Expressing $u$ in terms of $\tilde\lambda$ and $\lambda$, we obtain
  \begin{multline*}
    \widehat\alpha_{jk}^{i}u_{i} +
    \widehat\alpha_{ki}^{j}u_{j} +
    \widehat\alpha_{ij}^{k}u_{k}
    =
    \big(\tfrac{\pi}{2}-\widehat\alpha_{ij}^{k}\big)
    (\tilde\lambda_{ij}-\lambda_{ij})
    +\big(\tfrac{\pi}{2}-\widehat\alpha_{jk}^{i}\big)
    (\tilde\lambda_{jk}-\lambda_{jk})\\
    +\big(\tfrac{\pi}{2}-\widehat\alpha_{ki}^{j}\big)
    (\tilde\lambda_{ki}-\lambda_{ki})\,,
  \end{multline*}
  so
  \begin{equation*}
    E_{\T,\Theta,\lambda}(u) = \sum_{ijk\in T}
    \Big(
    2f\big(\tfrac{\tilde\lambda_{ij}}{2}, 
    \tfrac{\tilde\lambda_{jk}}{2}, 
    \tfrac{\tilde\lambda_{ki}}{2}\big)
    - \widehat\alpha_{ij}^{k}\tilde\lambda_{ij}
    - \widehat\alpha_{jk}^{i}\tilde\lambda_{jk}
    - \widehat\alpha_{ki}^{j}\tilde\lambda_{ki}
    \Big)
    + \const\,,
  \end{equation*}
  where here and in the following ``$\const$'' stands for terms that
  do not depend on $u$. Using the estimate of
  Proposition~\ref{prop:estimate} and
  equation~\eqref{eq:coherent_angle_sum_triangle}, one obtains
  \begin{equation*}
    \begin{split}
      E_{\T,\Theta,\lambda}(u) 
      &\geq 
      \sum_{ijk\in T} 
      \Big(
      \pi\max\{
      \tilde\lambda_{ij},
      \tilde\lambda_{jk},
      \tilde\lambda_{ki}\} -
      \widehat\alpha_{ij}^{k}\tilde\lambda_{ij} -
      \widehat\alpha_{jk}^{i}\tilde\lambda_{jk} -
      \widehat\alpha_{ki}^{j}\tilde\lambda_{ki} \Big) + \const
      \\
      &=\sum_{ijk\in T} 
      \Big(
      \widehat\alpha_{ij}^{k}
      \big(\max\{\ldots\}-\tilde\lambda_{ij}\big)
      +
      \widehat\alpha_{jk}^{i}
      \big(\max\{\ldots\}-\tilde\lambda_{jk}\big)
      +
      \widehat\alpha_{ki}^{j}
      \big(\max\{\ldots\}-\tilde\lambda_{ki}\big)
      \Big) + \const
      \\
      &\geq
      \min_{\begin{smallmatrix}
          k\\ij
        \end{smallmatrix}\in A}\{\widehat\alpha_{ij}^{k}\}
      \sum_{ijk\in T}
      \big(
      \max\{\tilde\lambda_{ij},\tilde\lambda_{jk},\tilde\lambda_{ki}\}
      -
      \min\{\tilde\lambda_{ij},\tilde\lambda_{jk},\tilde\lambda_{ki}\}
      \big) + \const
    \end{split}
  \end{equation*}
  Now if $ijk\in T$, then
  $u_{i}-u_{j}=\tilde\lambda_{ki}-\tilde\lambda_{jk}-\lambda_{ki}+\lambda_{jk}$,
  so
  \begin{multline*}
    \max\{u_{i},u_{j},u_{k}\} - \min\{u_{i},u_{j},u_{k}\}
    \leq 
    \max\{\tilde\lambda_{ij},\tilde\lambda_{jk},\tilde\lambda_{ki}\}
    -
    \min\{\tilde\lambda_{ij},\tilde\lambda_{jk},\tilde\lambda_{ki}\}
    + \const\,,
  \end{multline*}
  and because the triangulated surface is connected this implies
  \begin{equation*}
    E_{\T,\Theta,\lambda}(u) \geq \min\{\widehat\alpha_{ij}^{k}\}
    \big(\max_{i\in V}u_{i}-\min_{i\in V}u_{i}\big)+\const
    \qedhere
  \end{equation*}
\end{proof}

\section{The corresponding smooth conformal mapping problems and
  variational principles}
\label{sec:smooth}

A natural question regarding the two variational principles for
discrete conformal maps presented in
Sections~\ref{sec:variational_principle_1} and
\ref{sec:variational_principle_2} is: ``What are the corresponding
variational principles in the classical smooth theory of conformal
maps?'' In fact, even the question ``What exactly are the
corresponding smooth mapping problems?'' deserves a comment. For the
second variational principle it is not even obvious how the
variables---triangle angles---translate to the smooth theory.

\subsection{Background: Curvature, unit vector fields, and conformal metrics}
\label{sec:smooth_background}

Before we will address these questions in
Sections~\ref{sec:smooth_problems}
and~\ref{sec:smooth_second_principle}, we outline some classical
background material from the differential geometry of surfaces. The
purpose is twofold: first, to fix notation; second, our exposition
takes a particular point of view, focusing on unit vector fields,
which prepares the discussion of the second variational principle in
Section~\ref{sec:smooth_second_principle}.

Let $M$ be a smooth oriented surface, possibly with boundary, equipped
with a \emph{Riemannian metric} $g$ and the induced \emph{Levi--Civita
  connection} $\nabla$. The Riemannian metric and orientation induce a
\emph{$90\degrees$-rotation tensor}
\begin{equation*}
J:TM\rightarrow TM
\end{equation*}
and an \emph{area $2$-form}
\begin{equation*}
\sigma=g(J\cdot,\cdot).
\end{equation*}
A \emph{unit vector field} on $M$ is a tangent vector field $Y$ with
$g(Y,Y)=1$. Of course, the existence of a unit vector field imposes
restrictions on the topology of $M$. In any case, unit vector fields
exist locally, so purely local considerations remain valid for
arbitrary $M$.
The \emph{Gauss curvature} $K\in C^{\infty}(M)$ is defined by the equation
\begin{equation*}
  K=-g\big(R(Y,JY)Y,JY\big),
\end{equation*}
where $Y$ is any unit vector field, and $R$ denotes the Riemann
curvature tensor,
\begin{equation}
  \label{eq:Riem_curv}
  R(X,Y)Z=\nabla_{X}\nabla_{Y}Z-\nabla_{Y}\nabla_{X}Z-\nabla_{[X,Y]}Z.
\end{equation}
The \emph{curvature $2$-form} is defined by
\begin{equation*}
  \Omega=K\sigma.
\end{equation*}
For a unit vector field $Y$, we define the \emph{rotation $1$-form}
$\rho_{Y}$ by
\begin{equation*}
  \rho_{Y}(X)=g(\nabla_{X}Y,JY).
\end{equation*}

\begin{proposition}
  \label{prop:drho}
  For any unit vector field $Y$, 
  \begin{equation*}
    d\rho_{Y}=-\Omega.
  \end{equation*}
\end{proposition}

\begin{proof}
  The claim follows from the definition of the Riemann curvature tensor~\eqref{eq:Riem_curv},
  by a straightforward calculation:
  \begin{equation*}
    \begin{split}
      d\rho_{Y}(Y,JY) =& Y\cdot g(\nabla_{JY}Y,JY) -(JY)\cdot
      g(\nabla_{Y}Y,JY)
      -g(\nabla_{[Y,JY]}Y,JY) \\
      =& g(\nabla_{Y}\nabla_{JY}Y,JY)
      + \underbrace{g(\nabla_{JY}Y,\nabla_{Y}JY)}_{=0}\\
      &- g(\nabla_{JY}\nabla_{Y}Y,JY) -
      \underbrace{g(\nabla_{Y}Y,\nabla_{JY}JY)}_{=0}
      -g(\nabla_{[Y,JY]}Y,JY)\\
      =& g(R(Y,JY)Y, JY)=-K.
    \end{split}
  \end{equation*}
  We have used that $\nabla_{V} Y\perp Y$ and $\nabla_{V} JY\perp JY$
  for any vector field $V$, so that, because $M$ is two-dimensional,
  $g(\nabla_{V} Y,\nabla_{W} JY)=0$ for any $V,W$.
\end{proof}

Now consider a conformal change of metric with conformal factor
$e^{u}$ determined by equation~\eqref{eq:tilde_g}. Note that a
conformal change of metric is also characterized by the fact that the
$90\degrees$-rotation with respect to the new metric~$\tilde g$ is the
same tensor $J$.  The Levi--Civita connection $\widetilde\nabla$ of
$\tilde g$ is related to the Levi--Civita connection $\nabla$ of $g$
by
\begin{equation}
  \label{eq:tilde_nabla}
  \widetilde\nabla_{X}Z=
  \nabla_{X}Z+g(X,G)Z+g(Z,G)X-g(X,Z)G,
\end{equation}
where $G=\grad_{g}u$, that is, $du=g(G,\cdot)$. A unit vector field
$Y$ with respect to $g$ naturally determines a unit vector field
\begin{equation*}
  \tilde Y=e^{-u}\,Y
\end{equation*}
with respect to $\tilde g$. Its rotation $1$-form is 
\begin{equation*}
  \tilde\rho_{\tilde
  Y}(X)=\tilde g(\widetilde\nabla_{X}\tilde Y, J\tilde Y). 
\end{equation*}
\begin{proposition}
  \label{prop:tilde_rho}
  The rotation $1$-forms $\rho_{Y}$, $\tilde\rho_{\tilde Y}$ are
  related by
  \begin{equation*}
    \tilde\rho_{\tilde Y}=\rho_{Y}+*du\,,
  \end{equation*}
  where $*$ denotes the Hodge star operator for $g$.
\end{proposition}

\noindent%
The \emph{Hodge star operator} $*$ maps a $1$-form $\omega$ to the $1$-form
$*\omega=-\omega(J\cdot)$. It also maps a function ($0$-form) $f$ to
the $2$-form $*f=f\sigma$ and vice versa, $*{f\sigma}=f$. Note that on
a $2$-dimensional manifold, the action of the Hodge star operator on
$1$-forms depends only on the conformal class of the metric.

\begin{proof}[Proof of Proposition~\ref{prop:tilde_rho}]
  By the product rule,
  \begin{equation*}
    \widetilde\nabla_{X}\tilde Y = e^{-u}\,(
    - du(X)Y + \widetilde\nabla_{X}Y),
  \end{equation*}
  so
  $
      \tilde\rho_{\tilde Y}(X) 
      = 
      g(\widetilde\nabla_{X}Y,JY)\,.
  $
  Using~\eqref{eq:tilde_nabla} one obtains
  \begin{equation*}
    \begin{split}
      \tilde\rho_{\tilde Y}(X) = & 
      \rho_{Y}(X) + g(Y,G)g(X,JY) - g(X,Y)g(G,JY).
    \end{split}
  \end{equation*}
  Finally, since $J$ is skew, $J^{2}=-1$, and $(Y,JY)$ is an
  orthonormal frame,
  \begin{equation*}
    \begin{split}
      g(Y,G)g(X,JY) - g(X,Y)g(G,JY) &=
      -g(G,Y)g(Y,JX) - g(G,JY)g(JY,JX)\\
      &= -g(G, JX)=-du(JX)=*du(X).
    \end{split}
  \end{equation*}
  This completes the proof.
\end{proof}

As a corollary of Propositions~\ref{prop:drho}
and~\ref{prop:tilde_rho}, we obtain the equation relating the curvature
$2$-forms of $g$ and $\tilde g$,
\begin{equation}
  \label{eq:tilde_Omega}
  \tilde\Omega=\Omega-d*du,
\end{equation}
and hence \emph{Liouville's equation} for the curvatures,
\begin{equation}
  \label{eq:Liouville}
  e^{2u}\tilde K=K+\Delta u,
\end{equation}
where $\Delta$ is the \emph{Laplace--Beltrami} operator with respect to $g$,
\begin{equation*}
  \Delta f=-{*d{*df}}.
\end{equation*}
(We use the sign convention for the Laplace operator that renders it
positive semidefinite.)

\subsection{Smooth mapping problems and the first variational principle}
\label{sec:smooth_problems}

Which problems in the smooth theory are analogous to the discrete
conformal mapping problems discussed in this paper? There are two
fairly obvious candidates:

\begin{problem}
  \label{prob:smooth_K}
  Given $(M,g)$ and the function $\tilde K$, find a conformally
  equivalent Riemannian metric $\tilde g=e^{2u}g$ with curvature
  $\tilde K$.
\end{problem}

This amounts to solving Liouville's equation~\eqref{eq:Liouville} for
$u$.

\begin{problem}
  \label{prob:smooth_Omega}
  Given $(M,g)$ and the $2$-form $\tilde\Omega$, find a conformally
  equivalent Riemannian metric $\tilde g=e^{2u}g$ with curvature $2$-form
  $\tilde \Omega$. 
\end{problem}

Prescribing the target curvature $2$-form is equivalent to prescribing
$e^{2u}\tilde K$ instead of the target curvature $\tilde
K$. Problem~\ref{prob:smooth_Omega} amounts to solving
equation~\eqref{eq:tilde_Omega}, which is equivalent to
\emph{Poisson's equation}
\begin{equation*}
  \Delta u = f
\end{equation*}
with right-hand side $f=*\tilde\Omega-K$.

For both problems, one may prescribe either $u$ on the boundary
\emph{(Dirichlet conditions)} or $*du|_{T\partial M}$ with $\int_{\partial
  M}*du=-\int_{M}(\tilde\Omega-\Omega)$ \emph{(Neumann conditions)}.

Both Liouville's equation~\eqref{eq:Liouville} and the Poisson
equation~\eqref{eq:tilde_Omega} are variational, with the respective
functionals
\begin{equation}
  \label{eq:E_A}
  E_{A}(u)=\int_{M}\big(
  \tfrac{1}{2} du \wedge *du
  - (\tfrac{1}{2}e^{2u} \tilde K - u K)\sigma
  \big)
\end{equation}
for Liouville's equation, where $\tilde K$ and $K$ are fixed given
functions on $M$, and
\begin{equation}
  \label{eq:E_B}
  E_{B}(u)=\int_{M}\big(
  \tfrac{1}{2}du \wedge *du 
  - u(\tilde\Omega - \Omega)
  \big)
\end{equation}
for equation~\eqref{eq:tilde_Omega}, where $\tilde\Omega$ and $\Omega$
are fixed given $2$-forms on $M$. 

\textbf{Question:} Which of the two candidates, Problem~\ref{prob:smooth_K} or
Problem~\ref{prob:smooth_Omega}, is the smooth version of the discrete
mapping problems described in Section~\ref{sec:mapping_problems}?

\textbf{Answer:} Comparing the scaling behavior
shows that the discrete mapping problems are discretization of
Problem~\ref{prob:smooth_Omega} and not of
Problem~\ref{prob:smooth_K}. The function $E_{\T,\Theta,\lambda}$ of
the first variational principle corresponds to $E_{B}$ and not to
$E_{A}$.

Indeed, although we did denote the angle defect $2\pi-\Theta_i$ at a
vertex $i$ of a triangulation by $K_{i}$ in
Appendix~\ref{sec:nec_cond_exist}, this angle defect is a
discretization of the curvature $2$-form $\Omega$ and not of the Gauss
curvature $K$, the latter being an angle defect \emph{per surface
  area}.  The difference manifests itself in the scaling behavior: The
angle defects at vertices and the curvature $2$-form remain invariant
if lengths are scaled by a constant factor. The Gauss curvature, on
the other hand, is inversely proportional to the square of the scale
factor. Thus, prescribing the angle defects $K_{i}$ at the vertices of
a triangulation corresponds to prescribing the curvature $2$-form
$\Omega$ of a smooth surface, as in Problem~\ref{prob:smooth_Omega},
and not the Gauss curvature $K$, as in Problem~\ref{prob:smooth_K}.

\begin{remark}
  For $\tilde K=0$ and $\tilde\Omega=0$ there is no difference between
  Problems~\ref{prob:smooth_K} and~\ref{prob:smooth_Omega}, and
  $E_{A}=E_{B}$.
\end{remark}

\begin{remark}
  Consider the gradient flow of the discrete functional
  $E_{\T,\Theta,\lambda}$.  For a closed triangulated surface with
  prescribed angle sum $\Theta_{i}=2\pi$ for all vertices, this gradient
  flow is equal to
  \begin{equation}
    \label{eq:discrete_curvature_flow}
    \dot u_{i}(t) = -K_{i}(t),
  \end{equation}
  where $K_{i}(t)$ is the angle defect around vertex $i$ at time
  $t$. At first glance, this looks like a discrete version of the
  Ricci flow for surfaces~\cite{luo_combinatorial_2004}. For surfaces,
  Ricci flow is the same as Yamabe flow because the Ricci tensor is
  proportional to the Riemannian metric. The Riemannian metric evolves
  conformally, $g_{t}=e^{2u_{t}}g_{0}$, according to the law
  \begin{equation}
    \label{eq:Ricci}
    \quad\dot u_{t}=-K_{t},
  \end{equation}
  where $K_{t}$ is the Gauss curvature at time $t$.

  However, the above comparison of the scaling behaviors of angle
  defect $K_{i}$ and Gauss curvature $K$ shows
  that~\eqref{eq:discrete_curvature_flow} is not a discretization
  of~\eqref{eq:Ricci}. In fact, the
  flow~\eqref{eq:discrete_curvature_flow} is a discretization of the
  gradient flow of $E_{B}$,
  \begin{equation*}
    \dot u_{t}=-e^{2u_{t}}K_{t},
  \end{equation*}
  which is a curvature flow for surfaces that is different from the
  Ricci/Yamabe flow.

  The same scaling argument applies to other flows that have
  mistakenly been tagged combinatorial or discrete Ricci
  flow~\cite{chow_combinatorial_2003}, \cite{gu_computational_2008}.
\end{remark}

\subsection{The second variational principle and harmonic unit vector
  fields}
\label{sec:smooth_second_principle}

The variables of the second variational principle
(Section~\ref{sec:variational_principle_2}) are assignments of new
angles in a euclidean triangulation. How do these variables translate
to smooth manifolds? Consider an angle vector
$\alpha\in(\R_{>0})^{A_{\T}}$ that assigns angle values
$\alpha_{jk}^{i}>0$ to the corners 
$\begin{smallmatrix}
  i\\jk
\end{smallmatrix}\in A_{\T}
$ of the triangles in such a way that the sum is $\pi$ in each
triangle. While such an angle assignment fixes the shape of each
triangle (up to similarity), a consistent assignment of edge lengths
is generally not possible. The assigned angles do, however, induce a
sensible definition of parallel transport of unit vectors from edge to
edge: The unit vector that makes an angle $\beta_{ij}$ with the
directed edge $ij$ in triangle $ijk$ is transported to the unit vector
that makes an angle $\beta_{jk}=\beta_{ij}+\alpha_{ki}^{j}-\pi$ with
edge $jk$ (see Figure~\ref{fig:transport_directions}).
\begin{figure}
  \centering
  \includegraphics{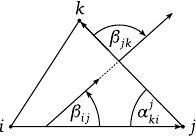}
  \caption{Parallel transport of directions from edge to edge.}
  \label{fig:transport_directions}
\end{figure}
Therefore, an angle assignment $\alpha\in(\R_{>0})^{A_{\T}}$ in a
euclidean triangulation corresponds to a connection of the unit
tangent bundle $T_{1}M$ of the smooth surface $M$.

For simplicity, our discussion of the second variational principle
will focus on the special case when the triangulation is topologically
a closed disk and the prescribed angle sums at interior vertices are
$2\pi$. In the smooth setting, angle assignments that sum to $\pi$ in
each triangle and to $2\pi$ around each vertex correspond to flat
connections of the unit tangent bundle.
Since the surface is assumed to be simply connected, for any such flat
connection there exists a parallel unit vector field and this is
unique up to rotation by a constant angle. Conversely, any unit vector
field is parallel for a unique flat connection. Thus, this special
case allows a more intuitive treatment involving unit vector fields
and rotation $1$-forms instead of connections and connection
$1$-forms. At the end of this section, we will indicate how to treat
the general case.

So assume for now that $M$ is diffeomorphic to a closed disk and
consider Problem~\ref{prob:smooth_Omega} with $\tilde\Omega=0$. That is, we are
looking for a conformally equivalent flat metric. 
The \emph{Dirichlet energy} of a unit vector field $Y$ is 
\begin{equation}
  \label{eq:smooth_Dirichlet_S}
  S(Y) = \tfrac{1}{2}\int_{M}\rho_{Y}\wedge *\rho_{Y}.
\end{equation}
Critical points of this Dirichlet energy are the \emph{harmonic
  sections of the unit tangent bundle}. Admissible variations are
within the space of unit vector fields, fixing the values on the
boundary.

\begin{proposition}[First variation]
  \label{prop:first_variation}
  Let $Y_{t}$ be a variation of the unit vector field $Y=Y_{0}$ with
  \begin{equation*}
    \frac{d}{dt}\Big|_{t=0} Y_{t}=h\,JY,
  \end{equation*}
  where $h\in C^{\infty}(M)$. Then
  \begin{equation*}
    \frac{d}{dt}\Big|_{t=0}
    S(Y_{t})=-\int_{M} h\;d {*\rho_{Y}}+\int_{\partial M}h\;{*\rho_{Y}}\,.
  \end{equation*}
\end{proposition}

\begin{proof}
  This follows from $\tfrac{d}{dt}\big|_{t=0}\,\rho_{Y_{t}}=dh$.
\end{proof}

\begin{corollary}
  A unit vector field $Y$ is a critical point of $S$ under variations
  that fix $Y$ on the boundary ${\partial M}$ if and only if
  \begin{equation*}
    d{*\rho_{Y}}=0.
  \end{equation*}
  It is also a critical point of $S$ under arbitrary variations if and
  only if, additionally,
  \begin{equation}
    \label{eq:star_rho_on_boundary_vanishes}
    {*\rho_{Y}}\big|_{T\partial M}=0.
  \end{equation}
\end{corollary}

Loosely speaking, the following proposition says that straightest unit
vector fields with respect to $g$ are parallel with respect to a
conformally equivalent flat metric $\tilde g$ with trivial global
holonomy.

\begin{proposition}[Smooth version of the second variational  principle]
  Suppose the unit vector field $Y$ is a critical point of $S$ under
  variations that fix $Y$ on the boundary. Define the function $u$ up
  to an additive constant by
  \begin{equation*}
    du=*\rho_Y.
  \end{equation*}
  (This is possible since $*\rho_{Y}$ is closed by the above corollary
  and we assumed that $M$ was diffeomorphic to a disk.) Let $\tilde
  g=e^{2u}g$. Then:

  (i) The unit vector field $\tilde Y = e^{-u}Y$ is parallel with
  respect to $\tilde g$, so $\tilde g$ is flat.

  (ii) The geodesic curvature of the boundary $\partial M$ with
  respect to $\tilde g$ is
  \begin{equation*}
    \tilde\kappa = \kappa - \rho_{Y}(T),
  \end{equation*}
  where $\kappa$ is the geodesic curvature with respect to $g$ and $T$
  is the positively oriented unit tangent vector field to $\partial
  M$.  

  (iii) If $Y$ is also a critical point of $S$ under arbitrary variations,
  then $u|_{\partial M}$ is constant.
\end{proposition}

\begin{proof}
  Since $*du=-\rho_{Y}$, the rotation form of $\tilde\rho_{\tilde Y}$
  vanishes by Proposition~\ref{prop:tilde_rho}. This implies~(i). The
  geodesic boundary curvatures are $\kappa=\rho_{T}(T)$ and
  $\tilde\kappa=\tilde\rho_{\tilde T}(\tilde T)$. (Locally extend the
  unit vector field $T$ inwards from the boundary.) Again by
  Proposition~\ref{prop:tilde_rho}, this implies~(ii). Finally, (iii)
  follows immediately from~\eqref{eq:star_rho_on_boundary_vanishes}.
\end{proof}

In the general case, $M$ is not restricted to be diffeomorphic to
a closed disk and one is looking for a conformally equivalent metric
$\tilde g$ with prescribed curvature $2$-form $\tilde\Omega$. To treat
this case variationally, consider the functional
\begin{equation*}
  S(\rho)=\tfrac{1}{2}\int_{M}\rho\wedge{*\rho}
\end{equation*}
on the affine space of $1$-forms $\rho$ satisfying
$d\rho=\tilde\Omega-\Omega$. We leave the details to the reader, not
because they are tedious but because they are interesting. The
critical points correspond to conformally equivalent similarity
structures, that is, ``metrics'' which may have global scaling
holonomy. (Compare the remark at the end of
Section~\ref{sec:variational_principle_2}.)
 
\section{Relation to circle patterns}
\label{sec:circle_patterns_etc}

\subsection{Two variational principles for circle patterns}
\label{sec:circle_patterns}

While the discrete conformal mapping problems essentially ask for
ideal hyperbolic polyhedra with prescribed metric, the circle pattern
problem below asks for an ideal polyhedron with prescribed dihedral
angles. Rivin's variational principle for this type of
problem~\cite{rivin_euclidean_1994} is very similar to our second
variational principle for discrete conformal maps. The function is
essentially the same, only the constraints placed on the angle
assignments are stronger. Also, the first variational principle needs
only a slight modification to become a variational principle for
circle patterns.

\begin{problem}[Circle pattern problem]
  \label{prob:circle_pattern}
  \emph{\textbf{Given}} a surface triangulation $\T$ and a function
  $\Phi\in(0,\pi]^{E}$, \emph{\textbf{find}} a discrete metric $\tilde\ell$ so
  that the euclidean triangulation $(\T,\tilde\ell)$ has circumcircle
  intersection angles $\Phi_{ij}$ as shown in
  Figure~\ref{fig:circle_pattern_theta}.
  \begin{figure}
    \centering
    \includegraphics{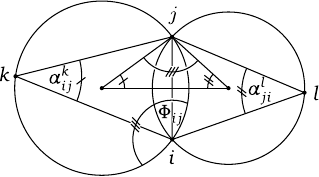}
    \caption{Circumcircle intersection angles
      $\Phi_{ij}=\alpha_{ij}^{k}+\alpha_{ji}^{l}$. For a boundary
      edge~$ij\in \Ebdy$, define $\Phi_{ij}=\alpha_{ij}^{k}$.}
    \label{fig:circle_pattern_theta}
  \end{figure}
\end{problem}

\begin{proposition}[Rivin~\cite{rivin_euclidean_1994}]
  The angles $\tilde\alpha\in\R^{A}$ are the angles of a euclidean
  triangulation $(\T,\tilde\ell)$ that solves
  Problem~\ref{prob:circle_pattern} if and only if
  $S_{\T,\lambda}(\tilde\alpha)$ is the maximum of $S_{\T,\lambda}$ on
  the set of all $\alpha\in\R^{A}$ that satisfy
  \begin{compactenum}[(i)]
  \item $\alpha>0$,
  \item $\alpha_{ij}^{k}+\alpha_{jk}^{i}+\alpha_{ki}^{j}=\pi$\; for all
    triangles $ijk\in T$,
  \item $\alpha_{ij}^{k}+\alpha_{jk}^{l}=\Phi_{ij}$\; for all interior
    edges $ij\in\Eint$,
  \item $\alpha_{ij}^{k}=\Phi_{ij}$\; for all boundary edges $ij\in\Ebdy$.
  \end{compactenum}
\end{proposition}

\noindent%
Due to conditions \emph{(iii)} and \emph{(iv)}, the choice of the parameter
$\lambda\in\R^{E}$ of $S_{\T,\lambda}$ does not matter because
\begin{equation*}
  S_{\T,\lambda}(\alpha)=S_{\T,0}(\alpha) + \sum_{ij\in E}\Phi_{ij}\lambda_{ij}.
\end{equation*}
So in connection with circle patterns, it makes sense to consider only
\begin{equation*}
  S_{\T,0}(\alpha)=\sum\ML(\alpha_{ij}^{k}).
\end{equation*}

Now consider the first variational principle for discrete conformal
maps. For $\Phi\in\R^{E}$, $\Theta\in\R^{V}$ define 
\begin{gather}
  \notag
  \mathcal{E}_{\T,\Phi,\Theta}:\R^{E}\times\R^{V}\longrightarrow\R,\\
  \label{eq:E_circle_pattern}
    \mathcal{E}_{\T,\Phi,\Theta}(\lambda,u)= \sum_{ijk\in T} 2\hat
    V(\lambda_{ij},\lambda_{jk},\lambda_{ki},-u_{i},-u_{j},-u_{k})
    -\sum_{ij\in E}\Phi_{ij}\lambda_{ij} +\sum_{i\in
      V}\Theta_{i}u_{i}.
\end{gather}
Compare equation~\eqref{eq:E_of_lambda_with_V}. If $\Phi$ is defined
by equation~\eqref{eq:theta_ij_for_dconf}, then
\begin{equation*}
  \mathcal{E}_{\T,\Phi,\Theta}(\lambda,u)=E_{\T,\Theta,\lambda}(u).
\end{equation*}
So if we fix $\lambda$ and vary $u$, then we obtain the first
variational principle for discrete conformal maps. If, on the other
hand, we fix $u$ and vary $\lambda$, then we obtain a variational
principle for circle patterns
(see Proposition~\ref{prop:var_princ_for_cp_with_E} below). Interpret the
circumcircles and triangle sides as hyperbolic planes in the
half-plane model. Then, using equation~\eqref{eq:dVhat} and the fact
that opposite dihedral angles in an ideal tetrahedron are equal, one
gets for an interior edge $ij\in\Eint$ that
\begin{equation*}
    \frac{\partial}{\partial\lambda_{ij}}\,\mathcal{E}_{\T,\Phi,\Theta}
    = \tilde\alpha_{ij}^{k}+\tilde\alpha_{ji}^{l}-\Phi_{ij},
\end{equation*}
where $\tilde\alpha$ are the angles in the euclidean triangulation
$(\T,\tilde\ell)$ with $\tilde\ell$ determined by
equations~\eqref{eq:tilde_ell} and~\eqref{eq:lambda}.  (Here we assume
that $\tilde\ell$ satisfies the triangle inequalities. Otherwise the
angles are $0$ or $\pi$ as stipulated in
Proposition~\ref{prop:f_extend}.) In the same way, one gets for a
boundary edge $ij$
\begin{equation*}
    \frac{\partial}{\partial\lambda_{ij}}\,\mathcal{E}_{\T,\Phi,\Theta}
    = \tilde\alpha_{ij}^{k}-\Phi_{ij},
\end{equation*}
implying the following variational principle for circle
patterns. (Note that $\mathcal{E}_{\T,\Phi,\Theta}(\lambda,u)$ is also
convex if we fix $u$ and consider $\lambda$ as variables.)

\begin{proposition}
  \label{prop:var_princ_for_cp_with_E}
  The function $\tilde\ell\in\R_{>0}^{E}$ defined in terms of
  $\lambda$ and $u$ by equations~\eqref{eq:tilde_ell}
  and~\eqref{eq:lambda} is a solution of
  Problem~\ref{prob:circle_pattern} if and only if $\tilde\ell$
  satisfies all triangle inequalities and
  $\mathcal{E}_{\T,\theta,\Theta}(\lambda,u)$ is the minimum of the
  function
  $\lambda\mapsto\mathcal{E}_{\T,\theta,\Theta}(\lambda,u)$. That is,
  $u$ is arbitrary but constant. (Without loss of generality one could
  fix $u=0$.)
\end{proposition}

\subsection{Discrete conformal equivalence for circular polyhedral
  surfaces}
\label{sec:circular_meshes}

In this section, we generalize the notion of discrete conformal
equivalence from surfaces composed of triangles to surfaces composed
of polygons inscribed in circles. The variational principle described
below is like a mixture of the first variational principle for
discrete conformal maps of Proposition~\ref{prop:first_variational}
and the variational principle for circle patterns of
Proposition~\ref{prop:var_princ_for_cp_with_E}.

An (abstract) \emph{polyhedral surface} is a surface that is a CW
complex. A \emph{euclidean polyhedral surface} is a polyhedral surface
obtained by gluing euclidean polygons edge-to-edge. If all of the
polygons have a circumscribed circle, we speak of a \emph{(euclidean)
  circular polyhedral surface}. A circular polyhedral surface is
determined by the polyhedral surface $\mathsf P$ and the function
$\ell\in(\R_{>0})^{E_{\mathsf P}}$ that assigns to each edge its
length. Conversely, a function $\ell\in(\R_{>0})^{E_{\mathsf P}}$
defines a circular polyhedral surface if and only if it satisfies the
``polygonal inequalities'': In each polygon, the length of any edge is
smaller then the sum of lengths of the other edges. If $\ell$
satisfies these conditions, we denote the resulting circular
polyhedral surface by $(\mathsf P, \ell)$.

\begin{definition}
  Two circular polyhedral surfaces, $(\mathsf P, \ell)$ and $(\mathsf
  P, \tilde\ell)$, are \emph{discretely conformally equivalent} if
  $\ell$ and $\tilde\ell$ are related by equation~\eqref{eq:tilde_ell}
  for some function $u\in\R^{V_{\mathsf P}}$.
\end{definition}

To solve the discrete conformal mapping problems for circular
polyhedral surfaces that are analogous to those described in
Section~\ref{sec:mapping_problems}, proceed as follows: First
triangulate the non-triangular faces of the given circular polyhedral
surface $(\mathsf P,\ell)$ to obtain a euclidean triangulation
$(\T,\hat\ell)$ (where $\hat\ell:E_{\T}\rightarrow\R_{>0}$,
$\hat\ell|_{E_{\mathsf P}}=\ell$). Then define $\Phi$ by
equations~\eqref{eq:theta_ij_for_dconf} and minimize
$\mathcal{E}_{\T,\Phi,\Theta}(\lambda,u)$, where
$\lambda_{ij}=2\log\hat\ell_{ij}$ is held fixed if $ij\in E_{\mathsf
  P}$ and considered a variable if $E_{\T}\setminus E_{\mathsf P}$,
and the $u_{i}$ are variables or fixed depending on the mapping
problem, as in the case of triangulations. If $\tilde\ell$ determined
by equations~\eqref{eq:tilde_ell} and~\eqref{eq:lambda} for the
minimizing $(\lambda, u)$ satisfies the triangle inequalities, it is a
solution of the mapping problem.

Note that the values $\hat\ell_{ij}$ for edges $ij\in E_{\T}\setminus
E_{\mathsf P}$ do not enter because the corresponding $\lambda_{ij}$
are variables.

\subsection{Discrete circle domains}
\label{sec:circle_domains}

A domain in the Riemann sphere $\widehat\C$ is called a \emph{circle
  domain} if every boundary component is either a point or a
circle. Koebe conjectured that every domain in $\C$ is conformally
equivalent to a circle domain. For a simply connected domain, this is
just the Riemann mapping theorem.  Koebe himself proved the conjecture
for finitely connected domains, and after various generalizations by
several other people, He and Schramm gave a proof for domains with at
most countably many boundary components~\cite{he_fixed_1993}. (Their
proof is based on circle packings.)

The method for mapping to the sphere described in
Section~\ref{sec:sphere} works (\emph{mutatis mutandis}) also for the
circular polyhedral surfaces discussed in the previous section. This
allows us to map euclidean triangulations to ``discrete circle
domains'', that is, domains in the plane that are bounded by circular
polygons.

Suppose $(\T,\ell)$ is a euclidean triangulation that is topologically
a disc with holes. To map $(\T,\ell)$ to a discrete circle domain,
simply fill the holes by attaching a face to each boundary polygon and
map the resulting circular polyhedral surface to the sphere. 

Note that for a topological disk with $0$ holes, we recover in a
different guise the procedure for mapping to a disk that was described
in Section~\ref{sec:disk}.

\end{appendices}

\section*{Acknowledgments}

We are grateful to Richard Kenyon and G\"unter Ziegler, who pointed
out the connection with amoebas, to Ulrich Bauer, Felix G{\"u}nther,
Hana Kou\v{r}imsk\'{a}, and Mathias Oster, who spotted typos and instances of
questionable style in earlier versions of this paper, and to the
referee whose comments helped us improve this paper still further. Any
remaining mistakes and shortcomings are our responsibility.

Stefan Sechelmann produced Figure~\ref{fig:pretzel} using software
written by himself in \emph{Java} together with the Java-based 3D
visualization package \emph{jReality}.

The other figures of discrete conformal maps were made with
\emph{Blender} and \emph{Python} scripts written by the third author,
which rely on other libraries to do the real work: the \emph{GNU
  Scientific Library}, providing an implementation of Clausen's
integral, the convex optimization library \emph{CVXOPT} by
M.~Andersen, J.~Dahl and L.~Vandenberghe, and J.~Shewchuk's mesh
generator \emph{Triangle}.

This research was supported by SFB/TR 109 ``Discretization in Geometry
and Dynamics'', the DFG Research Unit ``Polyhedral surfaces'' and by
the DFG Research Center \textsc{Matheon}.

\small
\bibliographystyle{gtart}
\bibliography{dconf}

\begin{thebibliography}{}
\providecommand\bibmarginpar{\leavevmode\marginpar}
\def\urlstyle#1{{\tt #1}}

\bibitem{alexandrov05:_convex_polyh}
\textbf{A\,D Alexandrov}, \emph{Convex polyhedra}, Springer Monographs in
  Mathematics, Springer-Verlag, Berlin (2005)Translated from the 1950 Russian
  edition

\bibitem{andreev_convex_1970-1}
\textbf{E\,M Andreev}, \href{http://dx.doi.org/10.1070/SM1970v010n03ABEH001677}
  {\emph{On convex polyhedra in {{L}oba\v{c}evski\u{i}} spaces}}, Math.
  {USSR-Sb.} 10 (1970) 413--440

\bibitem{andreev_convex_1970}
\textbf{E\,M Andreev}, \href{http://dx.doi.org/10.1070/SM1970v012n02ABEH000920}
  {\emph{On convex polyhedra of finite volume in {{L}oba\v{c}evski\u{i}}
  space}}, Math. {USSR-Sb.} 12 (1970) 255--259

\bibitem{bobenko06:_minim_surfac_from_circl_patter}
\textbf{A\,I Bobenko}, \textbf{T Hoffmann}, \textbf{B\,A Springborn},
  \href{http://dx.doi.org/10.4007/annals.2006.164.231} {\emph{Minimal surfaces
  from circle patterns: geometry from combinatorics}}, Ann. of Math. (2) 164
  (2006) 231--264

\bibitem{bobenko_alexandrovs_2008}
\textbf{A\,I Bobenko}, \textbf{I Izmestiev},
  \href{http://aif.cedram.org/item?id=AIF_2008__58_2_447_0} {\emph{Alexandrov's
  theorem, weighted {{D}elaunay} triangulations, and mixed volumes}}, Ann.
  Inst. Fourier {(Grenoble)} 58 (2008) 447{\textendash}505

\bibitem{bobenko04:_variat_princ_for_circl_patter}
\textbf{A\,I Bobenko}, \textbf{B\,A Springborn},
  \href{http://dx.doi.org/10.1090/S0002-9947-03-03239-2} {\emph{Variational
  principles for circle patterns and {K}oebe's theorem}}, Trans. Amer. Math.
  Soc. 356 (2004) 659--689

\bibitem{chow_combinatorial_2003}
\textbf{B Chow}, \textbf{F Luo},
  \href{http://projecteuclid.org/getRecord?id=euclid.jdg/1080835659}
  {\emph{Combinatorial {{R}icci} flows on surfaces}}, J. Differential Geom. 63
  (2003) 97{\textendash}129

\bibitem{colin_de_verdire_un_1991}
\textbf{Y {{Colin} de Verdi\`{e}re}},
  \href{http://dx.doi.org/10.1007/BF01245096} {\emph{Un principe variationnel
  pour les empilements de cercles}}, Invent. Math. 104 (1991)
  655{\textendash}669

\bibitem{dai_variational_2008}
\textbf{J Dai}, \textbf{X\,D Gu}, \textbf{F Luo},
  \href{http://www.ams.org/mathscinet-getitem?mr=2439807} {\emph{Variational
  Principles for Discrete Surfaces}}, volume~4 of \emph{Advanced Lectures in
  Mathematics {(ALM)}}, International Press, Somerville, {MA} (2008)

\bibitem{duffin_distributed_1959}
\textbf{R\,J Duffin}, \emph{Distributed and lumped networks}, J. Math. Mech. 8
  (1959) 793{\textendash}826

\bibitem{epstein_euclidean_1988}
\textbf{D\,B\,A Epstein}, \textbf{R\,C Penner},
  \href{http://projecteuclid.org/getRecord?id=euclid.jdg/1214441650}
  {\emph{Euclidean decompositions of noncompact hyperbolic manifolds}}, J.
  Differential Geom. 27 (1988) 67{\textendash}80

\bibitem{fock_dual_1997}
\textbf{V\,V Fock}, \href{http://arxiv.org/abs/dg-ga/9702018} {\emph{Dual
  {{T}eichm\"{u}ller} spaces}}, {arXiv:dg-ga/9702018}  (1997)

\bibitem{futer11:_from}
\textbf{D Futer}, \textbf{F Gu{\'e}ritaud},
  \href{http://dx.doi.org/10.1090/conm/541/10683} {\emph{From angled
  triangulations to hyperbolic structures}}, from: ``Interactions between
  hyperbolic geometry, quantum topology and number theory'', Contemp. Math.
  541, Amer. Math. Soc., Providence, RI (2011)  159--182

\bibitem{von_gagern_hyperbolization_2009}
\textbf{M von Gagern}, \textbf{J {Richter-Gebert}},
  \href{http://www.combinatorics.org/Volume_16/Abstracts/v16i2r12.html}
  {\emph{Hyperbolization of {{E}uclidean} ornaments}}, Electron. J. Combin. 16
  (2009) Research Paper 12, 29 pages

\bibitem{gelfand_discriminants_1994}
\textbf{I\,M Gelfand}, \textbf{M\,M Kapranov}, \textbf{A\,V Zelevinsky},
  \href{http://dx.doi.org/10.1007/978-0-8176-4771-1} {\emph{Discriminants,
  Resultants, and Multidimensional Determinants}}, Mathematics: Theory \&
  Applications, Birkh\"{a}user Boston Inc., Boston, {MA} (1994)

\bibitem{gu_computational_2008}
\textbf{X\,D Gu}, \textbf{S Yau},
  \href{http://www.ams.org/mathscinet-getitem?mr=2439718} {\emph{Computational
  Conformal Geometry}}, volume~3 of \emph{Advanced Lectures in Mathematics
  {(ALM)}}, International Press, Somerville, {MA} (2008)

\bibitem{he_fixed_1993}
\textbf{Z He}, \textbf{O Schramm}, \href{http://www.jstor.org/stable/2946541}
  {\emph{Fixed Points, {{K}oebe} Uniformization and Circle Packings}}, Ann. of
  Math. (2) 137 (1993) 369--406

\bibitem{kenyon09:_lectur}
\textbf{R Kenyon}, \emph{Lectures on dimers}, from: ``Statistical mechanics'',
  IAS/Park City Math. Ser. 16, Amer. Math. Soc., Providence, RI (2009)
  191--230

\bibitem{kenyon_planar_2006}
\textbf{R Kenyon}, \textbf{A Okounkov},
  \href{http://dx.doi.org/10.1215/S0012-7094-06-13134-4} {\emph{Planar dimers
  and {{H}arnack} curves}}, Duke Math. J. 131 (2006) 499{\textendash}524

\bibitem{kenyon_dimers_2006}
\textbf{R Kenyon}, \textbf{A Okounkov}, \textbf{S Sheffield},
  \href{http://dx.doi.org/10.4007/annals.2006.163.1019} {\emph{Dimers and
  amoebae}}, Ann. of Math. (2) 163 (2006) 1019{\textendash}1056

\bibitem{koebe_kontaktprobleme_1936}
\textbf{P Koebe}, \emph{Kontaktprobleme der konformen {{Abbildung}}}, Ber.
  S\"{a}chs. Akad. Wiss. Leipzig, Math.-phys. Kl. 88 (1936) 141--164

\bibitem{leibon_characterizingdelaunay_2002}
\textbf{G Leibon}, \href{http://dx.doi.org/10.2140/gt.2002.6.361}
  {\emph{Characterizing the {{D}elaunay} decompositions of compact hyperbolic
  surfaces}}, Geom. Topol. 6 (2002) 361--391

\bibitem{lewin_polylogarithms_1981}
\textbf{L Lewin}, \emph{Polylogarithms and Associated Functions},
  {North-Holland} Publishing Co., New York (1981)

\bibitem{luo_combinatorial_2004}
\textbf{F Luo}, \emph{Combinatorial {{Y}amabe} flow on surfaces}, Commun.
  Contemp. Math. 6 (2004) 765{\textendash}780

\bibitem{luo_rigidity_2010}
\textbf{F Luo}, \href{http://www.ams.org/mathscinet-getitem?mr=2744222}
  {\emph{Rigidity of polyhedral surfaces}}, from: ``Fourth International
  Congress of Chinese Mathematicians'', {AMS/IP} Stud. Adv. Math. 48, Amer.
  Math. Soc., Providence, {RI} (2010)  201--217

\bibitem{mikhalkin_amoebas_2004}
\textbf{G Mikhalkin}, \href{http://dx.doi.org/10.1007/0-306-48658-X_6}
  {\emph{Amoebas of algebraic varieties and tropical geometry}}, from:
  ``Different Faces of Geometry'', Int. Math. Ser. {(N.} Y.) 3,
  {Kluwer/Plenum,} New York (2004)  257{\textendash}300

\bibitem{milnor_hyperbolic_1982}
\textbf{J Milnor}, \href{http://dx.doi.org/10.1090/S0273-0979-1982-14958-8}
  {\emph{Hyperbolic geometry: the first 150 years}}, Bull. Amer. Math. Soc.
  {(N.S.)} 6 (1982) 9{\textendash}24

\bibitem{milnor_to_1994}
\textbf{J Milnor}, \emph{How to compute volume in hyperbolic space}, from:
  ``Collected Papers'', volume~1, Publish or Perish Inc., Houston, {TX} (1994)
  189--212

\bibitem{milnor_schlfli_1994}
\textbf{J Milnor}, \emph{The {{S}chl\"{a}fli} differential equality}, from:
  ``Collected Papers'', volume~1, Publish or Perish Inc., Houston, {TX} (1994)
  281--295

\bibitem{papadopoulos_handbook_2007}
\textbf{A Papadopoulos} (editor), \href{http://dx.doi.org/10.4171/029}
  {\emph{Handbook of {{T}eichm\"{u}ller} theory. Vol. I}}, volume~11 of
  \emph{{IRMA} Lectures in Mathematics and Theoretical Physics}, European
  Mathematical Society {(EMS),} Z\"{u}rich (2007)

\bibitem{passare_amoebas_2004}
\textbf{M Passare}, \textbf{H Rullg\r{a}rd},
  \href{http://dx.doi.org/10.1215/S0012-7094-04-12134-7} {\emph{Amoebas,
  {Monge-Amp\`{e}re} measures, and triangulations of the {{N}ewton} polytope}},
  Duke Math. J. 121 (2004) 481{\textendash}507

\bibitem{penner_decorated_1987}
\textbf{R\,C Penner}, \emph{The decorated {{T}eichm\"{u}ller} space of
  punctured surfaces}, Comm. Math. Phys. 113 (1987) 299{\textendash}339

\bibitem{pinkall_computing_1993}
\textbf{U Pinkall}, \textbf{K Polthier},
  \href{http://projecteuclid.org/getRecord?id=euclid.em/1062620735}
  {\emph{Computing discrete minimal surfaces and their conjugates}},
  Experiment. Math. 2 (1993) 15{\textendash}36

\bibitem{rivin_euclidean_1994}
\textbf{I Rivin}, \href{http://www.jstor.org/stable/2118572} {\emph{Euclidean
  Structures on Simplicial Surfaces and Hyperbolic Volume}}, Ann. of Math. (2)
  139 (1994) 553--580

\bibitem{rivin_intrinsic_1994}
\textbf{I Rivin}, \emph{Intrinsic geometry of convex ideal polyhedra in
  hyperbolic 3-space}, from: ``Analysis, Algebra, and Computers in Mathematical
  Research {(Lule\r{a},} 1992)'', Lecture Notes in Pure and Appl. Math. 156,
  Dekker, New York (1994)  275{\textendash}291

\bibitem{rivin_combinatorial_2003}
\textbf{I Rivin}, \href{http://dx.doi.org/10.1016/S0196-8858(03)00093-9}
  {\emph{Combinatorial optimization in geometry}}, Adv. in Appl. Math. 31
  (2003) 242--271

\bibitem{roek_quantization_1984}
\textbf{M Ro{\v{c}}ek}, \textbf{R\,M Williams},
  \href{http://dx.doi.org/10.1007/BF01581603} {\emph{The quantization of
  {{R}egge} calculus}}, Z. Phys. C 21 (1984) 371--381

\bibitem{rodin_convergence_1987}
\textbf{B Rodin}, \textbf{D Sullivan},
  \href{http://projecteuclid.org/getRecord?id=euclid.jdg/1214441375} {\emph{The
  convergence of circle packings to the {{R}iemann} mapping}}, J. Differential
  Geom. 26 (1987) 349{\textendash}360

\bibitem{roeder_andreevs_2007}
\textbf{R\,K\,W Roeder}, \textbf{J\,H Hubbard}, \textbf{W\,D Dunbar},
  \href{http://www.ams.org/mathscinet-getitem?mr=2336832} {\emph{{A}ndreev's
  theorem on hyperbolic polyhedra}}, Ann. Inst. Fourier (Grenoble) 57 (2007)
  825--882

\bibitem{springborn_conformal_2008}
\textbf{B Springborn}, \textbf{P Schr\"{o}der}, \textbf{U Pinkall},
  \href{http://dx.doi.org/10.1145/1360612.1360676} {\emph{Conformal equivalence
  of triangle meshes}}, {ACM} Trans. Graph. 27 (2008) article 77, 11 pages

\bibitem{stephenson_introduction_2005}
\textbf{K Stephenson}, \emph{Introduction to Circle Packing. The Theory of
  Discrete Analytic Functions}, Cambridge University Press, Cambridge (2005)

\bibitem{thurston_geometry_????}
\textbf{W\,P Thurston}, \emph{The geometry and topology of three--manifolds},
  Princeton lecture notes (1978--1981).
\ Available at \setbox0\hbox{\makeatletter\@url
{http://www.msri.org/communications/books/gt3m}}
\href{http://www.msri.org/communications/books/gt3m}
{\unhbox0}

\bibitem{thurston_minimal_1998}
\textbf{W\,P Thurston}, \href{http://arxiv.org/abs/math/9801039} {\emph{Minimal
  stretch maps between hyperbolic surfaces}}, \nolinkurl{arXiv:math/9801039}
  (1998)

\bibitem{troyanov86:_les}
\textbf{M Troyanov}, \emph{Les surfaces euclidiennes \`a singularit\'es
  coniques}, Enseign. Math. (2) 32 (1986) 79--94

\end{thebibliography}

\end{document}